\documentclass[reqno,11pt]{amsart}
\usepackage[applemac]{inputenc}

\usepackage[T1]{fontenc}

\usepackage{ulem}

\usepackage{pifont}

\usepackage{amsmath,esint,color}
\usepackage{amsthm,latexsym,epsfig,graphicx,mathrsfs}
\usepackage{amsmath}
\usepackage{amsfonts}
\usepackage{amssymb}
\usepackage{fullpage}
\numberwithin{equation}{section}
\usepackage[colorlinks=true, pdfstartview=FitV, linkcolor=blue, citecolor=blue, urlcolor=blue]{hyperref}

\def\N{{\mathbb N}}
\date{}
\definecolor{sah}{rgb}{0.66,0.33, 0.04}
\definecolor{adel4}{cmyk}{1,0,0,0}
\definecolor{adel3}{rgb}{0.66,0.33, 0.04}
\definecolor{adel1}{cmyk}{0,0.20,1,0}
\definecolor{adel2}{cmyk}{0,0.40,1,0.30}
\definecolor{adel0}{rgb}{0.99,0.60, 0.30}
\definecolor{trut}{rgb}{0.99,0.80, 0.00}
\definecolor{trus}{rgb}{0.00, 0.50, 0.00}
 \definecolor{trust}{rgb}{0.99, 0.99, 0.80}
\definecolor{MaCouleur}{rgb}{0,0.9,0.3}

\newcommand{\NN}{\mathbb{N}}

\newcommand{\RR}{\mathbb{R}}

\newcommand{\pxl}{\partial_{X_\lambda}}

\newcommand{\pxt}{\partial_{X_{t}}}
\newcommand{\pxz}{\partial_{X_{0}}}

\newcommand{\Div}{\textnormal{div }}
\newcommand{\Card}{\textnormal{card}}
\newcommand{\supp}{\textnormal{supp}}
\newcommand{\EE}{\varepsilon}

\newcommand{\Sg}{\textnormal{sign}}

\theoremstyle{plain}
\newtheorem{definition}{Definition}[section]
\newtheorem{theorem}{Theorem}[section]
\newtheorem{proposition}{Proposition}[section]
\newtheorem{lemma}{Lemma}[section]
\newtheorem{remark}{Remark}[section]
\newtheorem{coro}{Corollary}[section]

\def\virgp{\raise 2pt\hbox{,}}

\def\Xint#1{\mathchoice
 {\XXint\displaystyle\textstyle{#1}}%
 {\XXint\textstyle\scriptstyle{#1}}%
 {\XXint\scriptstyle\scriptscriptstyle{#1}}%
 {\XXint\scriptscriptstyle\scriptscriptstyle{#1}}%
 \!\int}
 \usepackage{todonotes}
\def\XXint#1#2#3{{\setbox0=\hbox{$#1{#2#3}{\int}$}
 \vcenter{\hbox{$#2#3$}}\kern-.5\wd0}}

\def\av_#1{\Xint-_{#1}}

\usepackage[textmath,displaymath,floats,graphics,delayed]{preview}
 \title{Rigidity aspects of  singular patches in  stratified flows }
\author[T. Hmidi]{Taoufik Hmidi}
\address[T. Hmidi]{NYUAD Research Institute, New York University in Abu Dhabi, Saadiyat Island\\
P.O. Box 129188, Abu Dhabi, United Arab Emirates }
\address{Univ Rennes, CNRS, IRMAR--UMR 6625, F-35000 Rennes, France}
\email{thmidi@univ-rennes1.fr.}
\author[H. Houamed]{Haroune Houamed}
\address[H. Houamed]{Department of Mathematics, New York University in Abu Dhabi, Saadiyat Island\\
P.O. Box 129188, Abu Dhabi, United Arab Emirates }
\email{haroune.houamed@nyu.edu}

\author[M. Zerguine]{Mohamed Zerguine}
\address[M. Zerguine]{LEDPA, Universit\'e Batna --2--\\ Facult\'e des Math\'ematiques et d'Informatique\\ D\'epartement de Math\'ematiques\\ 05000 Batna Alg\'erie}
\email{m.zerguine@univ-batna2.dz}

\begin{document}

\begin{abstract}
We explore the local well-posedness theory for the 2d inviscid Boussinesq system when the vorticity is given by a singular patch. We give a significant improvement of \cite{Hassainia-Hmidi} by replacing their compatibility assumption on the density with a constraint on its platitude degree on the singular set. The second main contribution focuses on the same issue for the partial viscous Boussinesq system. We establish a uniform LWP theory with respect to the vanishing conductivity. This issue is much more delicate than the inviscid case and one should carefully deal with various difficulties related to the diffusion effects which tend to alter some local structures.  The  weak a priori estimates are not  trivial  and  refined analysis on  transport-diffusion equation subject to a logarithmic singular potential is required. Another difficulty stems from some commutators arising in the control of the co-normal regularity that we counterbalance in part by the maximal smoothing effects of transport-diffusion equation advected by a velocity field which scales slightly below the Lipschitz class.
%
\end{abstract}
\maketitle
\tableofcontents

\section{Introduction}
 The  atmosphere motion and   the ocean  circulation  is an area of intensive activity offering challenging problems to explore. The governing equations stem from the thermodynamics and mechanics laws, and can be expressed in general as a system of nonlinear partial differential equations with higher degree of complexity. To capture the prominent feature of large scale motion, some approximations around equilibrium states are used in order to simplify slightly the equations. Notice that the aforementioned  equilibrium states are generally obtained from the hydrostatic or the geostrophic balances driven by the stratification, the gravitation and earth rotation mechanisms. There are several models used in  the literature and in what follows we shall be concerned with the Boussinesq system. In this case, the fluid motion is mainly influenced by the velocity field and the density with the assumption that in the momentum equation the density effect is only sensitive for the gravitational  force, see for instance  \cite{Pedlosky}. The outcome of this analysis in the plane is the following full viscous Boussinesq system 
\begin{equation} \label{B-mu-kappa}
\left\{ \begin{array}{ll}
  \partial_t v + v \cdot \nabla v-\mu\Delta v+\nabla p = \rho\begin{pmatrix}0\\
  1 \end{pmatrix},\quad (t,x)\in (0,\infty)\times \mathbb{R}^2,&\\ 
   \partial_t\rho+v\cdot\nabla \rho-\kappa\Delta\rho=0,&\vspace{1mm}\\
     \Div v=0,&\vspace{1mm}\\
  (v,\rho)|_{t=0}=(v_0,\rho_0),\tag{B$_{\mu,\kappa}$}
  \end{array}\right.
\end{equation}
where $v$ is the velocity field in the plane, $\rho$ is a scalar function that  stands for the fluid density, and $p$ is the pressure which is a scalar potential that can be connected to the velocity and the density through the incompressibility assumption  $\Div v=0.$ The parameters $\mu,\kappa$ are positive real numbers  called the viscosity and the conductivity, respectively.
For fluids with constant density,  the system \eqref{B-mu-kappa} boils down to the classical Navier-Stokes equations, \begin{equation}\label{NS-mu}
\left\{ \begin{array}{ll}
\partial_{t}v+v\cdot\nabla v-\mu\Delta v+\nabla p=0, &\vspace{1mm}\\
\Div v=0, &\vspace{1mm}\\
v_{| t=0}={v}_0,
\end{array} \right.\tag{NS$_\mu$}
\end{equation}
which in turn reduces to the incompressible Euler equation (E) in the absence of the viscosity.
It is convenient in the planar case  to take advantage of the vorticity structure and  write down the equations in terms of vorticity and density. Actually, the vorticity associated to the velocity field is determined by the scalar function
\begin{equation*}
\omega\triangleq\partial_1 v^2-\partial_2 v^1=\nabla^{\perp}\cdot v,\quad \nabla^{\perp}=(-\partial_2,\partial_1).
\end{equation*}
Then the system \eqref{B-mu-kappa} can be transformed into the equivalent reformulation
\begin{equation}\label{VD-mu-kappa}
\left\{ \begin{array}{ll}
  \partial_t \omega + v \cdot \nabla\omega-\mu\Delta \omega  = \partial_1\rho,&\vspace{1mm}\\ 
   \partial_t\rho+v\cdot\nabla \rho-\kappa\Delta\rho=0,&\vspace{1mm}\\
   v=\Delta^{-1}\nabla^\perp\omega,&\vspace{1mm}\\
  (\omega,\rho)|_{t=0}=(\omega_0,\rho_0).\tag{VD$_{\mu,\kappa}$}
  \end{array}\right.
\end{equation}
%
%
Now, the main task  is to overview the status of the art with respect to the well-posedness theory for these models.  There are a lot of important results achieved during  the past  decades and we shall restrict the discussion to some of them.\\
\ding{202} {\it Global well-posedness: from Euler to N-S equations}. The construction of global unique solutions in the H\"older class $C^{k,\alpha}$ with $k\ge1$ and $\alpha\in(0,1)$ was performed many years ago in the prominent works of H\"older \cite{Holder} and W\"olibner \cite{Wolibner}. Later, similar results were obtained by Kato \cite{Kato} in Sobolev spaces $H^s(\RR^2)$, with $s>2.$ The main ingredients for getting global solutions are the conservation of the $L^\infty$ norm of the vorticity which is simply advected by the flow and a logarithmic type estimate. For more results and expository texts we refer to \cite{Bourgain-Li-1,Bourgain-Li-2,Chae-1,Chemin-1,Elgindi-Masmoudi,Kato-Ponce,Majda-Bertozzi,Pak-Park,Vishik}. By taking advantage of the vorticity advection, Yudovich succeeded in \cite{Yudovich} to relax the classical theory and generate global unique solutions when the initial vorticity is bounded and integrable. In this setting the velocity field is not necessary Lipschitz but belongs to the class of log-Lipschitz functions. As a by-product, we obtain the global persistence and uniqueness with  the vortex patch structure. More precisely, if the initial vorticity $\omega_0={\bf 1}_{\Omega_0}$ is a patch, that is, the characteristic function of a bounded domain then the transported vorticity being also a patch $\omega(t) = {\bf 1}_{\Omega_t}$ such that $\Omega_t\triangleq\Psi(t,\Omega_0)$, where $\Psi$ is the flow associated to the velocity field. In this specific case  we reduce the dynamics to the boundary motion in the plane, which is subject to an integro-differential equation. The global in time persistence of the boundary regularity was accomplished in a remarkable work of Chemin \cite{Chemin} for a large class of data covering in particular simply connected patches with $C^{1+\EE}-$ boundary. The proof was slightly simplified for the patches by Bertozzi and Constantin in \cite{Bertozzi-Canstantin}. We point out that Chemin's toolbox is robust and flexible and it allows him to tackle through suitable adaptations the case of singular patches. He obtained under suitable assumptions the regularity persistence of the smooth part of the boundary. For more discussion and extension on this subject we may refer to the papers \cite{Bae-Kelliher,Bertozzi-Canstantin,Berto-Verdera,Danchin-0,Danchin-1,Depauw,Dutrifoy,Gamblin-Raymond,Serfati,Sueur-1, Xian-Zhang, LJSFZ, Paicu-Zhang-123} and the references therein. 

\hspace{0.5cm}As regards to the Navier-Stokes system \eqref{NS-mu}, the local/global well-posedness theory in any dimension space is one of the most studied subject in fluid dynamics and we may find a complete review on the recent developments on this topic in the textbooks \cite{Lemarie-Rieusset,Robinson-Rodrigo-Sadowski}. Here, we shall only go over some results that fit with our main scope related to the smooth and singular patch problem. The first result was achieved by Danchin \cite{Danchin-2}, who asserts that if $\omega_{0}={\bf1}_{\Omega_0}$ with  $\partial\Omega_0$ is a Jordan curve with $C^{1+\varepsilon}$ regularity then the velocity field of \eqref{NS-mu}  remains  Lipschitz  for any time uniformly with respect to the viscosity. In addition, the image of the initial boundary by the viscous flow is  $C^{1+\varepsilon^\prime}$ for all $\EE^\prime<\EE$. This regularity  loss is attributed to the use of Besov spaces $B^{s}_{p,\infty}$ with $p\in[1,\infty)$ where some maximal  smoothing effects of transport-diffusion equations are established using energy estimates. Unfortunately, with this techniques,  H\"{o}lder spaces seem to be out of reach due to the amplification of the estimates when $p$ grows to $\infty$. To remedy to this defect and meanwhile show the regularity persistence, the first author  extended in  \cite{Hmidi} the smoothing effects to this borderline case performing a new method based on Lagrangian coordinates and a frequency double localization lemma established by Vishik in \cite{Vishik}. Notice that in both papers, and compared to the inviscid case, the regularity persistence is a delicate issue where one should carefully deal with two opposite facets of the viscosity. The worst one manifests when we want to track  the co-normal regularity of the vorticity and in this case new commutators between suitable vector-fields and  the Laplacian operator emerge, generating a loss of regularity.  To counterbalance this side effect and repair the damage caused by the Laplacian, we  use its  maximal smoothing effects. As to the singular viscous patches, the literature is slightly poor and we refer to  the result of the first author  \cite{Hmidi}.  It was shown that starting with  $\omega_{0}={\bf1}_{\Omega_0}$ where the boundary $\partial\Omega_0$ is assumed to be  $C^{1+\varepsilon}$ outside a  singular set $\Sigma_0$,  we find that  the velocity field  is Lipschitz away the transported singular set by the flow, uniformly in the vanishing viscosity. It was also proved that the image  by the viscous flow of the smooth part is in  $C^{1+\varepsilon^\prime}$ for any $\EE^\prime<\EE$. It is worthy to note  that the velocity scales  slightly below the  Lipschitz class leading to serious   difficult technical problems. For instance, we cannot reach the maximal smoothing effect in transport-diffusion equation and logarithmic loss of regularity sounds to be unavoidable. This is the main technical factor behind the loss of regularity of the boundary and we will come back later with more explanations when we shall comment our main results and explore the proofs. 

\ding{203} {\it Well-posedness for partial viscous Boussinesq system.} There is a lot of activities around this subject depending on whether the viscosities are vanishing or not. In the full viscous case $\mu, \kappa>0$, global well-posed results were established in \cite{Cannon-Dibenedetto,Guo} using the energy estimates method. For the partial viscous case  $\mu>0, \kappa=0,$ global well-posedness was analyzed by Hou and Li \cite{Hou-Li} and Chae \cite{Chae} in the framework of Sobolev spaces $H^s, s>2$. Later, Abidi and Hmidi \cite{Abidi-Hmidi} succeeded to improve the previous results for $(v_0, \rho_0)\in B^{-1}_{\infty,1}\cap L^2\times B^{0}_{2,1}$. The case $v^0, \rho_0\in L^2$ was solved by Danchin and Paicu in \cite{Danchin-Paicu-2}. In the partial viscous case $\mu=0, \kappa>0$, similar studies have been carried out through various papers and we shall only cite some of them. In \cite{Chae}, Chae proved the global well-posedness for $v_0, \rho_0\in H^s$ for $s>2$ which was improved by Hmidi and Keraani in \cite{Hmidi-Keraani-1} for the critical Besov spaces, that is $(v_0,\rho_0)\in B^{1+2/p}_{p,1}\times B^{-1+{2}/{p}}_{p,1}\cap L^r,\;r>2$. In the Yudovich setting, Danchin and Paicu \cite{Danchin-Paicu-1} obtained the global existence and uniqueness when $(v^0, \rho^0)\in L^2\times L^2\cap B_{\infty, 1}^{-1}$ and $\omega^0\in L^r\cap L^\infty$ with $r\geqslant  2$. We emphasize that a lot of important developments regarding closely connected models have been conducted over the past few decades. For more details, we refer to the papers \cite{Berestycki-Constantin-Ryzhik,Chae-Kim-Nam-1,Chae-Kim-Nam-2,Feireisl-Novotny,Hmidi-Zerguine,Larious-Lusin-Titi,Liu-Wang-Zhang,Miao-Zheng,Taniuchi,Tao-Zhang,Wu-Xue} and the references therein.

\ding{204} {\it Inviscid Boussinesq system}. It corresponds to the system  \eqref{B-mu-kappa} with   $\mu=\kappa=0$, and it  will simply be denoted by  \eqref{B0}. Its vorticity-density reformulation reads as follows
\begin{equation} \label{B0}
\left\{ \begin{array}{ll}
  \partial_t \omega + v \cdot \nabla \omega = \partial_1\rho,&\vspace{1mm}\\ 
   \partial_t\rho+v\cdot\nabla \rho=0,\,
v=\Delta^{-1}\nabla^\perp\omega\vspace{1mm}\\
  (\omega,\rho)|_{t=0}=(\omega_0,\rho_0).\tag{B$_0$}
  \end{array}\right.
\end{equation}
Local well-posedness theory was established in the setting of sub-critical Sobolev  spaces  $v_0,\rho_0\in H^{s}, s>2$ or H\"older spaces $v_0,\rho_0\in C^{k+\epsilon}, k\in\mathbb{N}^\star, \varepsilon\in(0,1)$ by Chae and Nam in \cite{Chae-Kim-Nam-1,Chae-Kim-Nam-2}. Furthermore, they formulated a blow-up criterion related only to  the density in the following sense:   if $T^\star$ is the maximal lifespan of the solution of \eqref{B0} then
\begin{equation*}
 T^{\star}<\infty\Longrightarrow\int_{0}^{T^\star}\|\nabla\rho(\tau)\|_{L^\infty}d\tau=\infty.
\end{equation*}
This criterion remains valid for \eqref{B-mu-kappa} for different values of $\mu, \kappa$. The blow-up of classics solutions  has been addressed very recently  in two relevant papers \cite{Elgindi-1,Chen}. In the first one, Elgindi proved  that the blow-up  in finite time occurs for smooth solutions  when the fluid is evolving in a singular domain, typically an  infinite sector. However, in \cite{Chen} Chen and Hou  were able to show the finite time singularity formation  in a half-plane domain  for slightly rough velocity belonging to the H\"older class $C^{1+\epsilon}$, with $\epsilon$ sufficiently small.\\
\hspace{0.5cm} Concerning the construction of Yudovich type solutions, the situation turns out  to be more intricate and highly critical compared to Euler equations.   Actually, from the first equation of \eqref{B0} we infer 
$$
\|\omega(t)\|_{L^\infty}\leqslant \|\omega_0\|_{L^\infty}+\int_0^t\|\nabla \rho(\tau)\|_{L^\infty} d\tau.
$$ 
 Then, differentiating the second equation of \eqref{B0}  \begin{equation*}
(\partial_{t}+v\cdot\nabla)\partial_j\rho =\partial_j v\cdot\nabla\rho.
\end{equation*}
Consequently, the estimate of $\|\nabla\rho(t)\|_{L^\infty} $ requires the source term to be bounded, which is not the case in general because the velocity field is not Lipschitz in the framework of Yudovich solutions. This suggests that all the scenarios of well/ill-posedness are likely possible for this model, a fact that was confirmed by some recent works. Indeed, on one hand, according to Elgindi and Masmoudi contribution \cite{Elgindi-Masmoudi} an ill-posed result connected to some instability estimates in the Yudovich setting can be achieved with special initial data. On the second hand, a well-posedness result in the smooth/singular patch class was obtained by Hassainia and Hmidi \cite{Hassainia-Hmidi}. Although the patch structure is destroyed during the dynamics, one can still take advantage of the flexible vortex patch formalism developed by Chemin \cite{Chemin} and prove that in the smooth case, the velocity remains Lipschitz for a short time. However, when the patch is singular, a compatibility condition was required in \cite{Hassainia-Hmidi} stating that the density should be constant around the singular set of the boundary. This property is conserved along  the dynamics due to the transport equation, and this was a key observation in their proof. It was left open the question of whether local well-posedness result can be performed for the partial viscous case (B$_{0,\kappa}$) uniformly on the vanishing conductivity when the initial patch is singular. The presence of the dissipation induces a smoothing effect which unfortunately alters the compatibility condition affecting seriously the main ingredients in their arguments. In the smooth case, it turns out in view of the recent work of Meddour \cite{Meddour}, that a positive answer for the local well-posedness problem can be obtained using the same lines of the proof. More results and extensions to different models can be found in \cite{Danchin-Zhang,Danchin-Fanelli,Fanelli,Hmidi-Zerguine-1,Paicu-Zhu,Zerguine}. \\

\quad \ding{70}{{\bf Aims and main results.}}
The main task of the current work is to see whether we can relax the compatibility assumption needed  in  \cite{Hassainia-Hmidi} and meanwhile obtain a uniform result for the partial viscous system (B$_{0,\kappa}$). This system will be  denoted in short \eqref{B-0-kappa} and takes in the vorticity-density formulation the form 
\begin{equation}\label{B-0-kappa}
\left\{ \begin{array}{ll}
  \partial_t \omega + v \cdot \nabla\omega = \partial_1\rho,&\vspace{1mm}\\ 
   \partial_t\rho+v\cdot\nabla \rho-\kappa\Delta\rho=0,\vspace{1mm}\\
   v=\Delta^{-1}\nabla^\perp\omega,&\vspace{1mm}\\
  (\omega,\rho)|_{t=0}=(\omega_0,\rho_0).\tag{B$_{\kappa}$}
  \end{array}\right.
\end{equation}
As we have mentioned before, singular patches framework offers a delicate subject to analyze with regard to the local well-posedness problem for the system \eqref{B0}. In \cite{Hassainia-Hmidi}, this was achieved only when the initial density is constant around  the singular part of the domain. By doing that we suppress the singularity effects during the motion and prevent them to interact with the density  to avoid any possibility for the norm inflation. It is likely possible to conceive a well-prepared configuration where the singularity of the patch could breakdown the local well posedness through some sort of  instabilities. This is a relevant and important  issue which sounds to be out of reach of the techniques used in the vortex patch problem. Here we fix another connected target aiming to cover more class of initial data on the density, still with singular patches, leading to a well-posedness result.\\

Our first contribution concerns the inviscid system \eqref{B0} and we shall extend the result of \cite{Hassainia-Hmidi}. More precisely, we will prove the following theorem whose general statement is detailed in Theorem \ref{Th1:general:version}.   
\begin{theorem}\label{THEO:1:soft}
Let $s\in (0,1)$ and $\Omega_0$ be a bounded domain of the plane whose boundary $\partial\Omega_0$ is a Jordan curve of class $C^{1+s}$ outside a closed countable set $\Sigma_0\subset \partial \Omega_0$.
 Let   $\omega_0=1_{\Omega_0}$ and  \mbox{$  \rho_0\in  C^{1+s}\cap W^{1,a}$}, for some  $a\in (1,\infty)$. 
Suppose that $\nabla \rho_0$ is identically zero on $\Sigma_0$. 
 Then, there exists $T>0$ such that the  system  \eqref{B0} admits a unique local solution $(\omega,\rho)$ satisfying 
 $$\omega,\rho,\nabla \rho \in L^\infty\big([0,T], L^a\cap L^\infty ) .
 $$  
Furthermore, the velocity field $v$ satisfies
$$
\sup_{h\in(0,e^{-1}]}\frac{\Vert \nabla v(t)\Vert_{L^\infty((\Sigma_t)_h^c)}}{-\log h}\in L^\infty([0,T]),
$$
where $\Sigma_t\triangleq \Psi(t,\Sigma_0)$ and
$$
(\Sigma_t)_h^c\triangleq \big\{x\in \RR^2;\,d\big(x,\Sigma_t\big)\geqslant  h\big\}.
$$ 
In addition, the boundary of $\Psi(t,\Omega_0)$ is locally $C^{1+s}$ outside the set $\Sigma_t$.  
\end{theorem}
We first  remark that our theorem improves considerably the result of \cite{Hassainia-Hmidi} where it was assumed that $\nabla \rho_0$ vanishes in a small neighborhood of the singular set $\Sigma_0$. In our result, we need just this function to vanish exactly at the singular set. Let us now turn to the proof, sketch the main ideas and provide some details on the major difficulties that we have encountered and how to solve them. Below we shall explore the different steps implemented during the proof.\\
\foreach \x in {\bf a} {%
  \textcircled{\x}
} {\it $L^p$-estimates.} In this step, we establish weak a priori estimates for $\omega$ and $\nabla\rho$ in $L^p$ with $p\in[a,\infty]$. The full statement is described in Proposition \ref{Proposition inviscid} where the regularity and the vanishing condition imposed to $\nabla\rho$ are simply replaced by a slightly weaker assumption connected to the platitude degree of $\nabla \rho_0$ near the singular set as stated in \eqref{hypothesis1}. From this proposition, we succeeded to follow the time evolution of the platitude degree which is decaying and remains strictly positive during a short time. This fact is crucial to derive the required a priori estimates at the $L^p$ level following an energy method. Let us give a quick insight into how one could follow the platitude degree of $\nabla\rho(t)$. For this aim we simply proceed intuitively with a weighted transport localization around the singular set $\Sigma_t\triangleq \Psi(t,\Sigma_0)$ combined with energy estimates where the transport structure plays a central role to close the estimates. We use in a particular way a fundamental fact ensuring that the velocity is Lipschitz when we cut at some distance $h$ from the singular set and the maximum growth of the gradient behaves like $-\log h.$
\\
\hspace{0.3cm}
\foreach \x in {\bf b} {%
  \textcircled{\x}
} {\it Co-normal regularity.} In this part we perform suitable  tangential regularity estimates of the vorticity $\partial_X\omega$ using suitable transported vector-fields capturing the main feature of the boundary regularity. We follow basically the same arguments developed in \cite{Hassainia-Hmidi}. The main estimates related to this discussion will be explored in Proposition \ref{prop22} and Proposition \ref{prop220}.  

\foreach \x in {\bf c} {%
  \textcircled{\x}
} {\it Existence and uniqueness .} The construction of the solutions and the  checking of the uniqueness part   will be analyzed along Section \ref{Existence-uniq-inviscid}. It will be conducted in a classical way by smoothing out the initial data and solving the inviscid system \eqref{B0}. Afterwards, we should check  that this procedure generates a family of approximated solutions  which is  bounded uniformly in the resolution space  for a short uniform  time. As to the strong  convergence towards an exact solution to the initial value problem and the uniqueness follow in a straightforward way using similar arguments to \cite{Hassainia-Hmidi}.\\ {
 
  The second part of this paper seeks to address the same problem for the partial viscous system \eqref{B-0-kappa} and especially  we intend  to investigate whether the results of   Theorem \ref{THEO:1:soft} occur with  uniform bounds  for vanishing conductivity $\kappa.$ As we shall see, this is a delicate task due to several technical points among which  the compatibility assumption imposed for the density that could be scrambled by the diffusion effect. In addition, the approach developed  to track the platitude degree is only conceived to  work  with  the transport equations.  Our main result  related to this issue reads as follows  and  whose  general statement can be found in Theorem \ref{th2}.
\begin{theorem}\label{THEO:2:soft}
Let $(s,a)\in(0,1)\times[2,\infty)$. Consider a bounded domain $\Omega_0$ of the plane whose boundary $\partial\Omega_0$ is a Jordan curve of class $C^{1+s}$ outside a closed countable set $\Sigma_0$. Let $\omega_0=1_{\Omega_0},   \rho_0\in W^{2,a}\cap C^{2+s }$ and we  assume   that $\nabla \rho_0$ and $\nabla ^2 \rho_0$ are  identically zero on the singular set $\Sigma_0$. Then, there exists $T>0$ independent of $\kappa\in[0,1]$  for which the  system \eqref{B-0-kappa} admits a unique local solution $(\omega,\rho)$ such that 
 $$
\omega,\rho, \nabla \rho \in L^\infty\big([0,T], L^a\cap L^\infty ) ,
 $$  
 with uniform bounds with respect to $\kappa\in[0,1].$
Moreover, the velocity field $v$ satisfies the estimate 
$$
\sup_{h\in(0,e^{-1}]\atop \kappa\in[0,1]}\frac{\Vert \nabla v(t)\Vert_{L^\infty((\Sigma_t)_h^c)}}{-\log h}\in L^\infty([0,T]),
$$
where  $\Sigma_t\triangleq \Psi(t,\Sigma_0)$ and $\Psi$ denotes the viscous  flow associated to $v$.
  In addition, the boundary of $\Psi(t,\Omega_0)
$  
 is locally in $C^{1+s'}$ outside the set $\Sigma_t$, for all $s'<s$.  

\end{theorem}
Before going over the main ideas of the proof, we shall make some comments.
 \begin{remark}\label{remark-1}
 \begin{enumerate}
 \item Compared to the inviscid case discussed in Theorem $\ref{THEO:1:soft},$ here we need   to reinforce the  regularity on  the density and to strengthen  the compatibility condition in order to balance some undesirable viscosity effects that we cannot overlook. 
 \item Contrary to the inviscid case, where no assumptions are assumed for the singular set $\Sigma_0$, here we require  that this set is  composed with a countable family of points. This restriction can be significantly relaxed and replaced by the integrability assumption \eqref{int:cond0} which allows to cover  more pathological singular sets.
\item In the preceding theorem, we strengthen the regularity and the compatibility condition seen in Theorem $\ref{THEO:1:soft}$ by assuming that $\rho_0\in C^{2+s}$ and $\nabla^2 \rho_0 \equiv \nabla \rho_0 \equiv 0$ on $\Sigma_0.$ These constraints are needed to ensure the platitude assumption \eqref{hypothesis2} required in \mbox{Proposition $\ref{Proposition visous}.$} The degeneracy of the second derivative is one of the side effects brought by the dissipation which tends to diffuse the platitude information far away from the singular set. As to the regularity $\rho_0\in C^{2+s}$ \footnote{In fact, we need only to assume the $C^{2+s}$ regularity on a small neighborhood of the singular set as it is more clear in the statement of Theorem \ref{th2}.},
 we believe that it is quite possible to keep the final results of Theorem $\ref{THEO:2:soft}$ with the less restrictive assumption \eqref{hypothesis2}. It appears that with this latter condition we can derive all the a priori estimates, but in the existence part we need to reinforce the regularity assumption to generate a suitable approximation of the initial data that preserves the platitude condition \eqref{hypothesis2}. For more clarifications, we refer to Paragraph $\ref{sec:xx}.$
  \item To complete our study, we shall prove the inviscid limit of the system \eqref{B-0-kappa} towards \eqref{B0} when the diffusivity parameter goes to zero. In addition, we establish the rate of convergence of the velocities and the densities in Lebesgue spaces, as well as the convergence of the corresponding flows and the boundary parametrization of the patches outside the singular set.  
\end{enumerate}
\end{remark}
Let us now outline the main ideas of the proof and sketch the major difficulties encountered in particular at  the  a priori estimates step. \\
\foreach \x in {\bf a} {%
  \textcircled{\x}
} {\it $L^p$-estimates for $\omega, \nabla\rho$.} This will handled in the proof of Proposition \ref{Proposition visous} through a suitable ansatz and appropriate estimates. The main issue concerns the compatibility condition; that is  $\nabla\rho_0$ and $  \nabla^2\rho_0$ vanish at the singular set $\Sigma_0,$ which is not easy to  maintain during the motion  due to the diffusion effects of the heat equation. When $\kappa=0,$ we have already seen that the transport structure is  helpful and one can track   the platitude degree. In the current context, we should  find another alternative to  compile  the advection and diffusion mechanisms which are in competition. This is considered as  the main technical novelty   of the paper where we shall use different tricks to deal with uniform estimates in transport-diffusion equation with singular potential.    As we shall see, the key ingredients  are  a suitable decomposition for the quantity $\varrho\triangleq |\nabla\rho|^2$ combined with an  optimization arguments supplemented with  De Giorgi-Nash estimates type. Let us briefly discuss the main lines   to establish the results of Proposition  \ref{Proposition visous}.   First, by virtue of  the PDE  inequality \eqref {Eq-varrho} we infer
\begin{align*}
 \big(\partial_t+ v\cdot\nabla-\kappa\Delta \big)\varrho(t,x)\leqslant 2\|{v}(t)\|_{L(\Sigma_t)} \ln^+(d(x,\Sigma_t))\,\varrho(t,x),
\end{align*}
where  the norm  $\|{v}(t)\|_{L(\Sigma_t)}$ is defined in \eqref{Lsigma} and the function $\ln^+$ is stated in \eqref{log+}. Next we make  the splitting  
$$
\varrho\triangleq (\bar\varphi+\psi)\eta,
$$
where $\eta$ is a positive function satisfying the PDE's inequality
\begin{equation*}
\left\{ \begin{array}{ll}
   \big(\partial_t+ \big(v-2\kappa\nabla \log(\bar\varphi+\psi)\big)\cdot\nabla-\kappa\Delta \big)\eta\leqslant 0,&\vspace{1mm}\\ 
  \eta|_{t=0}=1,
  \end{array}\right.
\end{equation*}
the function $\bar\varphi$ satisfies the transport equation
 \begin{equation*}
\left\{ \begin{array}{ll}
   \big(\partial_t+ v\cdot\nabla\big)\bar\varphi(t,x)= 2\|{v}(t)\|_{L(\Sigma_t)} \ln^+(d(x,\Sigma_t))\,\bar\varphi(t,x),&\vspace{2mm}\\ 
 \bar \varphi|_{t=0}=|\nabla\rho_0|^2
  \end{array}\right.
\end{equation*}
and $\psi$ obeys  the following transport-diffusion equation with a singular logarithmic  potential, and subject to a vanishing initial data and a small forcing term 
\begin{equation} \label{Eq04}
\left\{ \begin{array}{ll}
   \big(\partial_t+ v\cdot\nabla-\kappa\Delta\big)\psi(t,x)= 2\|{v}(t)\|_{L(\Sigma_t)} \ln^+(d(x,\Sigma_t))\,\psi(t,x)+\kappa\Delta\bar\varphi,&\vspace{2mm}\\ 
  \psi|_{t=0}=0.
  \end{array}\right.
\end{equation}
Applying the maximum principle we obtain that $\|\eta(t)\|_{L^\infty}\leqslant 1$. The estimate of $\bar\varphi$ can be handled as in the inviscid case and one may track the platitude degree.  Actually, according to \eqref{estimate-nabla phi} one gets for an implicit  short time $T>0$ 
\begin{equation*}
\forall\, t\in[0,T],\quad \|  \bar\varphi(t,\cdot) \|_{L^p}+ \|\nabla \bar\varphi(t,\cdot) \|_{L^p}\leqslant    C_0. 
\end{equation*}  
However, the estimate of $\psi$ turns out to be more intricate due to the presence of the logarithmic potential. To counterbalance its short-range effect we use the smoothing effect of the heat equation together with the smallness of the source term $\kappa\Delta\bar\varphi$. Notice that the arguments vary depending on whether we are controlling $L^p$ with $p$ finite or infinite. In the first case, we implement an energy method based on a suitable integral decomposition related to the potential effect where we distinguish two regimes $d(x,\Sigma_t)\leqslant \varepsilon$ and $d(x,\Sigma_t)\geqslant \varepsilon$. Roughly speaking, in the first region we use a smoothing effect leading to a control of type $\varepsilon^\sigma \kappa^{-1}\|\psi(t)\|_{L^p}^p,$ for some $\sigma>0$ related to the structure of the singular set $\Sigma_0$. Nevertheless, in the second region, we simply get a control of type $\ln(1/\varepsilon)\|\psi(t)\|_{L^p}^p$. Then an optimization argument on $\varepsilon$ leads to the desired estimate. The final estimates do not enable  to conclude about the end limit case $p=\infty$ due to the  fast growth in $p$. To deal with this critical case we proceed with a second splitting of $\psi$ into three parts reflecting different regimes, see \eqref{psi-decompos}. For the worst regime described by $\psi_2$ and connected to the strong effect of the logarithmic  potential, we use De Giorgi-Nash estimates stated in Lemma \ref{Nashlem}. 

Before achieving this  discussion on  the a priori estimates, let us point out the following observation related to the point raised before in Remark \ref{remark-1}-(iii). We do not claim any sort of optimality of the additional constraints on the regularity and the compatibility condition for the density as stated in Theorem \ref{THEO:2:soft}. We expect the result to hold with only the compatibility assumption of \mbox{Theorem \ref{THEO:1:soft},} and we do believe that more refined analysis based on the maximal smoothing effects could be implemented. One of the main difficulties that we have encountered is connected to the $L^p$-estimates of $\psi$. The method that we have performed in the proof of Proposition \ref{Proposition visous} requires some extra-regularity on $\bar{\varphi}$ acting as a source term in \eqref{Eq04}. As a matter of fact, according to the energy estimates, the control of $\nabla \bar{\varphi}$ in $L^p$, for all $p\in[2,\infty]$ seems to be mandatory to get the suitable estimates.

\foreach \x in {\bf b} {%
  \textcircled{\x}
} {\it Co-normal regularity and  loss estimates.} This part is technical and more involved than the inviscid case due to some strong viscosity effects in measuring the tangential regularity of the vorticity and the density. It will be performed following the strategy developed in \cite{Hmidi-0} that we shall overview and isolate the main difficulties. First, the tangential regularity is associated with a suitable transport vector-field denoted here for the sake of simplicity by $X_t$; In principle, we should introduce a cut-off  with a parameter $h>0$ far away the singular set forcing to work with parametrized vector-fields. Thus from \eqref{X-t-omega} we get
\begin{equation*}
\big(\partial_t+v\cdot\nabla\big)\partial_{X_t}\omega=\partial_{ X_t}\partial_1\rho.
\end{equation*} 
Even though the velocity is not expected to be Lipschitz everywhere, the fact that the vector fields $X_t$ are supported far away the singular set at a distance $h$ with some degeneracy behavior, one can get a regularity persistence according to Proposition \ref{p11}. Then
to control $\partial_X \omega$ in the negative H\"older space $C^{s-1}$, with $s\in(0,1)$, it is enough to get an estimate of $\|\partial_X\rho(t)\|_{C^{s}}.$ This latter quantity introduces several technical issues. First, under Proposition \ref{Proposition visous} we can control at most $\|\nabla\rho\|_{L^\infty}$ and there is no hope to improve it in general due to the lack of regularity of the velocity. The alternative for estimating $\|\partial_X\rho(t)\|_{C^{s}}$ is to write down the dynamics governing the tangential regularity of the density stated in \eqref{TG}\begin{equation}\label{T-G}
\left\{\begin{array}{ll}
\big(\partial_t+v\cdot\nabla-\kappa \Delta\big)\partial_{X_t}\rho=-\kappa[\Delta,\partial_{X_t}]\rho &\vspace{1mm}\\
\partial_{X_t}\rho_{|_{t=0}}=\partial_{X_0}\rho_0.
\end{array}
\right.
\end{equation}
At this stage, the delicate point is to estimate the commutator which induces a sensitive loss of regularity and the main scope is to explore whether the heat equation can remedy to this defect through its maximal smoothing effect. For this aim, we proceed with a suitable decomposition of the commutator using para-differential calculus and perform two sorts of maximal smoothing effects: the first one is inherited from the transport-diffusion part in \eqref{T-G} and the second one stems from the density equation. We point out that the smoothing effects are  established only  with  Besov spaces constructed over  Lebesgue spaces $L^p$ with $p<\infty,$ see Proposition \ref{Max-reg}. Notice that there is an apparent frequency logarithmic loss of regularity connected to the low regularity of the velocity. This has a direct impact on the regularity persistence as stated in Proposition \ref{prop2222}. Actually, we are only able to propagate estimates  with some time loss and this explains in part the loss of regularity observed for the boundary in Theorem \ref{THEO:2:soft}. Another remark to state concerns the endpoint case $p=\infty$ which seems to be completely out of reach and it  is left open. To cover this case, we could be tempted to use the Lagrangian coordinates method developed successfully in \cite{Hmidi}, but unfortunately, it seems to work only when the velocity is Lipschitz, which is not the case here.\\

 \quad \ding{70} {{\bf{Outline of the paper}}. In Section \ref{Useful-Tec}, we shall introduce  a package of various notations freely used throughout this work. Afterward, we state some results on Littlewood-Paley theory, singular vortex patches, function spaces related to  the co-normal regularity of the vorticity, logarithmic estimate  and we end with a discussion on the maximal smoothing effects for   transport-diffusion equations governed by a $\log$-Lipschitz} velocity. In Section \ref{Inviscid-case}, we study the inviscid case governed by \eqref{B0} where we prove a  general version of Theorem \ref{THEO:1:soft}, see Theorem \ref{Th1:general:version}.  The proof will be performed into several steps, from the   $L^p$ a priori estimates for the vorticity and the gradient of the density, see Proposition \ref{Proposition inviscid}, to  the propagation of anisotropic regularity. Section \ref{Viscous-case} deals with the viscous  system \eqref{B-0-kappa} and aims at proving the  general statement  of  Theorem  \ref{THEO:2:soft},  described through Theorem \ref{th2}.  In section \ref{Inviscid-limit}, we study the inviscid limit of \eqref{B-0-kappa} towards \eqref{B0} when the diffusivity parameter goes to zero. In the final Section \ref{Appendix}, we collect some technical results used along the paper.

 \section{Techniques related to  the vortex patch problem}\label{Useful-Tec}
The main concern of this introductory section is to collect some concepts and discuss useful tools required in studying the vortex patch problem. We distinguish two important parts. The first one covers some results on the functional material centered around some aspects of the para-differential calculus, classical Besov spaces, and anisotropic Besov spaces. The second one focuses on the regularity persistence and the maximal smoothing effects for transport-diffusion equations when the drift is given by a velocity field that scales slightly below the Lipschitz class. This is typically the case for  vortex patches whose singular points are of corners type.
\subsection{Notation} 
For the convenience of the reader, we begin with introducing the main notation and conventions that will be used frequently in the sequel of this paper.
\begin{itemize}
\item We shall use the notation $A\lesssim B$ to denote $A\leqslant CB$, for some constant $C>0$ that is independent of the main quantities of the problem, namely, which is independent of $A$ and $B$. This constant is also allowed to differ from line to another.
\item We use $C_0$ (resp. $C_{0,T}$) to  denote a constant depending on the initial data of the problem (resp. and on $T$). This constant   may  vary from line to line.
\item  The notation $f\triangleq g$ means that $f$ equals to $g$ by definition.
\item The commutator between two operators $A$ and $B$ is defined by 
$$[A,B]\triangleq AB - BA .$$
\item For any non--empty subset $F\subset\RR^2$, the characteristic function of $F$, denoted by ${\bf 1}_F$, is defined by
\begin{equation*}
{\bf 1}_F(x)=\left\{\begin{array}{ll}
1 & \textrm{if $x\in F$,}\vspace{2mm} \\ 
0 & \textrm{if $x\in\RR^2\setminus F$.}
\end{array}
\right.
\end{equation*} 
\item For any $h\geqslant 0$ and for any non--empty subset $F\subset \mathbb{R}^2$, we denote 
\begin{equation}\label{distance-h:def}
 F_h \triangleq \left\{ x\in \mathbb{R}^2 :\; d(x,F)\leqslant h\right\}, \quad F_h^c \triangleq \left\{ x\in \mathbb{R}^2 :\; d(x,F)> h\right\},
\end{equation} 
with $d(x,F)$ is the distance of $x$ to $F$ given by
$$
d(x,F)=\inf_{y\in F}d(x,y).
$$
We remark that this function is $1$-Lipschitz, which  implies that its  derivative in the distribution sense is bounded, with 
$$
\sup_{x\in\RR^2}|\nabla d(\cdot,F)|\leqslant 1.
$$
\item For any non-empty subset $F\subset \mathbb{R}^2$ and for any  vector-field $v:\mathbb{R}^2\to\mathbb{R}^2$  we define  
\begin{equation}\label{Lsigma}
\|v(t)\|_{L(F)}\triangleq\sup_{h\in(0,e^{-1}]}\frac{\Vert  \nabla v(t)\Vert_{L^\infty(F_h^c)}}{-\log h}.
\end{equation}

\item  We define the function $\ln^+: (0,\infty) \longrightarrow [1,\infty)$ through
\begin{equation}\label{log+} 
\ln^+(t)=\left\{ \begin{array}{ll}
  -\ln(t),\quad\hbox{if} \quad 0<t\leqslant e^{-1}\vspace{2mm}\\
 1,\quad\hbox{if}\quad  t\geqslant e^{-1}.
  \end{array}\right.
\end{equation} 
\item For any non--empty set $\Sigma_0\subset\RR^2$, we define the functions $\chi$ and $\varphi_0$ by
\begin{equation} \label{tool:varphi0}
\chi(\tau)\triangleq \left\{ \begin{array}{ll}
 \tau,\quad\hbox{if} \quad \tau\leqslant e^{-1}\vspace{2mm}\\
  e^{-1},\quad\hbox{if}\quad  \tau\geqslant e^{-1}
  \end{array}\right.\quad\hbox{and}\quad 
\varphi_0(x)\triangleq\chi\big(d(x,\Sigma_0)\big), \; \text{ for } x\in \mathbb{R}^2.
\end{equation}
\end{itemize}

\subsection{Functional toolbox}
In this section, we  intend to gather some useful tools and state classical results related to frequency localization and functional spaces. Some aspects of para-differential calculus will be explored at the end. \\
Consider  a radial  function decreasing function $\chi\in\mathscr{D}(\RR^2)$ be a reference cut-off function, monotonically decaying along rays and so that
\begin{equation*}
\left\{\begin{array}{ll}
\chi\equiv1 & \textrm{if $ |\xi |\leqslant\frac12$,}\vspace{3mm}\\ 
0\leqslant \chi\leqslant1 & \textrm{if $\frac12\leqslant |\xi |\leqslant1$,}\vspace{3mm}\\
\chi\equiv0 & \textrm{if $ |\xi |\geqslant 1.$}
\end{array}
\right.
\end{equation*}
Define $\varphi(\xi)\triangleq\chi(\tfrac{\xi}{2})-\chi(\xi),$ then we have  $\varphi\geqslant 0,\, \supp\varphi\subset\mathcal{C}\triangleq\{\xi\in\RR^2:\frac12\leqslant |\xi |\leqslant 1\}$, and 
$$
\forall\, \xi\in\RR^2,\quad \chi(\xi)+\sum_{q\in\NN}\varphi(2^{-q}\xi)=1.
$$
The Littlewood-Paley  operators are defined as follows: for every $v\in\mathcal{S}'(\RR^2)$, 

\begin{equation*}
\Delta_{-1}v\triangleq\chi(\sqrt{-\Delta})v,\quad \forall\, q\in\NN,\quad  \Delta_{q}v\triangleq\varphi(2^{-q}\sqrt{-\Delta})v\quad \hbox{and}\quad  S_{q}v\triangleq\sum_{j\leqslant q-1}\Delta_{j}v.
\end{equation*}
The first useful application  is the decomposition
$$
v=\sum_{q\ge-1}\Delta_q v , \quad \text{in } \mathcal{S}'(\RR^2).
$$
Moreover, the operators $\Delta_q$ and $S_q$ enjoy the following properties, valid for all $u,v$ in $\mathcal{S}'(\mathbb{R}^2)$
\begin{enumerate}
\item[(i)] $\vert p-q\vert\ge2\Longrightarrow\Delta_p\Delta_q v\equiv0$,\vspace{2mm}
\item[(ii)] $\vert p-q\vert\ge4\Longrightarrow\Delta_q(S_{p-1}u\Delta_p v)\equiv0$,\vspace{2mm}
\item[(iii)] $\Delta_q, S_q: L^m \longrightarrow L^m$ uniformly with respect to  $q\geqslant -1$ and $m\in [1,\infty]$.
\end{enumerate}

Next, we recall the well-known Bernstein inequalities, see for instance \cite{Bahouri-Chemin-Danchin}.
\begin{lemma}\label{Bernstein} There exists a constant $C>0$ such that for $1\leqslant a\leqslant b\leqslant\infty$, for every function $v$ and every $q,k\in\NN $, we have
\begin{equation*}
\sup_{\vert\alpha\vert=k}\Vert\partial^{\alpha}S_{q}v\Vert_{L^{b}}\leqslant C^{k}2^{q\big(k+2\big(\frac{1}{a}-\frac{1}{b}\big)\big)}\Vert S_{q}v\Vert_{L^{a}},\\
\end{equation*}
\begin{equation*}
C^{-k}2^{qk}\Vert\Delta_{q}v\Vert_{L^{a}}\leqslant\sup_{\vert\alpha\vert=k}\Vert\partial^{\alpha}\Delta_{q}v\Vert_{L^{a}}\leqslant C^{k}2^{qk}\Vert\Delta_{q}v\Vert_{L^{a}}.
\end{equation*}
\end{lemma}
Besov spaces are defined below.
\begin{definition} For $(s,p,r)\in\RR\times[1,  \infty]^2$. The non-homogeneous Besov space $B_{p,r}^s$   is the set of all the tempered distributions $v\in\mathcal{S}^{'}$   such that
\begin{eqnarray*}
&&\Vert v\Vert_{{B}_{p, r}^{s}}\triangleq\|2^{qs}\Vert \Delta_{q} v\Vert_{L^{p}}\|_{\ell^r(q\geqslant -1)}<\infty.  
\end{eqnarray*}
\end{definition}

\begin{remark} In the  case $p=q=\infty,\; k\in\NN$ and $ \alpha\in (0,1)$ the Besov space  $B_{\infty,\infty}^{k+\alpha}$ coincides with the H\"older space $C^{k+\alpha}$. More precisely, we have
\begin{equation}\label{N-equivalent}
\Vert v\Vert_{{B}_{\infty, \infty}^{k+\alpha}}\lesssim \|v\|_{C^{k+\alpha}}\triangleq \Vert v\Vert_{L^\infty}+\sup_{\underset{|\alpha|=k}{x\neq y}}\tfrac{\vert \partial^{\alpha} v(x)-\partial^{\alpha}v(y)\vert}{\vert x-y\vert^s}\lesssim\Vert v\Vert_{{B}_{\infty, \infty}^{k+\alpha}}.
\end{equation}
To simplify the notations in some estimates, we will also use the notation
$$  \|v\|_{{s}} \triangleq  \|v\|_{C^{s}} ,\quad \forall s\in \mathbb{R}.$$
 \end{remark}
Let us now recall the definition of the Log-Lipschitz space. 
\begin{definition}
We denote  by ${LL}$ the space of log-Lipschitz functions, that is the set of functions  $v:\RR^2\to\RR^2$  satisfying
$$
\Vert v \Vert_{LL}\triangleq  \underset {0<\vert x- y\vert<1}{\sup}\frac{\vert v(x)-v(y)\vert}{\vert x-y\vert  \log \frac{e}{\vert x-y \vert} }<\infty.
$$
\end{definition}
The following classical estimate is very useful. It is a simple consequence of the embedding $B_{\infty,\infty} ^1\subset LL$ combined with Bernstein inequality and the  Biot--Savart law
\begin{equation}\label{b-s}
v(x)=\frac{1}{2\pi}\int_{\RR^2}\frac{(x-y)^\perp}{\vert x-y\vert^2}\omega(y)dy, x^\perp = (-x_2,x_1).
\end{equation}
\begin{lemma}\label{lem3}\cite[Chapter 7]{Bahouri-Chemin-Danchin} 
Let $v$ be a vector field  given by \eqref{b-s}. Then, for any   $a\in [1, \infty)$, we have
$$
\Vert v\Vert_{LL}\leqslant C \Vert \omega\Vert_{L^a\cap L^\infty},
$$
with $C$ depending only on $a$.
\end{lemma}
The characterization below of the space  $LL$ is  useful. 
\begin{proposition}\label{C-LL}\cite[Proposition 2.111]{Bahouri-Chemin-Danchin} For any function $v\in LL$, we have 
\begin{eqnarray*}
 \|v\|_{LL}&\lesssim& \sup_{q\in \mathbb{N}}\frac{\|\nabla S_{q}v\|_{L^\infty}}{2+q}\lesssim\|v\|_{LL}
 \end{eqnarray*}
 and for all $q\in \mathbb{N}$
 \begin{eqnarray*}
 \|\Delta_q v\|_{L^\infty}&\lesssim&  \|v\|_{LL}(2+q)2^{-q}.
\end{eqnarray*}
\end{proposition}
The well-known  {\it Bony's} decomposition \cite{Bahouri-Chemin-Danchin} enables us to split formally the product of two tempered distributions $u$ and $v$ into three pieces. In
what follows, we shall adopt the following definition 
\begin{definition}\label{Deff-Bony} For a given $u, v\in\mathcal{S}'$ we have
 $$
uv=T_u v+T_v u+R(u,v),
$$
with
$$T_u v=\sum_{q}S_{q-1}u\Delta_q v,\quad  R(u,v)=\sum_{q}\Delta_qu\widetilde\Delta_{q}v  \quad\hbox{and}\quad \widetilde\Delta_{q}=\Delta_{q-1}+\Delta_{q}+\Delta_{q+1}.
$$
\end{definition}
The following lemma provides some useful product laws in Besov spaces. 
\begin{lemma}\label{para--products}
Let  $(s,s_1, s_2,p)\in\mathbb{R}\times(-\infty,0) \times (0,\infty)\times[1, \infty].$ Let $T$ and $R$ be the operators given by Definition $\ref{Deff-Bony}$. Then, the following estimates hold.
\begin{equation*}
\|T_u v \|_{B^{s_1+ s}_{p,\infty}}\lesssim \| u\|_{B^{s_1}_{p,\infty}}\|v \|_{B^s_{\infty,\infty}}.\vspace{2mm}
\end{equation*} 
\begin{equation*}
\|T_u v \|_{B^{s}_{p,\infty}}\lesssim \| u\|_{L^\infty}\|v \|_{B^s_{p,\infty}}.\vspace{2mm}
\end{equation*} 
\begin{equation*}
\|R(u,v)\|_{B^{s_2}_{p,\infty}}\lesssim \| u\|_{L^\infty}\|v \|_{B^{s_2}_{p,\infty}}.
\end{equation*}
\end{lemma}
The proof of Lemma \ref{para--products}  is elementary, see for instance \cite[Chapter 2]{Bahouri-Chemin-Danchin}. As a corollary, one obtains the following useful estimate.
 \begin{coro}\label{ppxr0}
Let $(s,p)\in(0,1)\times[1,\infty]$ and $X$ be a vector field belonging to $B^s_{p,\infty}$ as well as its divergence. Let $f$  be a Lipschitz scalar function. Then, for $j\in\{1,2\}$, we have
$$
\Vert (\partial_{x_j}X)\cdot\nabla f\Vert_{B^{s-1}_{p,\infty}}\leqslant C\Vert \nabla f\Vert_{L^\infty}\big(\Vert \textnormal{div}X\Vert_{B^{s }_{p,\infty}}+\Vert X\Vert_{B^{s}_{p,\infty}}\big).
$$ 
\end{coro}
The proof of Corollary \ref{ppxr0} is based on Lemma \ref{para--products}. One may see also \cite{Hassainia-Hmidi} for a detailed proof in the case $p=\infty$, thereby, the general case $p\in [1,\infty]$ can be treated similarly.\\

We conclude this section by recalling two lemmas that will be helpful in the proof of Theorem \ref{THEO:2:soft}. For more details on the  proof of the first one, we refer to \cite[Lemma 5]{Hmidi-0}. Meanwhile, the second one is based on \cite[Lemma 7]{Hmidi-0}.
\begin{lemma}\label{inverse-embedding:lemma}
Let $(s,p)\in\RR_{+}\times[1,\infty]$ and $K$ be  a compact in $\mathbb{R}^d$, with $d\geqslant 1$. There exists a constant $C_K$ such that, for all tempered distribution $v$ in $C^s$, supported in $K$, we have 
\begin{equation*}
\|v\|_{B^s_{p,\infty}}\leqslant C_K \|v\|_{s}.
\end{equation*}
\end{lemma}
\begin{lemma}\label{lemma:partition of unit}
  There exists a partition $(\mathscr{O}_n)_{n\in\NN^{\star}}$ of $\RR^d$, for $d\geqslant 1$, and a sequence of positive smooth functions $(\psi_{n})_{n\in\NN^\star}$ satisfying
\begin{enumerate}
\item $\supp\;\psi_n\subset\mathscr{O}_n,$ with $|\mathscr{O}_n|\lesssim 1$,
\item $\displaystyle\sum_{n\in\mathbb{N}^\star}\psi_n(x)=1$, for all $  x\in \mathbb{R}^d$,
\item there exists $  N_d>0$ such that, for all $   x\in \mathbb{R}^d $, $\Card\Big\{n\in\NN^\star: x\in\mathscr{O}_n\Big\}\leqslant N_d$, 
\item for all $\alpha\in   \mathbb{N}^d$ and $p\in [1,\infty]$, we have that
\begin{equation}\label{psin-EST}
\|\nabla ^{\alpha} \psi_n\|_{L^\infty} \leqslant C_{\alpha,p},\quad \forall n\in \mathbb{N}^{\star}.
\end{equation}
\end{enumerate}
\end{lemma}
For the convenience of the reader, we present below a short self-contained proof of Lemma \ref{lemma:partition of unit}.
\begin{proof}
For simplicity, we only outline the idea of the proof in the two dimensional case $d=2.$ Let us introduce the following notations, for $r>0$ and $n\in \mathbb{N}^\star$, we denote
\begin{equation*}
\widetilde{B}_r\triangleq B(0,r) , \quad  
\widetilde{B}_{r,n}\triangleq \widetilde{B}_r + x_{n }. 
\end{equation*}
Above, the sequence $(x_{n})_{n\in \mathbb{N^\star}}$ is such that $\cup_{n\in \mathbb{N}^\star}\{x_n\} = \mathbb{Z}\times \mathbb{Z}.$ Moreover, we have 
\begin{itemize}
\item The collection $(\widetilde{B}_{2,n })_{n\in \mathbb{N}^*}$ is a recovery of $\mathbb{R}^2$ satisfying 
$$|\widetilde{B}_{2,n }| = |\widetilde{B}_{2  }|=4\pi.$$
\item There exists $N >0$ such that, for all $x\in \mathbb{R}^2$,  we have that $\Card\big\{ n\in\mathbb{N}^\star : x\in \widetilde{B}_{2,n } \big\} \leqslant N .$
\end{itemize}
Now, consider a smooth function $\varphi$ in $\mathcal{C}^\infty_c(\mathbb{R}^2)$ with values in $[0,1]$ such that 
\begin{equation*}
\varphi(x)=\left\{ \begin{array}{ll}
1, & \text{if } x\in \widetilde{B}_1,\vspace{2.5mm} \\
0, & \text{if } x\notin \widetilde{B}_2.
\end{array}
\right.
\end{equation*}
Thereby, let us define the sequence of functions $\varphi_n$ as 
$$\varphi_n(x) \triangleq \varphi(x-x_n). $$
 Remark that, by construction, $\varphi_n$ takes its values in $[0,1]$, for all $n\in \mathbb{N}^*$, and moreover we have that 
  $$\supp\, \varphi_n \subset \widetilde{B}_{2,n }, \quad \forall n\in \mathbb{N}^* $$
  and 
\begin{equation*}
\|\nabla ^{\alpha}\varphi_n \|_{L^p}\lesssim C_{\alpha,p}  , \quad \forall (p,\alpha,n)\in [1,\infty]\times \mathbb{N}^2\times \mathbb{N}^\star.
\end{equation*}
Thereafter, we define the function $\phi$ as
\begin{equation*}
\phi(x) \triangleq \sum_{n\in \mathbb{N}^\star} \varphi_n(x) .
\end{equation*} 
Observe that $\phi$ is well-defined and smooth since the above sum is in fact a finite sum containing at most $N$ terms. Furthermore, we have, for all $x\in \mathbb{R}^2$
\begin{equation*}
1 \leqslant \phi(x) \leqslant N.
\end{equation*}
Now, we define the partition of unity $\psi_n$ as 
\begin{equation*}
\psi_n \triangleq \frac{\varphi_n}{\phi}.
\end{equation*} 
Hence, it is easy to check that $\psi_n$ is smooth  for all $n\in \mathbb{N}^*$ and satisfies all the required claims of Lemma \ref{lemma:partition of unit}.
\end{proof}
\subsection{Vortex patch formalism}
The main goal in this subsection is twofold. First, we collect  some valuable  results on the set dynamics inclusion for a flow associated with a log-Lipschitz velocity. Second, we recall some definitions related to the anisotropic Besov and H\"older spaces frequently used to track the co-normal regularity of the patch. We end this short discussion with a logarithmic estimate which plays a crucial role, in particular in the vortex patch formalism. More details can be found in \cite{Chemin, Hassainia-Hmidi, Hmidi-0}.

\hspace{0.5cm}The flow associated to a vector field $v$ is defined by the mapping $\Psi:\RR_{+}\times\RR^2\longrightarrow\RR^2$ which satisfies the following integral equation
$$
\forall\, t\geqslant 0,\, \forall x\in\RR^2,\quad \Psi(t,x)=x+\int_{0}^{t}v\big(\tau,\Psi(\tau,x)\big)d\tau.
$$

First, we embark with the following result which describes the dynamics of a given set through the flow associated with a $LL$ vector field. The complete proof can be found in \cite{Chemin}.
\begin{lemma}\label{s2lem1}
 Let $A_0$ be a subset of $\RR^2, v$ be a vector field belonging  to $L^1_{loc}(\RR_+ ;{LL})$ and  $\Psi$  its flow. Then,  by setting $A(t) \triangleq\Psi(t,A_0)$ we get in view of notation \eqref{distance-h:def}
\begin{align}\label{def-delta}
\Psi\big(t,(A_0)_h^c\big) \subset \big(A(t)\big)^c_{\delta_t(h)},\quad \textnormal{with}\quad \delta_t(h) \triangleq h^{\exp\int_0^t\Vert v(\tau)\Vert_{LL}d\tau}.
\end{align}
For all $0 \leqslant \tau \leqslant t$,
$$
\Psi\big(\tau,\Psi^{-1}\big(t,(A_t)_h^c)\big) \subset \big(A(\tau)\big)^c_{\delta_{\tau,t}(h)},\quad \textnormal{with}\quad \delta_{\tau,t}(h) \triangleq h^{\exp\int_\tau^t\Vert v(\tau')\Vert_{LL}d\tau'}.
$$
\end{lemma}
 As a consequence,  we state the following elementary lemma which will be of a constant use in different situations that will be developed in this paper. 
\begin{lemma}\label{elementary propertiy LN 2}
Let $\Psi$  denote  the flow associated with a log-Lipschitz vector field $v\in L^1([0,T];LL)$, for some $T>0$. Let $\Sigma_0$ be a closed subset of $\mathbb{R}^2$  and denote 
\begin{align}\label{Sigma-def}
\Sigma_t\triangleq \Psi(t,\Sigma_0). 
\end{align}
Then for all $(t,x)\in [0,T]\times\mathbb{R}^2$, there holds
$$\ln^+ \left(d(x,\Sigma_t) \right)\leqslant e^{\int_0^t\|v(\tau)\|_{LL}d\tau} \ln^+ \left[ d(\Psi^{-1}(t,x),\Sigma_0)\right]. $$
\end{lemma}
\begin{proof}
From Lemma \ref{s2lem1} we find for $h\leqslant e^{-1}$
$$
d\big(\Psi^{-1}(t,x),\Sigma_0\big)\geqslant h\Longrightarrow d\big(x,\Sigma_t\big)\geqslant h^{\exp\int_0^t\Vert v(\tau)\Vert_{LL}d\tau}.
$$ 
Then using  the decreasing property of $t\mapsto \ln^+ t,$ combined with 
$$\ln^+(t^m) \leqslant m \ln^+ (t),\quad \forall t>0, \quad \forall m\geqslant 1,
\vspace{1mm}
$$
$$\ln^+\left(\min\{ t, e^{-1}\}\right) = \ln^+(t), \quad \forall t>0,$$
 and the particular choice  $h= \min \{ d\big(\Psi^{-1}(t,x),\Sigma_0\big), e^{-1} \}$ we obtain  the desired inequality.  This achieves the proof of the lemma.
\end{proof}
Now, let us setup the main tools required in  the  study the vortex patch problems. For more details about this subject, we refer to \cite{Chemin} and the references therein. 

\begin{definition} \label{Defintion-2.4}Let $(s,p)\in (0,1)\times  [1 ,\infty]$, then for any closed set $\Sigma$ of the plane  and any family of vector fields $\mathcal{X}=(X_\lambda)_{\lambda\in \Lambda}$ belonging to $B^s_{p,\infty}$ as well as their divergence,  we define the quantities
$$
I(\Sigma,X)\triangleq \inf_{x\notin\Sigma}\sup_{\lambda\in\Lambda}\vert X_\lambda(x)\vert ,
$$
\vspace{1mm}
$$
\widetilde{\Vert} X_\lambda\Vert_{B^s_{p,\infty}}\triangleq\Vert X_\lambda\Vert_{B^s_{p,\infty}}+\Vert \textnormal{div}X_\lambda\Vert_{B^s_{p,\infty} }
$$
and 
$$
N_{s,p}(\Sigma,X)\triangleq\sup_{\lambda\in\Lambda}\frac{\widetilde{\Vert} X_\lambda\Vert_{B^s_{p,\infty}}}{I(\Sigma,X)}.
$$
For each element  $X_\lambda$ of the previous family we define its action on bounded  real-valued functions $u$ in the weak sense as 
$$
\partial_{X_\lambda}u\triangleq \textnormal{div}(u\, X_\lambda)-u\, \textnormal{div}X_\lambda.
$$ 
Furthermore, for all $k\in \mathbb{N},$ we set 
 \begin{equation}\label{defre}
\Vert u\Vert^{s+k,p}_{\Sigma ,\mathcal{X} }\triangleq  N_{s,p}(\Sigma ,\mathcal{X} )\sum_{\vert \alpha\vert\leqslant   k} \Vert  \partial^\alpha u\Vert_{L^\infty}+\displaystyle{\sup_{\lambda\in\Lambda}}\, \frac{\Vert \pxl u\Vert_{B^{s+k-1}_{p,\infty}}}{I(\Sigma ,\mathcal{X} )}.
\end{equation} 
In the particular case $p=\infty$, we simplify the notation by removing  the exponent $p$ from the definitions.  That is for instance, we use the notation 
\begin{equation}\label{defre2}
\Vert u\Vert^{s+k }_{\Sigma ,\mathcal{X} }\triangleq\Vert u\Vert^{s+k,\infty }_{\Sigma ,\mathcal{X} }.  
\end{equation} 
The Besov space  $B_{\Sigma,\mathcal{X} }^{s+k,p}$ (respectively, $C_{\Sigma,\mathcal{X} }^{s+k}$)  is then defined as the set of  functions $u$ in $W^{k,\infty}$ such that
 $   \Vert u\Vert^{s+k,p }_{\Sigma ,\mathcal{X} }  
$ (respectively, $ \Vert u\Vert^{s+k  }_{\Sigma ,\mathcal{X} }$) is finite.
\end{definition} 
Let us now introduce a more accurate definition of the admissible vector fields that we will use in the sequel.
\begin{definition}\label{def11}
Let $\Sigma$  be a closed subset of $\RR^2$ and $\Xi = (\alpha, \beta, \gamma)$ be a triplet of non-negative real numbers. We consider a family $\mathcal{X} = (X_{\lambda,h})_{(\lambda,h)\in \Lambda\times (0,e^{-1}]}$ of vector fields  belonging to $B^s_{p,\infty}$ as well as their divergences, with $s\in(0,1)$ and we denote by $\mathcal{X}_h =(X_{\lambda,h})_{\lambda\in\Lambda}$. The  family $\mathcal{X}$ will be said  $\Sigma-$admissible of order $\Xi=(\alpha,\beta,\gamma)$ if and only if the following properties are satisfied:
\begin{equation*}
\forall (\lambda, h)\in\Lambda\times (0,e^{-1}],\ \textnormal{supp}X_{\lambda,h}\subset\Sigma_{h^\alpha}^c,
\end{equation*}
\vspace{0.5mm}
\begin{equation*}
 \inf_{ h\in(0,e^{-1}]}h^{-\gamma} I(\Sigma_h,\mathcal{X}_h)>0
\end{equation*}
and
\begin{equation*}
 \sup_{ h\in(0,e^{-1}]}h^{\beta}N_{s,p}(\Sigma_h,\mathcal{X}_h)<\infty.
\end{equation*}

\end{definition}
Some important remarks are in order.
\begin{remark}
Notice that in the pioneering  work \cite{Chemin}, Chemin sets up the above formalism only in the H\"older case $p=\infty$ which fits well  with the regularity persistence  for  the transported singular vortex patch associated to Euler equations. However, in the viscous case    \cite{Hmidi-0} additional technical issues emerge forcing to consider a family of vector fields $(X_\lambda)_{\lambda\in\Lambda}$ belonging to Besov classes $B^s_{p,\infty}$, with $p<\infty$. The main difficulty behind this restriction stems from  the maximal smoothing effects in the density equation  advected by a velocity field which scales below the Lipschitz class, see also \mbox{Remark $\ref{remark:p<infini}.$}
\end{remark}

 
\begin{remark}
Concretely, the family of vector fields $\mathcal{X}$ that we shall work with vanishes near the singular set and therefore we should assume that   $\gamma, \beta>0$. Albeit the parameter $\alpha>1$. 
\end{remark} 

The next result deals with a logarithmic estimate established  in \cite{Chemin} which is the main  key in the study of the generalized vortex patches. More precisely, we have the following estimate.

\begin{theorem}\label{propoo1} 
There exists an absolute constant $C $ such that for any $a\in [1,\infty)$ and $ s\in (0,1)$ the following holds. 
Let  $\Sigma$ be a closed set of the plane and $X$ be a  family  of vector fields. Consider a function $\omega \in C^s _{\Sigma,X}\cap L^a$ and let $v$ be the divergence-free vector field with vorticity $\omega$, then we have 
\begin{equation*}\label{t1}
\Vert\nabla v\Vert_{L^\infty(\Sigma^c)}\leqslant Ca\Vert\omega\Vert_{L^a}+\frac{C}{s}\Vert\omega\Vert_{L^\infty}\log\Big(e+\frac{\Vert\omega\Vert^{s}_{\Sigma,X}}{\Vert\omega\Vert_{L^\infty}}\Big).
\end{equation*}
 Above, the norm $ \Vert\cdot \Vert^{s}_{\Sigma,X}$ is given by \eqref{defre2}. 
\end{theorem} 
Now we discuss some important results for solutions to transport equations. 
\begin{proposition}\label{prop:transport:es}
Let  $\Sigma_0$ be a closed set in the plane, $(p,h)\in[1,\infty]\times (0,e^{-1}]$  and $v$ be a smooth divergence-free vector field with vorticity $\omega\triangleq \nabla \times v$. Then, using the notations \eqref{def-delta} and  \eqref{Sigma-def}  the following assertions hold.
Let  $f$ be a solution of the transport equation on $[0,T]$
\begin{equation}\label{T}
\left\{ \begin{array}{ll}
\partial_{t}f+v\cdot\nabla f =g, &\vspace{1mm}\\
f_{| t=0}=f_{0},\tag{T}
\end{array} \right.
\end{equation} 
where, we assume that $\supp \, f_0\subset (\Sigma_0)_h^c$. Let $t\in[0,T]\mapsto s_t\in (-1,1)$ be  a decreasing function, $\delta_t(h) $ be defined as in \eqref{def-delta} and let $g=g_1+g_2$ be given such that $\supp \, g(t)\subset (\Sigma_t)_{\delta_t(h)}^c$ for any $t\in[0,T]$ and we suppose  that $g_2$ satisfies 
$$
\Vert g_2(t)\Vert_{B^{s_t}_{p,\infty}}\leqslant -C\log h\;W (t)\Vert f(t)\Vert_{B^{s_t}_{p,\infty}}, \quad \forall t\in [0,T],
$$
for some universal constant $C>0$, where $W(t)$ is given by 
\begin{equation}\label{W:def}
 W  (t)\triangleq\big(\Vert  v(t)\Vert_{L(\Sigma_t)}+\| \omega (t)\|_{L^a\cap L^\infty} \big)\exp\bigg(\int_0^t\Vert v(\tau)\Vert_{LL}d\tau\bigg).
\end{equation}
Then, we have
$$
\Vert f(t)\Vert_{B^{s_t}_{p,\infty}}\leqslant \Vert f_0\Vert_{B^{s_0}_{p,\infty}} h^{-C\int_0^t W (\tau)d\tau}+\int_0^t h^{-C\int_\tau^t W (\tau')d\tau'}\Vert g_1(\tau)\Vert_{B^{s_\tau}_{p,\infty}} d\tau, \quad \forall t\in [0,T].
$$ 
\end{proposition}
The proof of Proposition \ref{prop:transport:es} is straightforward, we refer to \cite{Chemin,Hmidi} for more details. As a consequence, we get  the following corollary that will useful in this paper.
\begin{coro}\label{coro-transport:es}
Let $(s,h,a,p)\in(0,1)\times (0,e^{-1}]\times(1,\infty) \times [1,\infty]$. Consider  a non-increasing function $t\in[0,T]\mapsto\sigma_t\in (0,s]$.  Let $\Sigma_0$ be a closed set of the plane and $X_0$ be a vector field of class $B^{s}_{p,\infty}$ as well as its divergence and whose support  is embedded in $(\Sigma_0)_h^c$. Let $v$ be a divergence-free vector field with vorticity $\omega$ belonging to $L^a\cap L^\infty.$ If $ X_t$ is the solution of 
\begin{equation}\label{31}
\left\{ \begin{array}{ll}
\big(\partial_t +v\cdot\nabla\big)X_t=\pxt v, & \vspace{2mm}\\
X_{| t=0}=X_{0},
\end{array} \right.
\end{equation} 
then, the following assertions hold
\begin{equation*}
\supp X_t\subset \big(\Sigma_t\big)_{\delta_t(h)}^c,
\end{equation*} 
\begin{equation}\label{div:1}
\Vert \Div X_t\Vert_{B^{s}_{p,\infty}}\leqslant \Vert \Div X_{0}\Vert_{B^{s}_{p,\infty}}h^{-C\int_0^t W (\tau)d\tau},
\end{equation}
\begin{eqnarray}\label{38-E}
\widetilde{\Vert} X_t\Vert_{B^{\sigma_{t}}_{p,\infty}}&\leqslant & h^{-C\int_{0}^{t}W(\tau)d\tau}\Big(\widetilde{\Vert} X_0\Vert_{B^{s}_{p,\infty}} +C\int_0^t h^{C\int_{0}^{\tau}W(\tau')d\tau'}\|\partial_{ X_{\tau}}\omega(\tau)\|_{B^{\sigma_{\tau}-1}_{p,\infty}}d\tau\Big),
\end{eqnarray}
where $W(t)$ and  $\delta_t(h)$ are, respectively, defined by \eqref{W:def} and \eqref{def-delta}.
\end{coro}
\begin{proof}
The result of the support of $X_t$  can be deduced from  Lemma \ref{s2lem1} and the complete proof can be found in \cite{Chemin}. The estimate \eqref{div:1} follows by applying Proposition \ref{prop:transport:es} to the equation 
$$
\big(\partial_t +v\cdot\nabla\big)\Div X_{t}=0.
$$ 
 For the last estimate \eqref{38-E}, we shall  admit the following assertion and we refer to \cite[Chapter 9]{Chemin} and \cite[Lemma 12]{Hmidi-0} for more details. For all $t\in [0,T]$, we have that
 \begin{equation}\label{decompo:DXV}
  \partial_{X_{t}} v(t)= g_1(t)+g_2(t),
 \end{equation}
  with
\begin{eqnarray}\label{decompo:DXV:es}
\| g_1(t)\|_{B^{\sigma_t}_{p,\infty}} &\leqslant & C\Vert \partial_{X_{t}} \omega(t)\Vert_{B^{\sigma_t -1}_{p,\infty}}+C\Vert\textnormal{div}X_{t}\Vert_{B^{\sigma_t }_{p,\infty}}\Vert\omega(t)\Vert_{L^\infty}
\end{eqnarray}
and
\begin{eqnarray*}
\| g_2(t)\|_{B^{\sigma_t}_{p,\infty}} &\leqslant & -C\log h\; W(t)\Vert X_{t}\Vert_{B^{\sigma_t }_{p,\infty}}.
\end{eqnarray*}
Then, applying once again Proposition \ref{prop:transport:es} to the equation \eqref{31}, implies
\begin{eqnarray*}
\Vert X_{t}\Vert_{B^{\sigma_t }_{p,\infty}}\lesssim \Vert X_{0}\Vert_{B^{s}_{p,\infty}}h^{-C\int_0^t W(\tau)d\tau}&+&\int_0^t\Vert\Div X_{\tau}\Vert_{B^{\sigma_\tau }_{p,\infty}}\Vert\omega(\tau)\Vert_{L^\infty} h^{-C\int_\tau^t W(\tau')d\tau'}d\tau\\ & +&\int_0^t\Vert \partial_{X_{\tau}}\omega(\tau)\Vert_{B^{\sigma_t -1}_{p,\infty}}  h^{-C\int_{\tau}^t W(\tau')d\tau'}d\tau.
\end{eqnarray*}
According to the definition of $W(t)$, we can write
$$
\int_0^t\Vert\omega(\tau)\Vert_{L^\infty}d\tau \leqslant h^{-C\int_0^t W(\tau)d\tau},
$$
together with \eqref{div:1} and the embedding $B^{s}_{p,\infty}\hookrightarrow B^{\sigma_t }_{p,\infty}$ for all $t\in [0,T]$, 
 yield to \eqref{38-E}. Therefore, Corollary \ref{coro-transport:es} is  proved.
\end{proof}

\begin{remark}\label{RMK:GR:es}
A localized version of \eqref{38-E} is given by 
\begin{eqnarray} \label{X-B-sigma-0:XXX}
\|\Delta_{j}X_t \|_{L^p}&\leqslant & \|\Delta_{j}X_0\|_{L^p}+C\int_0^t 2^{-j\sigma_\tau} \|\textnormal{div} X_\tau\|_{B^{\sigma_\tau}_{p,\infty}}  W (\tau)d\tau  +  C \int_0^t 2^{-j\sigma_\tau}\| \partial_{X_{\tau}}\omega\|_{B^{\sigma_\tau-1}_{p,\infty}} d\tau \nonumber\\&&+C(-\log h)\int_0^t 2^{-j\sigma_\tau}\|X_{\tau}\|_{B^{\sigma_{\tau}}_{p,\infty}}W(\tau)d\tau
\end{eqnarray}  
which comes from the decomposition \eqref{decompo:DXV} and the estimate \eqref{decompo:DXV:es}. Estimate \eqref{X-B-sigma-0:XXX} will be used later  in the proof of Proposition \ref{prop2222} related to the viscous case.
\end{remark}
We conclude this part by the following proposition which aims to study the dynamic of an admissible family of vector fields away from the singular set. For the convenience of the reader, we also recall the details of the proof from \cite{Chemin,Hassainia-Hmidi}. 
\begin{proposition}\label{prop:admissible family}
  Let $v$ be a smooth divergence-free vector field and let $\mathcal{X}_0=\big(X_{0,\lambda,h}\big)_{\lambda\in\Lambda,h\in(0,e^{-1}]}$ be a  $\Sigma_0-$admissible family of vector fields. We define the  time dependent  vector fields   $ \big(X_{t,\lambda,h}\big)_{\lambda\in\Lambda,h\in(0,e^{-1}]}$   by 
$$
  X_{t,\lambda,h}(x) \triangleq \big(\partial_{X_{0,\lambda,h}} \Psi\big) \big( t,\Psi^{-1}(t,x)\big), \quad \forall\, t\geqslant 0, \quad\forall \,  x\in\RR^2,
$$  
where $\Psi$ is the flow associated with the vector field $v$. Then we have 
\begin{equation} \label{ixtd}
I\big((\Sigma_t)_{\delta_t^{-1}(h)},(\mathcal{X}_t)_h\big)
 \geqslant  I\big((\Sigma_0)_{h},(\mathcal{X}_0)_{h}\big)h^{\int_0^t W(\tau)d\tau},
\end{equation}
where 
\begin{equation}\label{delta:inverse}
 \delta_t^{-1} (h)\triangleq h^{\exp(-\int_0^t\Vert v(\tau)\Vert_{LL}d\tau)}
\end{equation}
and $W(t)$ is given by \eqref{W:def} and $I$ is introduced in Definition \ref{Defintion-2.4}.
\end{proposition}
\begin{proof}
From the definition of $X_{t,\lambda,h} $, one gets
$$
Y_{t,\lambda,h}(x)\triangleq X_{t,\lambda,h}\big(\Psi(t,x)\big)=\partial_{X_{0,\lambda,h}} \Psi(t,x),
$$
and clearly we have
$$
\partial_t Y_{t,\lambda,h}(x)=\{\nabla v(t,\Psi(t,x))\}\cdot Y_{t,\lambda,h}(x).
$$
Fix $t>0$ and set for $\tau\in [0,t]$, \, $Z(\tau,x)=Y_{t-\tau,\lambda,h}(x)$, then 
$$
\partial_\tau Z(\tau,x)=-\{\nabla v(t-\tau,\Psi(t-\tau,x))\}\cdot Z(\tau,x).
$$
Gronwall lemma provides us
\begin{eqnarray*}
|Z(t,x)|&\leqslant & |Z(0,x)|\, e^{\int_0^t|\nabla v(t-\tau,\Psi(t-\tau,x))|d\tau}\\
&\leqslant &|Z(0,x)|\, e^{\int_0^t|\nabla v(\tau,\Psi(\tau,x))|d\tau},
\end{eqnarray*}
which is equivalent to
$$
|Y_{0,\lambda,h}(x)|\leqslant |Y_{t,\lambda,h}(x)|e^{\int_0^t|\nabla v(\tau,\Psi(\tau,x))|d\tau}.
$$
This gives in turn, 
\begin{equation*}
\vert X_{0,\lambda,h}(\Psi^{-1}(t,x))\vert \leqslant \vert X_{t,\lambda,h}(x)\vert\, e^{\int_0^t\vert \nabla v(\tau,\Psi(\tau,\Psi^{-1}(t,x)))\vert d\tau}.
\end{equation*}
 Hence, we infer that
\begin{eqnarray}\label{ii}
\inf_{x\in(\Sigma_t)_{\delta_t^{-1}(h)}^{c}}\sup_{\lambda\in\Lambda}\vert X_{0,\lambda,h}(\Psi^{-1}(t,x))\vert &\leqslant &\inf_{x\in(\Sigma_t)_{\delta_t^{-1}(h)}^{c}}\sup_{\lambda\in\Lambda}\vert X_{t,\lambda,h}(x)\vert\\ &\times \nonumber&\exp\Big(\int_0^t\big\Vert \nabla v\big(\tau,\Psi(\tau,\Psi^{-1}(t,\cdot))\big)\big\Vert_{L^\infty((\Sigma_t)_{\delta_t^{-1}(h)}^c)} d\tau\Big),
\end{eqnarray}
where we recall that  $\Sigma_t=\Psi(t, \Sigma_0)$ and $(\Sigma_t)^c_{\eta}=\big\{x\in \RR^2;\, d\big(x,\Sigma_t\big)\geqslant  \eta\big\}$. According  to the second point in Lemma  \ref{s2lem1} one deduces that
$$\Psi^{-1}\big(t,\big(\Sigma_t\big)_{\delta_t^{-1}(h)}^c\big)\subset \big(\Sigma_0\big)_{\delta_t(\delta_t^{-1}(h))}^c=(\Sigma_0)_{h}^c.
$$
Then, we immediately find in accordance with \eqref{ii}
\begin{eqnarray}\label{lipsi}
\inf_{x\in(\Sigma_t)_{\delta_t^{-1}(h)}^c}\sup_{\lambda\in\Lambda}\vert X_{0,\lambda,h}(\Psi^{-1}(t,x))\vert &= &\inf_{y\in\Psi^{-1}(t,(\Sigma_t)_{\delta_t^{-1}(h)}^c)}\sup_{\lambda\in\Lambda}\vert X_{0,\lambda,h}(y)\vert\notag\\ 
&\geq& \inf_{y\in(\Sigma_0)_{h}^c}\sup_{\lambda\in\Lambda}\vert X_{0,\lambda,h}(y)\vert\notag\\ 
&\geqslant  & I\big((\Sigma_0)_{h},(\mathcal{X}_0)_{h}\big).
 \end{eqnarray}
Moreover, in view of the same lemma we have
$$
\Psi\Big(\tau,\Psi^{-1}\big(t,(\Sigma_t)_{\delta_t^{-1}(h)}^c\big)\Big)\subset\big(\Sigma_\tau \big)_{\delta_{\tau,t}(\delta_t^{-1}(h))}^c=\big(\Sigma_\tau\big)_{\delta_\tau^{-1}(h)}^c\subset\big(\Sigma_\tau\big)_{h}^c.
$$
Consequently,  we may write
\begin{eqnarray*}
\big\Vert \nabla v\big(\tau,\Psi(\tau,\Psi^{-1}(t,\cdot))\big)\big\Vert_{L^\infty\big(({\Sigma_t})_{\delta_t^{-1}(h)}^c\big)}&\leqslant & \Vert \nabla v(\tau)\Vert_{L^\infty((\Sigma_\tau)_{h}^c)}\notag\\ &\leq& -(\log h)\Vert \nabla v(\tau)\Vert_{L(\Sigma_\tau)}\notag\\ &\leq&  -(\log h)W(\tau).
\end{eqnarray*}
Combining the last estimate with \eqref{ii} and\eqref{lipsi} yields 
\begin{equation*}
I\big((\Sigma_t)_{\delta_t^{-1}(h)},(\mathcal{X}_t)_h\big)
 \geqslant   I\big((\Sigma_0)_{h},(\mathcal{X}_0)_{h}\big)h^{\int_0^t W(\tau)d\tau}.
\end{equation*}
 This  ends the proof of the estimate \eqref{ixtd}.
\end{proof}

\subsection{Maximal smoothing effects with rough drift} 
In this paragraph, we shall explore some a priori estimates for transport-diffusion equations governed by a $\log-$Lipschitz velocity. They are considered as one of the cornerstone tools in tracking the co-normal regularity of the vorticity and the density with respect to a judicious transported vector-fields, see in particular Proposition \ref{prop22} and Proposition \ref{prop2222}. We first start with recalling in Proposition \ref{p11} below some important results related to the regularity persistence in the transport-diffusion equation as well as the maximal smoothing effects. Then, in Proposition \ref{Max-reg} we shall provide specific smoothing effects with a source term, extending a previous result established in \cite{Hmidi-0}. We make a slight adaptation of the same method using Lagrangian coordinates and para-differential calculus. Finally, in Corollary \ref{cor-Max-reg} we formulate our main result of this section related to the maximal smoothing effects with a frequency logarithmic loss. This apparent loss seems to be unavoidable due to the lack of regularity of the velocity field which scales slightly  below the Lipschitz class. For a detailed proof, we refer to \cite{Chemin,Hmidi-0}.

\begin{proposition}\label{p11}
Let $(s ,r,p,a)\in(-1,1) \times[1,\infty]\times(1,\infty)\times(1,\infty)$ and $v$ be a smooth divergence-free vector field. Then, with the  notations \eqref{def-delta}, \eqref{Sigma-def} and \eqref{W:def}, the following assertions hold true. 
Let $f\in L^\infty \big([0,T],B^{s}_{p,\infty}\big)$ be a solution of the transport-diffusion model 
\begin{equation}\label{TD-kappa}
\left\{ \begin{array}{ll}
\partial_{t}f+v\cdot\nabla f-\kappa\Delta f ={\bf F}+\kappa {\bf G}, & \vspace{2,5mm}\\
f_{| t=0}=f_{0}.\tag{TD$_{\kappa}$}
\end{array} \right.
\end{equation} 
Assume that $\supp f(t)\subset (\Sigma_t)^{c}_{\delta_t(h)}$ with $h\in(0,e^{-1}]$. Then, for every $t\in[0,T]$, we have
\begin{eqnarray*}
 \|f(t)\|_{B^{s}_{p,\infty}}\lesssim C_p h^{{-C_s}\int_{0}^{t}A(\tau)d\tau}\Big(\|f_0\|_{B^{s}_{p,\infty}}+\big(\kappa^{-\frac{1}{r'}}+t^{\frac{1}{r'}}\big)\|{ \bf F}\|_{\widetilde{L}^{r}_{t}B^{s-\frac{2}{r'}}_{p,\infty}}+\big(1+\kappa t\big)\|{  {\bf G}}\|_{\widetilde{L}^{\infty}_{t}B^{s-2}_{p,\infty}}\Big), 
\end{eqnarray*}
where, $\delta_t(h)$ is defined in \eqref{def-delta} and  
$$
A(t)=\| v(t)\|_{L  (\Sigma_t)}e^{\int_{0}^{t}\|v(\tau)\|_{LL} d\tau}, 
 $$
$$C_p = \frac{p^2}{p-1}, \quad C_s = \frac{1}{1-s^2}, \quad r'= \frac{r}{r-1}. $$
 
More precisely, we have for all $q\geqslant -1$
\begin{eqnarray}\label{ES:before Gronwal}
 \|\Delta_q f(t)\|_{L^p}&\lesssim & C_p\|\Delta_q f_0\|_{L^p}+C_p\big(\kappa^{-\frac{1}{r'}}+t^{\frac{1}{r'}}\big)2^{-\frac{2q}{r'}}\|{\Delta_q \bf F}\|_{L^{r}_{t}L^p}+C_p\big(1+\kappa t\big)2^{-2q}\|{\Delta_q  {\bf G}}\|_{L^{\infty}_{t}L^p} \nonumber\\ && -C_p\log h \int_0^t C_s2^{-qs} \| f(\tau) \|_{B^{s}_{p,\infty}} A(\tau) d\tau.
\end{eqnarray}
\end{proposition}
\begin{remark}\label{remark:p<infini}  
Proposition \ref{p11} requires  $p<\infty$ because of the constant $C_p$ which is unbounded as  $p $ tends to infinity. This technical issue is relevant to the viscous case which  forces us to study the vortex patch problem in Besov spaces with integrability index $p<\infty$ contrary to the  inviscid case given by Proposition \ref{prop:transport:es} where we have no obstruction on the index $p$. 
\end{remark}
As already mentioned, we state the pivotal result of this part which focuses on another smoothing effect type for a transport-diffusion equation with a source term. In contrast with Proposition \ref{p11}, the  proposition stated below  can be extended to  the borderline case $p=\infty$ using the techniques of Lagrangian coordinates in the spirit of the work \cite{Hmidi-0} developed for the 2d Navier-Stokes equations. More precisely, we will prove.\begin{proposition}\label{Max-reg}
Let  $p,m\in[1,\infty], r\in  [m, \infty] $ and $T>0$ . Consider  a divergence-free vector field $v$  belonging to $L^1([0,T]; LL ), f_0\in L^p(\RR^2)$ and $  \mathtt{F}\in L^m([0,T];L^p)$. Let  $f$ be a solution to   
\begin{equation}\label{TD-max}
\left\{\begin{array}{ll}
\partial_t f+v\cdot\nabla f-\kappa\Delta f=  \mathtt{F} &\vspace{2,5mm}\\
f_{|t=0}=f_0
\end{array}
\right.
\end{equation}
that  satisfies
\begin{equation}\label{Cond-f}
\|f(t)\|_{L^p}\leqslant G(t),\quad\forall t\in[0,T],
\end{equation}
for some increasing    function  $G:[0,T]\to[0,\infty)$.
Then, for any $t\in[0,T]$ and for any $ q\geqslant -1$ 
\begin{equation*}
(\kappa2^{2q})^{\frac{1}{r}}\Vert \Delta_q f \Vert_{L_t^r L^p}\lesssim  G(t) \Big(1+\kappa t+(q+2)\int_0^t\Vert v(\tau)\Vert_{LL}d\tau \Big)^{\frac{1}{r}}+(\kappa 2^{2q})^{ \frac{1}{m}-1}    \|{\Delta_q\mathtt{F}}  \|_{L^m_tL^p} .
\end{equation*}
\end{proposition}
\begin{proof}
We localize $f$ by applying the operator $\Delta_q$ to \eqref{TD-max}, then by setting $f_q=\Delta_q f$ and $\mathtt{F}_q=\Delta_q \mathtt{F}$ one gets
\begin{equation}\label{f-q}
\partial_t f_q+S_{q-1}v\cdot\nabla f_q-\kappa\Delta f_q=S_{q-1}v\cdot\nabla f_q-\Delta_q(v\cdot\nabla f)+\mathtt{F}_q\triangleq  {\bf R}_q+\mathtt{F}_q.
\end{equation}
Let us denote by $\Psi_q(t)$ the flow associated to the regularized velocity $S_{q-1}v$, which satisfies an integral equation of type
$$
\Psi_q(t,x)=x+\int_0^tS_{q-1}v(\tau,\Psi_q(\tau,x))d\tau.
$$
Thanks to the Proposition \ref{C-LL}, we write
$$
\int_0^t\Vert \nabla S_{q-1}v(\tau)\Vert_{L^{\infty}}d\tau \leqslant C(q+2)\int_0^t\Vert v(\tau) \Vert_{LL}d\tau \triangleq V_q(t).
$$
Thus, we can deduce from straightforward computations that  the flow  $\Psi^1_q\triangleq\Psi_q$ and its inverse $\Psi^{-1}_q$ satisfy 
\begin{equation}\label{Pr-flow}
\Vert\nabla\Psi_q^{\pm1}(t)\Vert_{L^{\infty}}\leqslant e^{V_q(t)}\quad \textnormal{and}\quad  \Vert\nabla^2\Psi_q^{\pm1}(t)\Vert_{L^{\infty}}\leqslant C2^qV_q(t)e^{V_q(t)},
\end{equation}
Let us consider the functions
$$ 
\bar{f}_q(t,x) =f_q(t,\Psi_q(t,x)), \quad \bar{\mathtt{F}}_q(t,x) =\mathtt{F}_q(t,\Psi_q(t,x))\;\;\textnormal{and}\;\; \bar{\bf R}_q={\bf R}_q(t,\Psi_q(t,x)).   
$$
Then from direct computations  we infer
\begin{eqnarray}\label{aq-1}
\Delta f_q(t)\circ \Psi_q(t,x)&=&\nabla\bar{f}(t,x)\cdot (\Delta\Psi_q^{-1})(t,\Psi_q(t,x))\\
\nonumber&&+
\sum_{i=1}^d \Big\langle \nabla^2\bar{f}(t,x)\cdot(\partial_i\Psi_q^{-1})(t,\Psi_q(t,x)),(\partial_i\Psi_q^{-1})(t,\Psi_q(t,x))\big\rangle.
\end{eqnarray}
In addition, differentiating the  integral equation of $\Psi_q^{-1}$ leads after straightforward computations to the following result
\begin{equation}\label{aq-2}
(\partial_i\Psi_q^{-1})(t,\Psi_q(t,x))=e_i+g_q^i(t,x),
\end{equation}
where the functions  $g^i_q$ satisfy
\begin{equation*}
\Vert g_q^i \Vert_{L^{\infty}}\leqslant V_q(t)e^{V_q(t)}\triangleq g_q(t).
\end{equation*}
Collecting \eqref{aq-1} and \eqref{aq-2} allows to check that $\bar{f}_q$ satisfies the following parabolic equation 
\begin{eqnarray*}
\big(\partial_t-\kappa\Delta\big)\bar{f}_q(t,x)&=&\kappa \sum_{i=1}^d\big\langle\nabla^2\bar{f}_q(t,x)g^i_q(t,x),g^i_q(t,x)\big\rangle+2\kappa \sum_{i=1}^d\big\langle\nabla^2\bar{f}_q(t,x)e_i,g^i_q(t,x)\big\rangle\\&&+\kappa\nabla \bar{f}_q(t,x)\cdot(\Delta\Psi_q^{-1})(t,\Psi_q(t,x))
+ \bar{\bf R}_q +\bar{\mathtt{F}}_q(t,x).
\end{eqnarray*}
One of the main difficulties that we should face is that the frequency spectrum of $\bar{f}_q$ is not necessarily localized in a compact set. To deal with this, we shall first localize it through the operators $\Delta_j$, with $j\in \NN$. Hence, classical smoothing effects (see for instance \cite[Lemma 2.4]{Bahouri-Chemin-Danchin}) of the heat semigroup allow to get
\begin{align*}
\Vert \Delta_j\bar{f}_q(t)\|_{L^p} & \leqslant Ce^{-c\kappa t2^{2j}}\Vert \Delta_j f_q(0) \Vert_{L^p}+C\int_0^t e^{-c\kappa(t-\tau)2^{2j}}\Vert \Delta_j \bar{\bf R}_q(\tau)\Vert_{L^p}d\tau\\
& +C\kappa(g_q+g_q^2(t))\int_0^t e^{-c\kappa(t-\tau)2^{2j}}\Vert \nabla^2\bar{f}_q(\tau)\Vert_{L^p}d\tau\\
& + C\kappa2^qg_q(t)\int_0^t e^{-c\kappa(t-\tau)2^{2j}}\Vert \nabla\bar{f}_q(\tau)\Vert_{L^p}d\tau+\int_0^t e^{-c\kappa(t-\tau)2^{2j}}\Vert \Delta_j\bar{\mathtt{F}}_q(\tau)\Vert_{L^p}d\tau.
\end{align*}
To bound the second term of the r.h.s. of the above estimate, Bernstein's inequality helps to write 
$$ 
\|\Delta_j \bar{\bf R}_q(t)\|_{L^p}\leqslant C2^{-j}\|\nabla({\bf R}_q\circ\Psi_q(t))\|_{L^p}\leqslant C 2^{q-j}\|\nabla\Psi_q(t)\|_{L^\infty}\|{\bf R}_q(t)\|_{L^p}. 
$$
Consequently, the combination of the last two estimates, the condition \eqref{Cond-f} and  Lemma \ref{Tec-lem-1} applied with  $s=0$
enable to have 
\begin{eqnarray*} 
\\
\nonumber\|\Delta_j \bar{f}_{q}(t)\|_{L^p}&\leqslant & Ce^{-c\kappa t 2^{2j}}\|\Delta_j \bar{f}_{q}(0)\|_{L^p}+(q+2)2^{q-j}\int_{0}^{t}e^{-c\kappa (t-\tau) 2^{2j}}\|v(\tau)\|_{LL}e^{V_q(\tau)}\| f(\tau)\|_{ B^0_{p,\infty}}d\tau\\\nonumber
&&+C\kappa\big(g_{q}(t)+g_{q}^{2}(t)\big)\int_{0}^{t}e^{-c\kappa (t-\tau) 2^{2j}}\|\nabla^2 \bar{f}_q(\tau)\|_{L^p}d\tau\\\nonumber
&&+C\kappa 2^q g_q(t)\int_{0}^{t}e^{-c\kappa (t-\tau) 2^{2j}}\|\nabla \bar{f}_q(\tau)\|_{L^p}d\tau+C\int_{0}^{t}e^{-c\kappa (t-\tau) 2^{2j}}\| \Delta_j \bar{\mathtt{F}}_q(\tau)\|_{L^p}d\tau.\nonumber
\end{eqnarray*}
Now, owing to the fact 
$$ \| f(\tau)\|_{ B^0_{p,\infty}} \leqslant \| f(\tau)\|_{L^p} \leqslant G(t),\quad \forall\tau\leqslant  t,$$
we obtain 
\begin{align}\label{Delta-j-bar-a}
\nonumber\|\Delta_j \bar{f}_{q}(t)\|_{L^p}&\leqslant Ce^{-c\kappa t 2^{2j}}\|\Delta_j \bar{f}_{q}(0)\|_{L^p}+(q+2)2^{q-j} G(t)\int_{0}^{t}e^{-c\kappa (t-\tau) 2^{2j}}\|v(\tau)\|_{LL}e^{V_q(\tau)} d\tau\\
\nonumber
&+C\kappa\big(g_{q}(t)+g_{q}^{2}(t)\big)\int_{0}^{t}e^{-c\kappa (t-\tau) 2^{2j}}\|\nabla^2 \bar{f}_q(\tau)\|_{L^p}d\tau\\
&+C\kappa 2^q g_q(t)\int_{0}^{t}e^{-c\kappa (t-\tau) 2^{2j}}\|\nabla \bar{f}_q(\tau)\|_{L^p}d\tau+C\int_{0}^{t}e^{-c\kappa (t-\tau) 2^{2j}}\| \Delta_j \bar{\mathtt{F}}_q(\tau)\|_{L^p}d\tau.
\end{align}
Taking the $L^{r}_{t}-$norm in  both sides of \eqref{Delta-j-bar-a} and employing Young's inequality we obtain, due to \eqref{Cond-f} and the time monotonicity of $g_q$ 
\begin{align}\label{Delta-j-bar-a-1}
\nonumber\|\Delta_j \bar{f}_{q}(t)\|_{L^{r}_{t}L^p}\leqslant & C(\kappa 2^{2j})^{-\frac{1}{r}}\|\Delta_j \bar{f}_{q}(0)\|_{L^p}+C2^{q-j}G(t)(\kappa 2^{2j})^{-\frac{1}{r}}\big(e^{V_q(t)}-1\big)\\
\nonumber&+C\big(g_{q}(t)+g_{q}^{2}(t)\big)2^{-2j}\|\nabla^2 \bar{f}_q\|_{L^{r}_{t}L^p}+Cg_{q}(t)2^{q}2^{-2j}\|\nabla \bar{f}_q\|_{L^{r}_t L^p}\\
&+C(\kappa2^{2j})^{-\frac1r- \frac{1}{m'}}\|\Delta_j\bar{\mathtt{F}}_q\|_{L^m_t L^p},
\end{align}
for all $m\in [1,r]$. To estimate  $\|\Delta_j\bar{\mathtt{F}}_q\|_{L^m_t L^p}$ we use  the boundedness of $\Delta_j$  together  with  the fact that $\Psi_q$ preserves Lebesgue measure, 
\begin{equation*}
\|\Delta_j \bar{\mathtt{F}}_q(t)\|_{L^p}\leqslant \|  \bar{\mathtt{F}}_q(t)\|_{L^p} \leqslant \| \mathtt{F} _q(t)\|_{L^p}.
\end{equation*}
By making use of \eqref{Pr-flow} and applying once again  the invariance of Lebesgue measure by the flow   $\Psi_q$, one deduces
\begin{equation}\label{Pr-1}
\|\nabla\bar{f}_q(t)\|_{L^p}\leqslant C2^{q}e^{V_q(t)}\|f_q(t)\|_{L^p}
\end{equation}
and
\begin{equation}\label{Pr-2}
\|\nabla^2\bar{f}_q(t)\|_{L^p}\leqslant C2^{2q}e^{V_q(t)}\|f_q(t)\|_{L^p}.
\end{equation}
Hence, plugging these estimates into \eqref{Delta-j-bar-a-1} and summing over $j\geqslant  q-N_0$, we end up with 
\begin{align}\label{Delta-j-bar-a-2}
 (\kappa 2^{2q})^{\frac{1}{r}}\sum_{j\geqslant q-N_0}\|\Delta_{j}\bar{f}_q\|_{L^{r}_{t}L^p}\leqslant & 2^\frac{2N_0}{r}\|f_q(0)\|_{L^p}+2^{2N_0}g_q(t) (\kappa 2^{2q})^{\frac1r}\|f_q\|_{L^r_t L^p}\\
\nonumber&+C2^{N_0(1+\frac2r)}G(t)\big(e^{V_q(t)}-1\big)+2^{2N_0( \frac{1}{r}+\frac{1}{m^\prime})  }(\kappa 2^{2q})^{ -\frac{1}{m'}} \|{\mathtt{F}}_q \|_{L^m_tL^p}.
\end{align}   Concerning the lower frequencies, we exploit Lemma \ref{Tec-lem-V} in order to get 
\begin{equation}\label{sum-1}
\sum_{j\leqslant q-N_0}\|\Delta_j \bar{f}_{q}\|_{L^r_t L^p}\leqslant C 2^{-N_0}e^{V_q(t)}\|f_q\|_{L^r_t L^p}.
\end{equation}
Putting together \eqref{Delta-j-bar-a-2},  \eqref{sum-1} and  the estimate
$$\|f_q(0)\|_{L^p} \leqslant \|f \|_{L^\infty_t L^p}\leqslant G(t), $$
yields in view of $ \frac{1}{r}+\frac{1}{m^\prime} \leqslant 1$
\begin{eqnarray}\label{Delta-j-bar-a-3}
(\kappa 2^{2q})^{\frac{1}{r}} \|  f _q\|_{L^{r}_{t}L^p}&\leqslant 2^{{2N_0}}G(t)&+2^{2N_0}g_q(t) (\kappa 2^{2q})^{\frac1r}\|f_q\|_{L^r_t L^p}\\
&&+C2^{-N_0}e^{V_q(t)}(\kappa 2^{2q})^{\frac1r}\|f_q\|_{L^r_t L^p}\nonumber\\
&&+C2^{3N_0}G(t)\big(e^{V_q(t)}-1\big)\nonumber\\&&+2^{ {2N_0}  }(\kappa 2^{2q})^{ -\frac{1}{m'}} \|{\mathtt{F}}_q \|_{L^m_tL^p}.\nonumber
\end{eqnarray}  
Accordingly, if we impose to $t$ and $N_0$ the following two conditions
\begin{equation}\label{Vis-C-1}
2^{2N_0}g_q(t)\leqslant\frac14\quad\mbox{and}\quad C2^{-N_0}e^{V_q(t)}\leqslant\frac14,
\end{equation}
then we will have for every $q\geqslant  N_0$
\begin{equation}\label{Vis-R-1}
(\kappa 2^{2q})^{\frac{1}{r}}\|f_q\|_{L^{r}_{t}L^p}\leqslant C2^{4N_0} \big(G(t) + (\kappa 2^{2q})^{ -\frac{1}{m'}} \|{\mathtt{F}}_q \|_{L^m_tL^p}\big).
\end{equation}
In fact, the conditions given \eqref{Vis-C-1} are satisfied as soos as one takes
\begin{equation}\label{small-t}
(q+2)\int_{0}^{t}\|v(\tau)\|_{LL}d\tau\leqslant C_0.
\end{equation}
The constant $C_0$ is absolute, and, a fortiori $N_0$, whose choice does not depend on $q$. To be clear, we first take $V_q(t)\leqslant 1$ and we select  $N_0$ so that $2^{-N_0}\leqslant \frac{1}{4eC}$, then we to get the  first condition in \eqref{Vis-C-1} it suffices to take  $V_q(t)$ small enough such that  $C2^{2N_0}g_q(t)\leqslant\frac14$. This latter point  is feasible because the function $xe^{x}$ tends to zero when  $x$ approaches  zero. Finally, we establish the existence of a fixed  integer $N_0$ and a small constant $C_0$  such that, for any $t$ satisfying  \eqref{small-t} and for any $q\geqslant  N_0$ we have   
$$
(\kappa 2^{2q})^{\frac{1}{r}}\|{f}_q\|_{L^{r}_{t}L^p}\leqslant C\big(G(t) +(\kappa 2^{2q})^{ -\frac{1}{m'}} \|{\mathtt{F}}_q \|_{L^m_tL^p}\big).
$$
In what follows, we will remove the smallness condition \eqref{small-t} and  extend the above estimate for  any arbitrary time $t$. For this aim, we take any integer $q\geqslant  N_0$ and  split the time interval  $[0,t]$ in the following way
$$
 0 =t_0<t_1<\ldots<  t_N=t , \quad (q+2)\int_{t_i}^{t_{i+1}}\|v(\tau)\|_{LL}d\tau\approx C_0, \quad\forall i\in\{1,2,.., N-1  \}.
$$
Notice that by summing up the last condition, we find the estimate
$$
N\leqslant  C\Big(1+(q+2)\int_{0}^{t}\|v(\tau)\|_{LL}d\tau\Big),
$$
Now, reproducing the previous analysis in the interval $[t_i,t_{i+1}]$ with the suitable adaptations in particular for the flow estimates, we succeed to establish 
$$
(\kappa 2^{2q})^{\frac{1}{r}}\|{f}_q\|_{L^{r}([t_i,t_{i+1}],L^p)}\leqslant C \big(G( t_{i} ) +(\kappa 2^{2q})^{ -\frac{1}{m'}}   \|{\mathtt{F}}_q \|_{L^m([t_i,t_{i+1}];L^p)}\big).
$$
Therefore, since $G$ is increasing we get for any  $q\geqslant  N_0$
\begin{equation*} 
(\kappa 2^{2q})\|{f}_q\|_{L^{r}_t L^p}^r\leqslant C^rN \left(G(t)\right)^r + C^r (\kappa 2^{2q})^{ -\frac{r}{m'}}  \sum_{i=0}^{N-1} \|{\mathtt{F}}_q \|_{L^m([t_i,t_{i+1}];L^p)}^r.
\end{equation*}
Remind that  $m\leqslant r$ and then by  the embedding $\ell^1 \hookrightarrow \ell^{\frac{r}{m}} $ we infer 
\begin{equation}\label{F-1}
(\kappa 2^{2q})\|{f}_q\|_{L^{r}_t L^p}^r\leqslant C^rN \left(G(t)\right)^r+ C^r (\kappa 2^{2q})^{ -\frac{r}{m'}}    \|{\mathtt{F}}_q \|_{L^m([0,t];L^p)}^r.
\end{equation}
Concerning the lower frequencies  $q\leqslant N_0$, we write
\begin{equation}\label{F-2}
(\kappa 2^{2q})^{\frac1r}\|{f}_q\|_{L^{r}_t L^p}\leqslant C(\kappa t)^{\frac1r}G(t).
\end{equation}
Finally, combining \eqref{F-1} and \eqref{F-2} we reach the desired estimate. 
\end{proof}
An important application  of the above proposition is the establishment of maximal smoothing effects for transport-diffusion equations drifted by a  log-Lipschitz vector field.  Consider a  solution of the equation.
\begin{equation}\label{TD-Equa}
\left\{\begin{array}{ll}
\partial_t\rho + v\cdot \nabla \rho -\kappa \Delta \rho = 0,\vspace{2,5mm}\\
\rho_{|_{t=0}}=\rho_0.
\end{array}\right.
\end{equation}
then the following result holds.
\begin{coro}\label{cor-Max-reg} Let $T>0$ and  $v$ be a divergence-free vector field belonging to $L^1([0,T]; LL )$ whose vorticity satisfies  $\omega\in L^1([0,T]; L^\infty)$. If $\rho$ is a solution of \eqref{TD-Equa}   such that
\begin{equation}\label{assumption:000}
 \|\nabla\rho(t)\|_{L^\infty}\leqslant  G(t),\quad \forall t\in[0,T],  
\end{equation}
for some increasing   function $G$, then,  for any  $q\geqslant -1$, $i\in \{1,2\}$ and $ t\in[0,T]$, there holds that
\begin{eqnarray}\label{equa:corollary}
 \kappa2^{2q} \Vert \Delta_q (\partial_i\rho) \Vert_{L_t^1 L^\infty}\lesssim G(t)  \Big(1+\kappa t   + (q+2)\int_0^t\Vert v(\tau)\Vert_{LL}d\tau \Big)     .
\end{eqnarray}
\end{coro}
\begin{proof}Applying  the partial derivative $\partial_i$ to \eqref{TD-Equa} yields
 \begin{equation}\label{TD-max-1}
\partial_t \partial_i\rho+v\cdot\nabla\partial_i\rho-\kappa\Delta\partial_i\rho=\partial_i v\cdot\nabla\rho.
\end{equation}
Hence, by applying Proposition \ref{Max-reg} with $f\equiv\partial_i\rho$ and  $\mathtt{F}\equiv\partial_i v\cdot\nabla\rho$ yields
\begin{equation}\label{equa:corollary:00}
 \kappa2^{2q} \Vert \Delta_q (\partial_i\rho) \Vert_{L_t^1 L^\infty}\lesssim G(t) \Big(1+\kappa t+(q+2)\int_0^t\Vert v(\tau)\Vert_{LL}d\tau \Big) +\|\Delta_q(\partial_i v\cdot\nabla\rho)\|_{L^1_t L^\infty}.
\end{equation}
It remains to estimate the term $\|\Delta_q(\partial_i v\cdot\nabla\rho)\|_{L^1_t L^\infty}$. By virtue of Bony's decomposition, we write
\begin{eqnarray*}
\|\Delta_q(\partial_i v\cdot\nabla\rho)\|_{L^1_t L^\infty}&\leqslant&\sum_{|j-q|\le4}\|\Delta_q\big(S_{j-1}\partial_i v\cdot\Delta_j \nabla\rho)\|_{L^1_t L^\infty}+\sum_{|j-q|\le4}\|\Delta_q\big(\Delta_j \partial_i v\cdot S_{j-1}\nabla\rho)\|_{L^1_t L^\infty}\\
&&+\sum_{{j\geqslant  q-4}}\|\Delta_q\big(\widetilde{\Delta}_{j}\partial_i v\cdot\Delta_j\nabla\rho\big)\|_{L^1_t L^\infty}\\
&=&\textnormal{II}_1+\textnormal{II}_2+\textnormal{II}_3.
\end{eqnarray*}
We treat each term separately. For $\textnormal{II}_1$, Proposition \ref{C-LL} ensures that
\begin{eqnarray*}
\textnormal{II}_1&\leqslant&\sum_{|j-q|\le4}\|S_{j-1}\partial_i v\|_{L^1_t L^\infty}\|\Delta_j\nabla\rho \|_{L^\infty_t L^\infty} \\
&\leqslant&C(q+2)\|\nabla\rho \|_{L^\infty_t L^\infty} \int_{0}^{t}\|v(\tau)\|_{LL}d\tau.
\end{eqnarray*}
 To estimate  the term $\textnormal{II}_2$, we use again  Proposition \ref{C-LL} leading to
\begin{eqnarray*}
\textnormal{II}_2&\leqslant&\sum_{|j-q|\le4}\|S_{j-1}\nabla\rho\|_{L^\infty_t L^\infty} \|\Delta_j \nabla v \|_{L^1_t L^\infty} \\
&\leqslant &C(q+2)\|\nabla\rho \|_{L^\infty_t L^\infty} \int_{0}^{t}\|v(\tau)\|_{LL}d\tau.
\end{eqnarray*}
Finally, in order to estimate the remainder term $\textnormal{II}_3$, we need to use the divergence-free condition on $v$ to obtain 
\begin{eqnarray*}
\textnormal{II}_3&\leqslant & C 2^q\sum_{j\geqslant  q-4} \| \widetilde\Delta_j \nabla v\|_{L^1_t L^\infty}\|{\Delta}_{j}\rho \|_{L^\infty_t L^\infty} \\
&\leqslant &C\|\nabla\rho \|_{L^\infty_t L^\infty}\int_0^t \|v(\tau)\|_{LL}d\tau \sum_{j\geqslant  q-4} (j+2)2^{q-j}  \\
&\leqslant &C\|\nabla\rho \|_{L^\infty_t L^\infty}\int_0^t \|v(\tau)\|_{LL}d\tau \Big( (q+2)\sum_{j\geqslant  q-4}  2^{q-j} + \sum_{j\geqslant  q-4} (j-q )2^{q-j}  \Big).
\end{eqnarray*}
Now, using the fact that 
$$\sup_{q\in\NN}\sum_{j\geqslant  q-4}  2^{q-j} + \sum_{j\geqslant  q-4} (j-q )2^{q-j} <\infty $$
yields 
\begin{equation*}
\textnormal{II}_3 \leqslant C(q+2)\|\nabla\rho \|_{L^\infty_t L^\infty} \int_{0}^{t}\|v(\tau)\|_{LL}d\tau. 
\end{equation*}
Consequently, plugging the preceding estimates into  \eqref{equa:corollary:00} and using the assumption \eqref{assumption:000} achieves   the proof of Corollary \ref{cor-Max-reg}.   
\end{proof}
\section{Inviscid case}\label{Inviscid-case}
This section is devoted to the study of the inviscid Boussinesq system written in the vorticity-density reformulation as
\begin{equation} \label{eqn:omega}
\left\{ \begin{array}{ll}
  \partial_t \omega + v \cdot \nabla \omega = \partial_1\rho,\quad (t,x)\in \mathbb{R}_+\times \mathbb{R}^2,&\vspace{2mm}\\ 
   \partial_t\rho+v\cdot\nabla \rho=0,&\vspace{2mm}\\
      v= \nabla^\perp \Delta^{-1}\omega,&\vspace{2mm}\\
  (\omega,\rho)|_{t=0}=(\omega_0,\rho_0).
  \end{array}\right.
\end{equation}
We shall state a local well-posedness result for \eqref{eqn:omega} with general initial data covering the result of Theorem \ref{THEO:1:soft}. Thereby, we show first how we deduce the proof of Theorem \ref{THEO:1:soft} from the general statement of Theorem \ref{Th1:general:version} below. Afterward, we shall move to  the proof of the general version. We recall  that all the notation and definitions used in this section are borrowed from Section \ref{Useful-Tec}.

\subsection{Generalized singular patches}
We intend to describe a full statement on the well-posedness issue for \eqref{eqn:omega} with a general framework covering a large class containing in particular    singular patches. Our main result reads as follows.
\begin{theorem}\label{Th1:general:version}
Let $(s,m,a)\in(0,1)^2\times(1,\infty)$ and let $\Sigma_0$ be a compact negligible subset of the plane. Consider a divergence-free vector field  $v_0$ with  vorticity  $\omega_ 0$ in  $ L^{a}\cap L^\infty$ and let $\rho_0$ be a real-valued function in $W^{1,a}(\mathbb{R}^2)\cap W^{1,\infty}(\mathbb{R}^2)\cap C^{1+m}((\Sigma_0) _r) $,\footnote{We recall that the definition of an $h$-enlargement of a set $F$ in $\mathbb{R}^2$ is given by \eqref{distance-h:def}.} for some  $r>0$, such that $\nabla \rho_0$ vanishes on $\Sigma_0.$ Let  $\mathcal{X}_0=(X_{0,\lambda,h})_{(\lambda,h)\in \Lambda\times (0,e^{-1}]}$ be a  family of vector fields of class $C^s$ as well as their divergences and suppose that this family is  $\Sigma_0$-admissible of order $\Xi  = (\alpha,\beta,\gamma )$  such that
 \begin{equation*} 
\sup_{h\in (0,e^{-1}]} h^{\beta }\Vert\omega_0\Vert_{(\Sigma_0)_{h},(\mathcal{X}_ {0,h})}^{s}+ \sup_{h\in (0,e^{-1}]} h^{\beta }\Vert\rho_0\Vert_{(\Sigma_0)_{h},(\mathcal{X}_ {0,h})}^{s+1}<\infty.
\end{equation*}  
 Then, there exists $T>0$ such that the Boussinesq system \eqref{eqn:omega} admits a unique solution
 $$(\omega,\rho)\in L^\infty\big([0,T],L^a\cap L^\infty\big)\times L^\infty\big([0,T], W^{1,a}\cap W^{1,\infty}\big).
 $$  
 In addition, we have
$$
\sup_{h\in(0,e^{-1}]}\frac{\Vert \nabla v(t)\Vert_{L^\infty((\Sigma_t)_h^c)}}{-\log h}\in L^\infty([0,T]),$$
where $\Sigma_t=\Psi(t,\Sigma_0)$ and $\Psi$ is the flow associated to $v$. Moreover, there exists $C_{0,T}>0$, depending only on the initial data and $T$, and there exists a  triplet  $\Xi_1=(\alpha_1,\beta_1,\gamma_1)$ of non-negative numbers, such that 

\begin{equation} \label{T1:persistence of reg:1}
\sup_{h\in (0,e^{-1}]\atop t\in[0,T]} h^{\beta_1}\Vert\omega(t)\Vert_{(\Sigma_t)_{h^{\alpha_1}},(\mathcal{X}_{t,h})}^{s}+ \sup_{h\in (0,e^{-1}]} h^{\beta_1}\Vert\rho(t)\Vert_{(\Sigma_t)_{h^{\alpha_1}},(\mathcal{X}_{t,h}) }^{s+1}    \leqslant  C_{0,T},
\end{equation}  
 where the transported $\mathcal{X} _t$ of $\mathcal{X} _0$ by the flow $\Psi$, defined by 
$$
X_{t,\lambda,h}(\Psi (t,\cdot)) \triangleq \partial_{X_{0,\lambda,h}} \Psi(t, \cdot), 
$$
is $\Sigma_t$--admissible of order $\Xi_1$ and belongs to $L^\infty ([0,T]; C^s )$, for all $(\lambda,h) \in \Lambda\times (0,e^{-1}]$.\\ Furthermore, there exists a continuous non-increasing function $z:[0,T]\longrightarrow (0,m]$ such that for \mbox{any  $p\in [a,\infty]$,} we have
\begin{equation}\label{platitude:propagation}
\forall t\in[0,T],\quad \| (d(\cdot,\Sigma_t))^{-z(t)} \nabla \rho(t,\cdot) \|_{L^p} <\infty,
\end{equation}
\end{theorem} 
\begin{remark}
{Overall, the main improvement in Theorem \ref{THEO:1:soft} compared to the result of \cite{Hassainia-Hmidi} is the following: The initial density is no longer required to be constant in a neighborhood of the singular set $\Sigma_0$ as assumed in \cite{Hassainia-Hmidi}. This technical assumption is replaced here by a weaker compatibility assumption: $\rho_0\in C^{1+m}((\Sigma_0) _r)$ and $\nabla\rho_0$ should vanish  on the singular set $\Sigma_0$.  It is not at all clear whether the LWP issue  can  be performed without the compatibility assumption and we expect some  instability behavior to occur.}
\end{remark} 

\begin{remark}
In Theorem \ref{Th1:general:version}, we assume that $\rho_0$ is of class $ C^{1+m}$  in a small neighborhood of the singular set $\Sigma_0$.  This assumption can be relaxed for the a priori estimates as stated in  \mbox{Proposition \ref{Proposition inviscid}}.  However, we still need the H\"older regularity in the construction of the solutions by smoothing the initial data that should obey  the compatibility assumption. Notice that the persistence of $C^{1+m}-$regularity is out of reach due to the low regularity of the velocity field.
\end{remark}
Let us now show how to deduce  the  proof for Theorem \ref{THEO:1:soft} from the result of Theorem \ref{Th1:general:version}. 
 
\begin{proof}[Proof of Theorem \ref{THEO:1:soft}]
The idea is the same as in  \cite{Hassainia-Hmidi}. For the convenience of the reader, we briefly  outline  the proof. By assumption, we may write
$$
\partial \Omega_0 = \big\{ x\in V:  f_0(x) =0 \big\}, 
$$
for some $V$,  a neighborhood of $\partial \Omega_0$, and for some function $f_0 $ in $C^{1+s}$ satisfying 
$$
|\nabla f_0(x)| \geqslant C d(x,\Sigma_0)^{\gamma}\quad\hbox{and}\quad \gamma>0. 
$$
Let $(\theta_h)_{h\in(0,e^{-1}]}$ be a family of smooth functions  with $\theta_h\equiv1$ in $(\Sigma_0)_h^{c}$ and supported in $(\Sigma_0)_{h^{\alpha }}^{c}$ \mbox{(where $\alpha>1$)} which satisfies, for any $\widetilde{s}>0$
$$
\|\theta_h\|_{C^{\widetilde{s}}}\leqslant C_{\widetilde{s}} h^{{-\widetilde{s}}}.
$$ 
Given a function $\widetilde{\theta}$ in $C^\infty$, with $\widetilde{\theta}\equiv 1$ in $V$, and setting
$$
X_{0,0,h}=\nabla^\perp(\theta_h f_0),\quad X_{0, 1,h}=\theta_h(1-\widetilde{\theta})\vec e_1.
$$
It is easy to show that $\mathcal{X}_0= (X_{0,j,h})_{(j,h)\in\{0,1\}\times(0,e^{-1}]}$ is of class $C^s$ and $\Sigma_0$--admissible of some order $\Xi =(\alpha ,\beta ,\gamma )$, see, for instance \cite{Hassainia-Hmidi}. Additionally, we can check that 
$$\partial_{X_{0,j,h}} \omega_0 = 0, \quad \forall j\in \{1,2\}, $$
and the assumption  $\rho_0 \in C^{1+s}$ in Theorem \ref{THEO:1:soft} implies $$\partial_{X_{0,j,h}} \rho_0\in  C^{ s}, \quad \forall j\in \{1,2\}.$$
More precisely, we have
$$h^{\beta } \frac{\|\partial_{X_{0,j,h}} \rho_0  \|_{C^{ s}}}{I(\Sigma_h,X_{0,j,h})} \leqslant h^{ \beta }\frac{  \| X_{0,j,h}\|_{C^s}}{I(\Sigma_h,X_{0,j,h})} \|\rho_0 \|_{C^{s + 1}}\leqslant \sup_{h\in (0,e^{-1}]}h^{\beta } N_{s}(\Sigma_h,\mathcal{X}_h)  \|\rho_0 \|_{C^{s + 1}} .$$
 Consequently, we obtain 
 \begin{equation*} 
\sup_{h\in (0,e^{-1}]} h^{\beta }\Vert\omega_0\Vert_{(\Sigma_0)_{h},(\mathcal{X}_ {0,h})}^{s}+ \sup_{h\in (0,e^{-1}]} h^{\beta }\Vert\rho_0\Vert_{(\Sigma_0)_{h},(\mathcal{X}_ {0,h})}^{s+1}    <\infty.
\end{equation*} 
Finally, the assumption $\rho_0\in C^{1+s}(\mathbb{R}^2)$ in Theorem \ref{THEO:1:soft} insures then that $\rho_0\in C^{1+m}((\Sigma_0)_r)$ for the particular choice $m=s.$
Therefore, the hypotheses of Theorem \ref{Th1:general:version} are  satisfied and  the local well-posedness in \mbox{Theorem \ref{THEO:1:soft} follows.} Moreover, Theorem \ref{Th1:general:version} gives 
$$
\sup_{h\in (0,e^{-1}]}\frac{ \|\nabla  v(t) \|_{L^\infty((\Sigma_t)_h^c)}}{-\log h}  \in L ^\infty([0,T])
$$
and  for all $h\in (0,e^{-1}]$,
\begin{equation}\label{R0}
X_{t,0,h}(\Psi (t,\cdot)) \triangleq \partial_{X_{0,0,h}} \Psi(t, \cdot) \in L^\infty ([0,T]; C^s). 
\end{equation}  
Finally, concerning the persistence regularity of the boundary of the initial vortex patch, the arguments are the same as in \cite[Chapter 9]{Chemin}. For the sake of completeness  we  shall recall the main lines of the proof. Take  $ x_0\in \partial  \Omega_0\textbackslash \Sigma_0  $ and consider  $\gamma^0$ as  the $C^{s+1}_{\text{loc}}(\mathbb{R} )$-parametrization of the connected component of $\partial \Omega_0\textbackslash \Sigma_0$ containing $x_0$  
\begin{equation*}
\left\{ \begin{array}{l}
\partial_\tau \gamma^0 (\tau) = X_{0,0,h} (\gamma^0(\tau)),\vspace{2mm}\\ 
\gamma^0(0)= x_0.
\end{array}
\right.
\end{equation*}
Thus  the connected component of $\partial \Omega_t\textbackslash \Sigma_t$ containing $\psi(t,x_0)$ is parametrized by  
$$\gamma_t (\tau) \triangleq \Psi(t,\gamma^0(\tau)) , \quad \forall \tau \in \mathbb{R}$$
which satisfies 
\begin{equation*}
\left\{ \begin{array}{l}
\partial_\tau \gamma_t (\tau) = \partial_{X_{0,0,h}}\Psi (t,\gamma^0(\tau)),\vspace{2mm}\\ 
\gamma_t (0)= \Psi(t,x_0).
\end{array}
\right.
\end{equation*}
Hence, according to \eqref{R0}, we infer that $\gamma_t  \in C^{1+s}_{\text{loc}}(\mathbb{R})$ and $\partial_\tau \gamma_t (\tau)$ is not vanishing provided that $h$ is small enough.
 We finally  deduce  that $\gamma_t$ is a local  parametrization of the connected component of $\Psi(t,\partial \Omega_0 \textbackslash \Sigma_0)$ containing $\Psi(t,x_0) $.
\end{proof}
The rest of this section will be dedicated to the proof of   Theorem \ref{Th1:general:version}.
\subsection{A priori estimates}
Since we are interested in the solutions generated from initial data belonging to Yudovich class, the major issue to prove Theorem \ref{Th1:general:version} arises from the difficulty of the propagation of the $L^p$-norms of $\nabla \rho$  and $\omega$. The key observation to do is the following. The quantity $\nabla \rho$ obeys the forcing  transport equation
\begin{equation}\label{nabla-rho:equa}
(\partial_t + v\cdot\nabla ) \nabla \rho = - \nabla v \cdot \nabla \rho.
\end{equation}
Then the control of the Lipschitz norm of $v$ seems to be relevant  to propagate the $L^p$-norm of $\nabla \rho.$ Unfortunately, this information on  the velocity field  is not guaranteed if the vorticity $\omega$ is only bounded. In \cite{Hassainia-Hmidi}, the authors impose to the initial density to be constant on a small neighborhood of the singular set, and in view of the transport structure, the density $\rho(t)$ remains  constant in a neighborhood of the singular set $\Sigma_t$. This allows to control the source term in \eqref{nabla-rho:equa} and deduce   $L^p$-estimates for $\nabla \rho.$ Our goal here is to relax the compatibility assumption in \cite{Hassainia-Hmidi} and replace it by a platitude assumption.  The precise statement is discussed below and constitutes the cornerstone of the a priori estimates that we shall establish later.
\begin{proposition}\label{Proposition inviscid}
Let $\Sigma_0$ be a compact negligible set in $\mathbb{R}^2$ and $(m,p)\in(0,1) \times [1,\infty)$. Let $\rho_0$ be a scalar function satisfying
\begin{equation}\label{hypothesis1}
\nabla \rho_0 \in L^p,   \text{ and } \quad \sup_{x\notin  \Sigma_0 }{|\varphi_0(x)|}^{-m}|\nabla\rho_0(x)|<\infty,
\end{equation} 
where $\varphi_0$ is given by \eqref{tool:varphi0}. Let $(\omega,\rho)$ be any smooth maximal solution  to \eqref{eqn:omega} defined on $[0,T^\star)$ and let the function $m(t)$ be given by 
\begin{equation}\label{alpha:def}
 m(t)\triangleq m- \int_0^t\|{v}(\tau)\|_{L(\Sigma_\tau)} e^{\int_0^\tau\Vert v(s)\Vert_{LL}ds}d\tau, \quad \forall t\in [0,T^\star).
\end{equation}
There exists $C_0>0$ depending only on $\rho_0$ and $\Sigma_0$ such that  the following holds. 
If $m(T)\geqslant 0$, for some $T<T^\star$, then for all $t\in [0,T]$ and for all $r\in [p,\infty]$, we have that
\begin{equation*}
\|\nabla\rho(t)\|_{L^r}\leqslant   C_0
\end{equation*} 
and
\begin{equation*}
  \|\omega(t)\|_{L^r}\leqslant \|\omega_0\|_{L^r}+C_0t   .
\end{equation*}
Furthermore, let  $\Psi$ be the flow associated to $v$ and set
\begin{equation}\label{phi:t:def}
\varphi_t(x) \triangleq \varphi_0(\Psi^{-1}(t,x)),
\end{equation}   
 then, for all $t\in [0,T]$ and for all $r\in [p,\infty]$, we have
$$\| \varphi ^{-m(t)}_t \nabla\rho(t)\|_{L^r}\leqslant   C_0. $$
\end{proposition} 
\begin{proof}

Let $\beta$ be the function given by $\beta(t)= 2 m(t)$ and introduce
$$
\varrho(t,x)=\big[\varphi_t(x)\big]^{-\beta(t)}|\nabla\rho(t,x)|^2,$$
where $\varphi_t$ is given by \eqref{phi:t:def}. One can check that
\begin{equation}\label{phi_t:equa}
\big(\partial_t+ v\cdot\nabla \big)\big[\varphi_t(x)\big]^{-\beta(t)}=-\beta^\prime(t)\log\big(\varphi_t(x)\big)\big[\varphi_t(x)\big]^{-\beta(t)}. 
\end{equation} 
We notice that from the definition of $L(\Sigma_t)$ stated in \eqref{Lsigma},  one deduces
\begin{equation*}
|\nabla {v}(t,x)|\leqslant \|{v}(t)\|_{L(\Sigma_t)} \ln^+(d(x,\Sigma_t)), \quad \forall (t,x)\in [0,T] \times \mathbb{R}^2 .
\end{equation*}
On the other hand, from the equation of $\rho$, we have
$$
 \big(\partial_t+ {v}\cdot\nabla \big)(\partial_i\rho)^2=-2(\partial_i{v})\cdot\nabla\rho\, \partial_i\rho.
$$
Thus, we may write\footnote{We recall that the modulus of a  $2$d vector  is given by $|v|= \sqrt{v_1^ 2+v_2^2}.$}
\begin{align*}
 \big(\partial_t+ {v}\cdot\nabla \big)|\nabla\rho|^2&=-2\sum_{i=1}^2(\partial_i{v})\cdot\nabla\rho\, \partial_i\rho\\
 &\leqslant 2|\nabla {v}||\nabla \rho|^2.
\end{align*}
Hence, $|\nabla\rho|^2 $ satisfies the following system on $\theta$
\begin{equation}\label{equa-theta}
\left\{ \begin{array}{l}
\big(\partial_t+ {v}\cdot\nabla \big)\theta \leqslant 2 \|{v}(t)\|_{L(\Sigma_t)} \ln^+(d(x,\Sigma_t)) \theta \vspace{2mm}\\
\theta_{|_{t=0}}= |\nabla\rho_0|^2.
\end{array}\right.
\end{equation}
Therefore, combining \eqref{phi_t:equa} and \eqref{equa-theta} for $\theta= |\nabla \rho|^2$, gives
\begin{align}\label{F-I1-1}
 \big(\partial_t+ {v}\cdot\nabla \big)\varrho &\leqslant \big(-\beta^\prime(t)\ln\big(\varphi_t(x)\big)+2 \|{v}(t)\|_{L(\Sigma_t)} \ln^+(d(x,\Sigma_t)\big)\varrho.
\end{align}
Since we have, from \eqref{phi:t:def} and \eqref{log+}, that
$$  \ln\big(\varphi_t(x)\big)=-\ln^+\big[d\big(\Psi^{-1}(t,x),\Sigma_0\big)\big],
$$
then, inserting this identity into \eqref{F-I1-1} yields
\begin{align}\label{F-I2}
 \big(\partial_t+ {v}\cdot\nabla \big)\varrho &\leqslant \Big( \beta^\prime(t)\ln^+\big[d\big(\Psi^{-1}(t,x),\Sigma_0\big)\big]+2 \|{v}(t)\|_{L(\Sigma_t)} \ln^+(d(x,\Sigma_t))\Big)\varrho .
\end{align}
Using the inequality of Lemma \ref{elementary propertiy LN 2} we infer
\begin{align}\label{F-I3}
 \big(\partial_t+ {v}\cdot\nabla \big)\varrho &\leqslant \ln^+\big[d\big(\Psi^{-1}(t,x),\Sigma_0\big)\big] \Big(\beta^\prime(t)+2 \|{v}(t)\|_{L(\Sigma_t)} e^{\int_0^t\Vert v(\tau)\Vert_{LL}d\tau}\Big)\varrho .
\end{align}
At this stage  we shall impose the conditions  
$$
\beta^\prime(t)+2 \|{v}(t)\|_{L(\Sigma_t)} e^{\int_0^t\Vert v(\tau)\Vert_{LL}d\tau}\leqslant 0\quad\hbox{with}\quad \beta(0)= 2m,
$$
which are satisfied with the choice
$$
\beta(t)=2m-2 \int_0^t\|{v}(\tau)\|_{L(\Sigma_\tau)} e^{\int_0^\tau\Vert v(s)\Vert_{LL}ds}d\tau,
$$
and equivalently
$$m(t)=m- \int_0^t\|{v}(\tau)\|_{L(\Sigma_\tau)} e^{\int_0^\tau\Vert v(s)\Vert_{LL}ds}d\tau .$$
With this choice of $m(t)$, we obtain 
\begin{equation}\label{inequa:varrho}
\big(\partial_t+ {v}\cdot\nabla \big)\varrho \leqslant0,
\end{equation}
and consequently, we get for all $x\in \Psi(t,\Sigma_0^c)$
\begin{equation}\label{L:infty:1st:es}
{[\varphi_t(x)]}^{-2m(t)}|\nabla\rho(t,x)|^2\leqslant \sup_{y\notin  \Sigma_0 }{[\varphi_0(y)]}^{-2m}|\nabla\rho_0(y)|^2\triangleq (C_\infty)^2. 
\end{equation}
Therefore, the fact that $\Sigma_0$ is a negligible set implies that $\Sigma_0^c$ is dense in $\mathbb{R}^2$. Hence, the fact that   the flow $\Psi$ is a homeomorphism yields that $ \Psi(t,\Sigma_0^c)$ is also dense in $\mathbb{R}^2$. It follows that \eqref{L:infty:1st:es} holds for all $x\in \mathbb{R}^2.$ Thereafter, we find that 
$$
|\nabla\rho(t,x)|^2\leqslant  C_\infty ^2{[\varphi_t(x)]}^{2m(t)}, \quad \forall x\in \mathbb{R}^2.
$$
This implies in particular that
$$
m(t)\geqslant 0\Longrightarrow \|\nabla\rho(t)\|_{L^\infty}\leqslant C_\infty e^{-m(t) }\leqslant C_\infty.
$$
Similarly, we obtain from \eqref{inequa:varrho}
\begin{equation}\label{P:es:INV:1}
\|(\varphi_t(\cdot )\big)^{ -m(t)} \nabla \rho(t,\cdot) \|_{L^p } \leqslant C_{p}, 
\end{equation}
where  
$$
C_p \triangleq \|  \varphi_0 ^{-m}  \nabla\rho_0   \|_{L^p}.
$$
Remark that $C_p  $ is finite due to the assumption \eqref{hypothesis1} and the fact that $\Sigma_0$ is compact. Indeed, we can write in view of notation \eqref{distance-h:def}
\begin{eqnarray}\label{ppp}
\int_{\mathbb{R}^2} \left| \varphi_0(x) ^{-m}  \nabla\rho_0 (x)  \right|^p dx &= & \int_{(\Sigma_0)_{e^{-1}}} \varphi_0(x) ^{-m}  \nabla\rho_0 (x) dx + \|\nabla \rho_0 \|_{L^p((\Sigma_0)_{e^{-1}}^c)}^p \nonumber\\
& \leqslant & \left|(\Sigma_0)_{e^{-1}} \right| C_\infty + \|\nabla \rho_0 \|_{L^p }^p.
\end{eqnarray}
On the other hand, we have 
$$
m(t)\geqslant 0 \Longrightarrow |\nabla \rho(t,\cdot)|=  |\nabla \rho(t,x)| [\varphi_t(x )\big]^{- m(t)} [\varphi_t(x )\big]^{ m(t)}\leqslant |\nabla \rho(t,x)| [\varphi_t(x )\big]^{-m(t)}.   
$$
Therefore, we obtain for any $r\in [p,\infty]$
\begin{equation}\label{P:es:INV:2}
m(t)\geqslant 0 \Longrightarrow \|\nabla \rho(t,\cdot) \|_{L^r}\leqslant     C_r. 
\end{equation}  
Using this estimate together with  the vorticity equation in \eqref{eqn:omega} give  for all $r\in [p,\infty]$
\begin{align*}
 \|\omega(t)\|_{L^r}&\leqslant \|\omega_0\|_{L^r}+\int_0^t \|\nabla\rho(\tau)\|_{L^r}d\tau\\
 &\leqslant  \|\omega_0\|_{L^r}+C_r t,
\end{align*}  
This achieves  the proof of Proposition \ref{Proposition inviscid}.
\end{proof}
Now we shall discuss the following corollary about a useful blow up criterion.
\begin{coro}\label{cor-blowup}
Let $\epsilon\in(0,1)$ and $(v,\rho)\in C\big([0,T^\star),C^{1+\epsilon}(\RR^2)\big)$ be a maximal solution to \eqref{eqn:omega} satisfying the assumptions of Proposition \ref{Proposition inviscid}. Then
$$
T^\star<\infty\Longrightarrow \lim_{t\to T^\star} m(t)< 0,
$$
where the function  $m$ is defined in \eqref{alpha:def}.
\end{coro}
\begin{proof}
Assume that $T^\star<\infty$, then according to the blow up criterion proved in \cite{Chae-Kim-Nam-2}, we have
\begin{equation}\label{bcbb}
\int_0^{T^{\star}}\Vert \nabla \rho (\tau)\Vert_{L^\infty}d\tau= \infty.
\end{equation}
Now since $t\in[0,T^\star)\mapsto m(t)$ is non-increasing then $ \displaystyle{\lim_{t\to T^\star} m(t)\triangleq \ell}$ exists and belongs to $[-\infty,\infty)$. At this level we distinguish two cases. The first one  is when $\ell<0$ and in this case the result is proved. However in the second case $\ell\geqslant 0$ we get by the monotonicity  that
$$
\forall t\in[0,T^\star),\quad m(t)\geqslant 0.
$$
Applying Proposition \ref{Proposition inviscid} we find
\begin{equation*}
\|\nabla\rho(t)\|_{L^\infty}\leqslant   C_0
\end{equation*} 
which implies that
\begin{equation*}
\int_0^t\|\nabla\rho(\tau)\|_{L^\infty}d\tau\leqslant   C_0 t\leqslant C_0T^\star.
\end{equation*}
This contradicts \eqref{bcbb}.
\end{proof}
\hspace{0.5cm} In the sequel, we shall explore some applications of  Proposition \ref{Proposition inviscid}. The first one deals with the persistence of  anisotropic regularity of $\omega$ and $\rho$.

\begin{proposition}\label{prop22}
Let $(s,a)\in(0,1)\times [1,\infty)$  and  $\Sigma_0$ be a compact negligible set of the plane. Let $X_0$ be a vector field of class $C^s$ as well as its divergence and whose support  is embedded in $(\Sigma_0)_h^c$, for some $h\in (0,e^{-1}]$. Let $\rho_0$ be a scalar function satisfying the assumptions  of Proposition \ref{Proposition inviscid}. Consider   a smooth maximal solution $(\omega,\rho)$  of the system (\ref{eqn:omega}) defined on $[0,T^\star)$, with the  initial data $(\omega_0,\rho_0)$ and denote by $X_t$ the solution of the transport equation
\begin{equation*}
\left\{\begin{array}{ll}
\big(\partial_t +v\cdot\nabla\big)X_t=\pxt v. &\vspace{2mm}\\
X_{| t=0}=X_{0}.
\end{array} \right.
\end{equation*} 
Then, the following assertions hold
\begin{equation*}
\supp X_t\subset \big(\Sigma_t\big)_{\delta_t(h)}^c
\end{equation*}
and
\begin{equation*}
\Vert \Div X_t\Vert_{s}\leqslant \Vert \Div X_{0}\Vert_{s}h^{-C\int_0^t W(\tau)d\tau}, \quad t\in[0,T^\star),
\end{equation*}
where, $\delta_t(h)$ is defined in \eqref{def-delta}.
Furthermore, let $T\in(0,T^\star)$ such that $m(T)\geqslant 0$,  
with  $m$ being the function defined in \eqref{alpha:def}.
Then we get  for all $t\in [0,T]$, 
\begin{eqnarray*}
\widetilde{\Vert} X_t\Vert_{s} +\Vert\pxt \omega(t)\Vert_{s-1} &\leqslant & C e^{C_0t }\Big(\widetilde{\Vert} X_0\Vert_{ s}+\Vert \pxz\omega_0\Vert_{s-1}+\Vert \pxz\rho_0\Vert_{s}\Big)h^{-C\int_0^t W(\tau)d\tau}\notag,
\end{eqnarray*}
\begin{equation*}
\Vert \pxt \rho(t)\Vert_{s}\leqslant \Vert \pxz \rho_{0}\Vert_{ s}h^{-C\int_0^t W(\tau)d\tau},
\end{equation*}
where $C>0$ is a  constant independent of the initial data and $W(t)$ is given by \eqref{W:def}.
\end{proposition}
\begin{proof}
The results concerning the support of  $X_t$  were already proved in Corollary \ref{coro-transport:es}. Moreover, by applying Corollary \ref{coro-transport:es} for $p=\infty$ it happens 
\begin{eqnarray}\label{a-38-E*****}
\widetilde{\Vert} X_t\Vert_{s}&\leqslant & h^{-C\int_{0}^{t}W(\tau)d\tau}\Big(\widetilde{\Vert} X_0\Vert_{s} +C\int_0^t h^{C\int_{0}^{\tau}W(\tau')d\tau'}\|\partial_{ X_{\tau}}\omega(\tau)\|_{s-1}d\tau\Big),
\end{eqnarray}
Now, we focus on the estimate  of  $\partial_{X_t} \omega.$ Since $\partial_{X_t}$ commutes with the transport, that is
\begin{equation}\label{commutator-DIFF}
 [\partial_t+v\cdot\nabla , \partial_{X_t}] =0,
\end{equation} 
then we deduce from  the vorticity  equation that
\begin{equation*}
\big(\partial_t +v\cdot\nabla\big)\partial_{X_{t}}\omega=\partial_{X_{t}}\partial_1\rho.
\end{equation*}
Applying  Proposition \ref{prop:transport:es} allows to find 
\begin{eqnarray*}
\Vert \partial_{X_{t}}\omega(t)\Vert_{ s-1}&\lesssim &\Vert \partial_{X_{0}}\omega_0\Vert_{ s-1}h^{-C\int_0^t W(\tau)d\tau}+\int_0^t\Vert  \partial_{X_{\tau}}\partial_1\rho(\tau)\Vert_{s-1}h^{-C\int_\tau^t W(\tau')d\tau'}d\tau\notag\\ &\lesssim & \Vert \partial_{X_{0}}\omega_0\Vert_{ s-1}h^{-C\int_0^t W(\tau)d\tau}+\int_0^t\Vert  \partial_{X_{\tau}}\rho(\tau)\Vert_{s}h^{-C\int_\tau^t W(\tau')d\tau'}d\tau\\
&+&\int_0^t\Vert\nabla\rho(\tau)\Vert_{L^\infty}\widetilde{\Vert} X_{\tau}\Vert_{s}h^{-C\int_\tau^t W(\tau')d\tau'}d\tau,
\end{eqnarray*}
where, we have used  the following inequalities   
\begin{eqnarray}\label{pp5}
\Vert\partial_{X_{\tau }}\partial_1\rho(\tau)\Vert_{s-1}&\lesssim & \Vert \partial_{X_{\tau }}\rho(\tau)\Vert_s+\Vert(\partial_{1}X_{\tau })\cdot\nabla\rho(\tau)\Vert_{s-1}\notag \\  &\lesssim &\Vert \partial_{X_{\tau }}\rho(\tau)\Vert_s+\Vert\nabla\rho(\tau) \Vert_{ L^{\infty}}\widetilde{\Vert} X_{\tau }\Vert_{ s}
\end{eqnarray} 
and the estimate of the last term follows from  Corollary \ref{ppxr0}. Now, due to \eqref{commutator-DIFF}, we have that
 $$
 (\partial_t+v\cdot\nabla )\partial_{X_{t}}\rho(t)=0.
 $$ 
It is easy to check that  $\partial_{X_{0}}\rho_0$ is supported  in $(\Sigma_0)_{h}^c$ since $\text{supp } X_{0} \subset (\Sigma_0)_{h}^c$. Thus, Proposition \ref{prop:transport:es} gives 
$$
\Vert  \partial_{X_{\tau}}\rho(\tau)\Vert_{s}\leqslant \Vert  \partial_{X_0}\rho_0\Vert_{s}\,h^{-C\int_0^\tau W(\tau')d\tau'}.
$$
 Hence, we obtain 
\begin{eqnarray*}
\Vert \partial_{X_{t}}\omega(t)\Vert_{ s-1}&\lesssim& \Big(\Vert \partial_{X_{0}}\omega_0\Vert_{ s-1}+\Vert  \partial_{X_{0}}\rho_0\Vert_{s}t\Big)h^{-C\int_0^t W(\tau)d\tau}\\
&+&\int_0^t\Vert\nabla\rho(\tau) \Vert_{L^\infty}\widetilde{\Vert} X_{\tau}\Vert_{s}h^{-C\int_\tau^t W(\tau')d\tau'}d\tau.
\end{eqnarray*}
Putting together the preceding estimate with \eqref{a-38-E*****} we find that
\begin{equation*}
\zeta(t)\lesssim  \zeta(0)+\Vert \partial_{X_{0}}\rho_0\Vert_{ s}t+\int_0^t\Big(\Vert \nabla\rho(\tau)\Vert_{L^\infty}+1\Big)\zeta(\tau)d\tau,
\end{equation*}
where  $$\zeta(t)\triangleq\Big(\Vert \partial_{X_{t}}\omega(t)\Vert_{s-1}+\widetilde{\Vert} X_{t}\Vert_{s}\Big)h^{C\int_0^t W(\tau)d\tau}.$$ 
Therefore, Gronwall's lemma ensures that
\begin{equation}\label{gam}
\zeta(t)\leqslant \big(\zeta(0)+\Vert  \partial_{X_{0}}\rho_0\Vert_{ s}\big)e^{C\int_0^t\Vert \nabla \rho(\tau)\Vert_{L^\infty}d\tau}e^{Ct} 
\end{equation}
and applying  Proposition \ref{Proposition inviscid} completes the proof.
\end{proof}
\hspace{0.5cm}Now, we are in a position to control the Lipschitz norm of the velocity outside the transported set of $\Sigma_0$ by the flow. In addition to that, we shall get  a lower bound of the function $m(t)$  in terms of the initial data. This provides an a priori bound  on the existence time of  the solutions to the system \eqref{eqn:omega}.  
\begin{proposition}\label{prop220}
Let $(s,a)\in(0,1)\times [1,\infty)$, $\Sigma_0$ be a compact negligible set of the plane and $\mathcal{X}_0=\big(X_{0,\lambda,h}\big)_{\lambda\in\Lambda,h\in(0,e^{-1}]}$ be a family of vector fields  which is   $\Sigma_0$-admissible of order $\Xi =(\alpha ,\beta ,\gamma )$. Let  $(\omega,\rho)$ be a smooth maximal solution of the  \mbox{system $(\ref{eqn:omega})$} defined  on a  time interval  $[0,T^\star)$. We assume that   $\omega_0  \in L^a \cap L^\infty$, and that $\rho_0$  satisfies the assumptions in Proposition \ref{Proposition inviscid}. 
Assume in addition  that
  \begin{equation*} 
\sup_{h\in (0,e^{-1}]} h^{\beta }\Vert\omega_0\Vert_{(\Sigma_0)_{h},(\mathcal{X}_ {0,h})}^{s}+ \sup_{h\in (0,e^{-1}]} h^{\beta }\Vert\rho_0\Vert_{(\Sigma_0)_{h},(\mathcal{X}_ {0,h})}^{s+1}    <\infty.
\end{equation*}  
Then, there exists $T\in(0,T^\star) $ such that the estimates for $\rho$ and $\omega$ in Proposition \ref{Proposition inviscid} hold on $[0,T]$ and 
$$
\sup_{t\in[0,T]}\Vert   v(t)\Vert_{L(\Sigma_t)}\leqslant C_0.
$$
Notice that $T$ is defined as the unique solution to $\Gamma(T)= m$, where
\begin{equation}\label{Gamma:def***}
\Gamma(t) \triangleq t e^{ \exp{(C_0(1+t)^2)}}
\end{equation}
and $C_0>0$ is a constant depending only on the initial data and $\Sigma_0.$\\
Furthermore, there exist some continuous functions  $\Xi(t)=(\alpha(t),\beta(t),\gamma(t))$ such that, 
 \begin{equation*}
\forall\, t\in[0,T],\quad \sup_{h\in (0,e^{-1}]} h^{\beta(t)}\Vert\omega(t)\Vert_{(\Sigma_t)_{\delta_t^{-1}(h)},(\mathcal{X}_{t,h}) }^{s}+\sup_{h\in (0,e^{-1}]} h^{\beta(t)}\Vert\rho(t)\Vert_{(\Sigma_t)_{\delta_t^{-1}(h)},(\mathcal{X}_{t,h})}^{s+1}  \leqslant C_0.
\end{equation*}  
Moreover,  the transported family  $\mathcal{X} _t$ of $\mathcal{X} _0$ by the flow $\Psi$ defined by
$$X_{t,\lambda, h}(\cdot) \triangleq  \partial_{X_{0,\lambda,h}}\Psi(t,\Psi^{-1}(t,\cdot) )  , $$
is $\Sigma_t$--admissible of order $\Xi(t)$ and belongs to $L^\infty([0,T]; C^s),$ for all $(\lambda,h) \in \Lambda \times (0,e^{-1}].$
\end{proposition}
\begin{proof}
Let  $T\in(0,T^\star)$ such that
 \begin{equation}\label{assumption on T max}
m(T)\geqslant 0,
\end{equation} where $m(t)$ was defined in Proposition \ref{prop22} and takes the form  
\begin{equation}\label{alpha-def**}
m(t)=m- \int_0^t\|{v}(\tau)\|_{L(\Sigma_\tau)} e^{\int_0^\tau\Vert v(\tau')\Vert_{LL}d\tau'}d\tau.
\end{equation} 
Since  $m$ is non-increasing then  we easily find 
$$
\forall t\in[0,T],\, m(t)\geqslant 0
$$
This ensures in particular  that all the a priori estimates developed in Proposition \ref{Proposition inviscid} and \mbox{Proposition \ref{prop22}}  are valid on $[0,T]$.
Let us examine the estimate of 
$$
\Upsilon(t)\triangleq  \Vert\omega(t)\Vert_{L^\infty}\widetilde{\Vert} X_{t,\lambda,h}\Vert_{s}+\Vert \partial_{X_{t,\lambda,h}} \omega(t)\Vert_{s-1}.
$$
Then combining Proposition \ref{Proposition inviscid} and Proposition \ref{prop22} yields for any $t\in[0,T],$
\begin{eqnarray*}
\Upsilon(t)&\lesssim   C_0 e^{C_0t }\big(1+\Vert\omega \Vert_{L^\infty}\big)\Big(\widetilde{\Vert} X_{0,\lambda,h}\Vert_{s}+\Vert \partial_{X_{0,\lambda,h}} \omega_0\Vert_{s-1}+\Vert \partial_{X_{0,\lambda,h}} \rho_0\Vert_{s-1}\Big)    h^{-C  \int_0^t W(\tau)d\tau }.
\end{eqnarray*}
Hence, we deduce in view of the  Definition \ref{def11} and \eqref{ixtd}
\begin{eqnarray}\label{eqom}
\Upsilon(t) &\leqslant & C_0 e^{C_0t }\big(1+\Vert\omega_0\Vert_{L^\infty}\big)\Big(N_s\big((\Sigma_0)_{h},(\mathcal{X}_0)_{h}\big) + \Vert\omega_0\Vert^{s}_{(\Sigma_0)_{h},(\mathcal{X}_0)_{h}}+\Vert\rho_0\Vert^{s+1}_{(\Sigma_0)_{h},(\mathcal{X}_0)_{h}}  \Big)\notag\\
 &&\times  I\big((\Sigma_0)_{h},(\mathcal{X}_0)_{h}\big) h^{-C \int_0^t W(\tau)d\tau} \notag\\ 
&\leqslant &   C_0 e^{C_0t }\big(1+\Vert\omega_0\Vert_{L^\infty}\big)\sup_{0<h\leqslant e^{-1}}h^{\beta }\Big(N_s((\Sigma_0)_{h},(\mathcal{X}_0)_{h})+  \Vert\omega_0\Vert^{s}_{(\Sigma_0)_{h},(\mathcal{X}_0)_{h}}+\Vert\rho_0\Vert^{s+1}_{(\Sigma_0)_{h},(\mathcal{X}_0)_{h}} \Big)\notag\\ 
&&\times  I\big((\Sigma_t)_{\delta_t^{-1}(h)},(\mathcal{X}_{t,h})\big)h^{-\beta -C \int_0^t W(\tau)d\tau} \notag.
\end{eqnarray}
Putting together  this estimate  with    the norm  definition (\ref{defre}) allows to get  (with a suitable adjustment of $C_0$)
\begin{eqnarray}\label{a}
\Vert\omega(t)\Vert_{(\Sigma_t)_{\delta_t^{-1}(h)},(\mathcal{X}_{t,h})}^{s}&\leqslant & C_0 e^{C_0t } h^{-\beta -C\int_0^t W(\tau)d\tau}  .
\end{eqnarray}
Consequently, we obtain 
\begin{equation*}
\sup_{h\in (0,e^{-1}]} h^{\beta(t)}\Vert\omega(t)\Vert_{(\Sigma_t)_{\delta_t^{-1}(h)},(\mathcal{X}_{t,h})}^{s}  \leqslant  C_0 e^{C_0t },
\end{equation*}  
where 
\begin{equation}\label{beta-def**}
 \beta(t) \triangleq  \beta  + C \int_0^t W(\tau)d\tau>0. 
\end{equation}  
 Similarly, repeating the same arguments yields 
 \begin{equation*}
\sup_{h\in (0,e^{-1}]} h^{\beta(t)}\Vert\rho(t)\Vert_{(\Sigma_t)_{\delta_t^{-1}(h)},(\mathcal{X}_{t,h})}^{s+1}+\sup_{\underset{\lambda \in \Lambda}{h\in (0,e^{-1}]}} h^{\beta(t)}\frac{ \widetilde{\Vert} X_{t,\lambda,h}\Vert_{s} }{I\big((\Sigma_t)_{\delta_t^{-1}(h)},(\mathcal{X}_{t,h})\big)}   \leqslant  C_0 e^{C_0t }.
\end{equation*} 
The next target is to  establish a Lipschitz estimate  for  the velocity. For this purpose, we combine Theorem \ref{propoo1} with Proposition \ref{Proposition inviscid} and the monotonicity of the \mbox{map $x\mapsto x\log \big(e + \frac{a}{x}\big)$} to get
\begin{eqnarray*}
\Vert\nabla v(t)\Vert_{L^\infty((\Sigma_t)_{\delta_t^{-1}(h)}^c)}\leqslant  C\Big(\Vert\omega_0\Vert_{L^a\cap L^\infty}+tC_\infty\Big)\log\bigg(e+\frac{ \Vert\omega\Vert_{(\Sigma_t)_{\delta_t^{-1}(h)},(\mathcal{X}_{t,h})}^{s}}{\Vert \omega_0\Vert_{L^\infty}}\bigg).
\end{eqnarray*}
Therefore, we find according to the estimate (\ref{a}), 
\begin{eqnarray*}
\Vert\nabla v(t)\Vert_{L^\infty((\Sigma_t)_{\delta_t^{-1}(h)}^c)}&\leqslant & C_0 (1+t  )      \bigg( C_0(1+t)-\Big(\beta +C\int_0^t W(\tau)d\tau\Big)\log h\bigg).
\end{eqnarray*}
Since $1 \leqslant - \log h$, then it follows by virtue of  \eqref{delta:inverse} that
\begin{eqnarray*}
\frac{\Vert\nabla v(t)\Vert_{L^\infty((\Sigma_t)_{\delta_t^{-1}(h)}^c)}}{-\log \delta_t^{-1}(h)}&\leqslant &  C_0 (1+t  )  \bigg(1+t + \int_0^t W(\tau)d\tau \bigg)   e^{\int_0^t\Vert v(\tau)\Vert_{LL}d\tau}. 
\end{eqnarray*}
Thus we infer from \eqref{Lsigma}
\begin{eqnarray*}
\| v(t)\|_{L(\Sigma_t)}&\leqslant &  C_0 (1+t  )  \bigg(1+t + \int_0^t W(\tau)d\tau \bigg)   e^{\int_0^t\Vert v(\tau)\Vert_{LL}d\tau}. 
\end{eqnarray*}
Thereby, from the definition of the function $W$ \eqref{W:def}, we obtain
\begin{eqnarray*}
\forall \, t\in[0,T],\quad W(t)&\leqslant &   C_0 (1+t  )  \bigg(1+t + \int_0^t W(\tau)d\tau \bigg)    e^{2\int_0^t\Vert v(\tau)\Vert_{LL}d\tau} .
\end{eqnarray*}
On the other hand, from Lemma \ref{lem3} and Proposition \ref{Proposition inviscid}, we get  
\begin{eqnarray}\label{esll}
\Vert v(t)\Vert_{LL}&\leqslant &\Vert \omega(t)\Vert_{L^a\cap L^\infty}\notag\\ &\leqslant &  C_0(1+t).
\end{eqnarray}
Hence,
\begin{eqnarray*}
\forall \, t\in[0,T],\quad W(t)&\leqslant &  C_0 (1+t  ) ^2 e^{C_0(1+t)^2}+ e^{C_0(1+t)^2} \int_0^t W(\tau)d\tau.
\end{eqnarray*}
Applying  Gronwall lemma implies
\begin{eqnarray*}
\forall \, t\in[0,T],\quad W(t) &\leqslant &  \exp\Big(e^{C_0 (1+t)^2 }\Big).
\end{eqnarray*}
It follows that 
\begin{align}\label{ESS-T1}
\forall \, t\in[0,T],\quad \int_0^t W(\tau)d\tau\leqslant \Gamma(t) \triangleq t\,e^{\exp{C_0(1+t)^2}} \leqslant \Gamma(T).
\end{align}
Consequently, from the definition of $m(t)$ in \eqref{alpha-def**} and the definition of $W$ in \eqref{W:def},  we infer that
\begin{equation}\label{alpha-gamma}
m(T)\geqslant 0\Longrightarrow \forall \, t\in[0,T],\quad m(t) \geqslant m - \int_0^tW (\tau)d\tau \geqslant  m -  \Gamma(T).
\end{equation}   
Now, from Corollary \ref{cor-blowup} combined with the continuity of $t\mapsto m(t)$ we deduce the existence of $\widetilde{T}\in(0,T^\star)$ such that 
$$
m(\widetilde{T})=0.
$$
Combined with  \eqref{alpha-gamma} yield
$$
 m -  \Gamma(\widetilde{T})\leqslant 0.
$$
Next, by defining  $T$ such that,
 \begin{equation}\label{T:max}
 \Gamma(T)=m,
 \end{equation}
we obtain the lower bound of the  maximal time existence
 $$
 T\leqslant \widetilde{T}<T^\star.
 $$
 This ends  the proof of the Proposition \ref{prop220}. We observe that $T\in(0,1).$
\end{proof}
\hspace{0.5cm}We conclude this section by the following corollary which is useful  in the proof of the regularity persistence of the transported vortex patch.
\begin{coro}\label{cor:X-Psi:1}
Under the assumptions of Proposition \ref{prop220}, we have
$$ t\longmapsto\partial_{X_{0,\lambda,h}} \Psi(t,\cdot) =  X_{t,\lambda,h}(\Psi(t,\cdot)) \in L^\infty\big([0,T]; C^s \big), \quad \forall (\lambda, h)\in \Lambda\times(0,e^{-1}]. $$
More precisely, there exists a constant $C_{0,h,T}>0$, depending only on the initial data, $h$ and $T,$ such that
\begin{equation}\label{X:es:0.0.0.0.}
\forall\, t\in[0,T],\quad \| X_{t,\lambda,h}(\Psi(t,\cdot))\|_{C^{s }}\leqslant C_{0,h,T}.
\end{equation}
\end{coro}
\begin{proof}
Remark first that, due to Proposition \ref{prop22}, we have
$$\| X_{t,\lambda,h}(\Psi(t,\cdot))\|_{L^\infty}\lesssim \| X_{t,\lambda,h} \|_{L^\infty}\lesssim \| X_{t,\lambda,h} \|_{s} \leqslant C_{0,h,T}.$$
For H\"older estimates, we claim that for all $x,x'\in \mathbb{R}^2$ such that $|x-x'|\leqslant \tfrac{h}{2}$
\begin{equation}\label{claim:coro:3}
\big|X_{t,\lambda,h}(\Psi(t,x) )- X_{t,\lambda,h}(\Psi(t,x'))  \big| \leqslant C_{0,h,T}  |x-x'|^s.
\end{equation}
To prove it, we shall consider   two cases.\\
 \ding{118} If $x,x'\in (\Sigma_0)_h$, then  from $\supp X_{0,\lambda,h}\subset (\Sigma_0)_h^c$ we easily get
$$\partial_{X_{0,\lambda,h}} \Psi(t,x) =  X_{t,\lambda,h}(\Psi(t,x))=\partial_{X_{0,\lambda,h}} \Psi(t,x') =  X_{t,\lambda,h}(\Psi(t,x'))=0. $$
 \ding{118} If $x,x'\in (\Sigma_0)_{\frac{h}{2}}^c $, then we proceed as follows. We recall first from  Proposition \ref{prop220} that
$$
\partial_{X_{0,\lambda,h}} \Psi(t,\Psi^{-1}(t,\cdot)) =  X_{t,\lambda,h}(\cdot) \in L^\infty \big([0,T]; C^s \big), \quad \forall (\lambda, h)\in \Lambda\times(0,e^{-1}]. 
$$
This implies that, for all $x,x'\in (\Sigma_0)_h^c$  
\begin{eqnarray}\label{0.0.0.0}
\big|X_{t,\lambda,h}(\Psi(t,x) )- X_{t,\lambda,h}(\Psi(t,x'))  \big| &\lesssim &\|  X_{t,\lambda,h}\|_{s}\big|\Psi(t,x)-\Psi(t,x')  \big|^s \nonumber\\
& \lesssim &\|  X_{t,\lambda,h}\|_{s}\big|x-x'\big|^s \sup_{z\in [x,x']} \big|\nabla \Psi(t,z) \big|^s.
\end{eqnarray}
By assumption $|x-x'|\leqslant \frac{h}{2}$ and therefore  
$$ 
\sup_{z\in [x,x']} \big|\nabla \Psi(t,z) \big| \leqslant  \|\nabla \Psi(t,\cdot) \|_{L^\infty ((\Sigma_0)_{\frac{h}{2}}^c)},
$$
which implies in turn that
\begin{equation}\label{aaa0.0}
\big|X_{t,\lambda,h}(\Psi(t,x) )- X_{t,\lambda,h}(\Psi(t,x'))  \big|  \lesssim\|  X_{t,\lambda,h}\|_{s}\big|x-x'\big|^s  \|\nabla \Psi(t,\cdot) \|_{L^\infty ((\Sigma_0)_{\frac{h}{2}}^c)}.
\end{equation}
Next, we intend to estimate $\nabla \Psi$ in the right hand side of the preceding inequality. To this end, differentiating the equation of the flow and using Gronwall's lemma yields
\begin{eqnarray*}
\|\nabla \Psi(t,\cdot) \|_{L^\infty ((\Sigma_0)_{\frac{h}{2}}^c)  } \leqslant \exp \Big(\int_0^t \|\nabla v(\tau,\Psi(\tau, \cdot)) \|_{L^\infty ((\Sigma_0)_{\frac{h}{2}}^c)  } d\tau\Big).
\end{eqnarray*}
Moreover,  Lemma \ref{s2lem1} gives
$$ \Psi(t,(\Sigma_0 )_{\frac{h}{2}}^c)  \subset \big( \Sigma_t \big)_{\delta_t(\frac{h}{2} )}^c. $$
Hence we deduce 
\begin{eqnarray}\label{nabla-psi:0.0}
\|\nabla \Psi(t,\cdot) \|_{L^\infty ((\Sigma_0)_{\frac{h}{2}}^c)  } &\leqslant &\exp \Big(\int_0^t \|\nabla v(\tau, \cdot ) \|_{L^\infty ((\Sigma_\tau)_{\delta_\tau(\frac{h}{2})}^c)  } d\tau\Big)\nonumber\\
& \leqslant &\exp\Big(-\int_0^t \log(\delta_\tau(\tfrac{h}{2}))\| v(\tau, \cdot ) \|_{L   (\Sigma_\tau)    } d\tau\Big)\\
& \leqslant &  e^{ C_{0,h}}, \nonumber
\end{eqnarray}
where we have use in the last inequality the following consequences of Proposition \ref{Proposition inviscid} and Proposition \ref{prop220},  for all $t\leqslant T$
\begin{equation*}
\sup_{\tau\in [0,t]} \|v(\tau) \|_{LL} \leqslant \sup_{\tau\in [0,t]} \|\omega(\tau) \|_{L^2\cap L^\infty} \leqslant C_0(1+t),
\end{equation*} 
\begin{equation*} 
\sup_{\tau\in [0,t]} \|v(\tau) \|_{L(\Sigma_\tau)}\leqslant C_0.
\end{equation*} 
and the fact that $T\leqslant 1$ (due to \eqref{T:max}).
Finally, in order to prove  \eqref{claim:coro:3}, it is enough to implement \eqref{nabla-psi:0.0} in \eqref{aaa0.0}.
\end{proof}
\subsection{Proof of Theorem \ref{Th1:general:version}}\label{Existence-uniq-inviscid}
In the sequel we shall give the main details on the proof of \mbox{Theorem \ref{Th1:general:version}.}
\begin{proof}[\ding{202} \textit{Approximated solutions}] 
To prove the existence of a solution, we need to implement some modifications in the proof given in \cite{Hassainia-Hmidi}. For $n\in \mathbb{N}^\star$, let us smooth out the initial data and  consider the solutions to the  system, 
\begin{equation}\label{schema-0}
\left\{ \begin{array}{lll}
\partial_{t}v_n+v_n\cdot\nabla v_n+\nabla p_n =\rho_n \begin{pmatrix}0\\
  1 \end{pmatrix}, &\\
\partial_{t}\rho_n+v_n\cdot\nabla\rho_n =0, &\vspace{2mm}\\
\Div v_n=0,\\
v_{0,n} = S_nv_0,\quad \rho_{0,n} = \widetilde{S} _n\rho_0,
\end{array} \right. 
\end{equation}
where $S_n$ is the cut-off in frequency defined in Section 2 and $\widetilde{S}_n$ is given by
$$
\widetilde{S}_n \rho_0 \triangleq \chi_r \rho_0 + (1-\chi_r)S_n \rho_0,
$$
with $\chi_r$ is a smooth cut-off function with values in $[0,1]$ and satisfying 
$$
\chi_r(x) = \left\{\begin{array}{ll}
1 & \text{ if } d(x,\Sigma_0)\leqslant\frac{r}{2},\\
0 & \text{ if } d(x,\Sigma_0)\geqslant r.
\end{array} \right. 
$$
Notice that the density is smoothed in a different way in order to get the compatibility conditions. 
Now, by assumption, we have $ \chi_r \rho_0 \in C^{1+m} $, where $m\in (0,1)$. Also, it is clear that $v_{0,n}$ and $(1-\chi_r)S_n  \rho_0  $ are smooth and belong to $C^{1+m}$. Then, we can apply  the  result of  \cite{Chae-Kim-Nam-2} and construct,  for each $n\in \mathbb{N}^{\star}$, a unique maximal solution $v_n,\rho_n\in C\big([0,T^{\star}_n);C^{1+m}\big).$ \\

\ding{203} {\it Boundedness of the initial data.}
The first thing to check is that the initial data of the scheme \eqref{schema-0} satisfy the assumptions in Theorem \ref{Th1:general:version} and Proposition \ref{Proposition inviscid}, uniformly with respect to $n$. Let us first check that     
$$
\Vert\omega_{0,n}\Vert_{L^a \cap L^\infty}, \Vert \rho_{0,n}\Vert_{W^{1,a} \cap W^{1,\infty}}
$$
and
\begin{equation}\label{CLAIM:N}
\sup_{h\in(0,e^{-1]}} h^{\beta}\Big(\Vert \omega_{0,n}\Vert^{s}_{(\Sigma_0)_h, \mathcal{X} _{0,h}}+\Vert \rho_{0,n}\Vert^{s+1}_{(\Sigma_0)_h,\mathcal{X} _{0,h}}\Big)
\end{equation}
are uniformly bounded with respect to $n$. 
The first claim follows from the uniform continuity of the operator $S_n:L^p\to L^p$. As to the boundedness of  \eqref{CLAIM:N},  we proceed as follows. We repeat first the arguments in pages 62-63 from \cite{Chemin}  allowing to get
\begin{equation}\label{CL:1}
\Vert \partial_{X_{0,\lambda,h}}\omega_{0,n}\Vert_{s-1}\leqslant C\big(\Vert \partial_{X_{0,\lambda,h}}\omega_{0}\Vert_{s-1}+\widetilde{\Vert} X_{0,\lambda,h}\Vert_{s}\Vert\omega_0\Vert_{L^\infty}\big). 
\end{equation}
On the other hand, for the second term in \eqref{CL:1}, we write first
\begin{eqnarray*}
 \partial_{X_{0,\lambda,h}}\rho_{0,n} &=&   \partial_{X_{0,\lambda,h}}\big( \chi_r\rho_{0 }\big)- S_n\rho_{0 } \partial_{X_{0,\lambda,h}}\big(\chi_r\big)  +  (1-\chi_r)\partial_{X_{0,\lambda,h}}\big(S_n \rho_{0 }\big) \\
 &=&  \partial_{X_{0,\lambda,h}}\big( \chi_r\rho_{0 }\big)- S_n\rho_{0 } \partial_{X_{0,\lambda,h}}\big(\chi_r\big)  +  (1-\chi_r)\Bigg(S_n\partial_{X_{0,\lambda,h}}\rho_{0 }+ [\partial_{X_{0,\lambda,h}}, S_n ]\rho_{0 }\Bigg) \\
 & \triangleq & \sum_{i=1}^4 \mathcal{J}_i.
\end{eqnarray*} 
The estimates of the first three terms are straightforward and one gets
$$
 \sum_{i=1}^3 \| \mathcal{J}_i \|_s\leqslant C\big(\Vert \partial_{X_{0,\lambda,h}}\rho_0\Vert_{s}+\Vert X_{0,\lambda,h}\Vert_{s}\Vert \rho_0\Vert_{s}\big),
$$
where, $C$ depend only on the $C^{s+1}$-norm of $\chi_r$.   
As  to the estimate $\mathcal{J}_4$, we use \cite[Proposition 8]{Hassainia-Hmidi} in order to find
$$
 \| \mathcal{J}_ 4\|_s \leqslant C\big(\Vert \partial_{X_{0,\lambda,h}}\rho_0\Vert_{s}+\widetilde{\Vert} X_{0,\lambda,h}\Vert_{s}\Vert \nabla \rho_0\Vert_{L^\infty}\big),
$$
Consequently, we end up with 
\begin{equation}\label{CL:2}
\Vert \partial_{X_{0,\lambda,h}}\rho_{0,n}\Vert_{s}\leqslant C\big(\Vert \partial_{X_{0,\lambda,h}}\rho_0\Vert_{s}+\widetilde{\Vert} X_{0,\lambda,h}\Vert_{s}\Vert \rho_0\Vert_{W^{1,\infty}}\big).
\end{equation}
Finally, \eqref{CLAIM:N} follows from \eqref{CL:1}, \eqref{CL:2} and the assumptions in Theorem \ref{Th1:general:version} on $\rho_0$ and $ X_{0,\lambda,h}.$  
Next, the assumptions ${(\nabla \rho_0)}_{|_{\Sigma_0}}=0 $ and $\chi_r\rho_0\in C^{1+m}(\mathbb{R}^2)$ ensure the platitude condition \eqref{hypothesis1} for the regularized initial density $\widetilde{S}_n\rho_0$. Indeed, we have by construction 
$$\widetilde{S}_n  \rho_0(x)= \rho_0(x) , \quad \forall x\in (\Sigma_0)_{\frac{r}{2},}$$
implying that we can replace $\widetilde{S}_n\rho_0$ by $\rho_0$ in the neighborhood $(\Sigma_0)_{\frac{r}{2}}.$
Since   $(\nabla \rho_0)_{|_{\Sigma_0}}=0$, then
\begin{eqnarray*}
\sup_{\underset{x\neq y,\; |x-y|<\frac{r}{2}}{y\in \Sigma_0} } \frac{\left|\nabla \rho_0(x)\right|}{|x-y|^m}&=& \sup_{\underset{x\neq y,\; |x-y|<\frac{r}{2}}{y\in \Sigma_0} }\frac{\left|\nabla \rho_0(x)-\nabla \rho_0(y)\right|}{|x-y|^m}\\
& \lesssim & \|\nabla \rho_0 \|_{C^m((\Sigma_0)_{r})}\\
&\lesssim &\|  \rho_0 \|_{C^{1+m}((\Sigma_0)_{r})} .
\end{eqnarray*}

  \ding{204}{\it Uniform lower bound on the lifespan $T_n^\star$.} This can be done by virtue of Corollary \ref{cor-blowup} and Proposition \ref{prop220} leading to
$$
\forall\, n\in\NN^\star,\quad T<T_n^\star,
$$
where $T$ is the unique solution to $\Gamma(T)=m$ and $\Gamma$ is defined in \eqref{Gamma:def***}. Notice that the constant $C_0$ in \eqref{Gamma:def***} is independent of $n$.
%
Consequently, we find that 
\begin{equation}\label{limitinf T}
{\inf_{ n \in\NN^\star} T_n^{\star} } \geqslant T > 0\quad \quad \text{and }\quad\quad  v_n,\rho_n\in C\big([0,T ],C^{1+m}\big), \;\forall n\in \mathbb{N}^\star. 
\end{equation} 
Once we get \eqref{limitinf T}, we can justify again all the a priori estimates proved in the previous propositions since the approximate solution $(v_n,\rho_n)$ is smooth enough on $[0,T].$ The next step consists then in proving uniform estimates, with respect to $n$, on the time interval $[0,T].$\\

\ding{205}\textit{Uniform boundedness of the approximated  solutions.} 
  Combining the results of  \ding{203}, Propositions \ref{Proposition inviscid}, \ref{prop22}  and \ref{prop220} we obtain the following uniform bounds in $n$, for all $t\in [0,T]$
  \begin{equation*}
 \Vert\omega_n(t)\Vert_{L^a \cap L^\infty}+\Vert  \rho_n(t)\Vert_{W^{1,a} \cap W^{1,\infty}}\leqslant C_{0,T} 
\end{equation*}
\begin{eqnarray}\label{vn-bounds}
\Vert  v_n(t)\Vert_{LL }+\Vert  v_n(t)\Vert_{L(\Sigma_{t,n})}&\leqslant &C_{0,T},
\end{eqnarray}
and
\begin{equation}\label{Cs-n:reg}
\sup_{ h \in   (0,e^{-1}]}h^{\beta_n(t)}\Big(\Vert \rho_n(t)\Vert^{s+1 }_{\Sigma_{t,n}, \mathcal{X} _{t,h, n}}+\Vert \omega_n(t)\Vert ^{s  }_{\Sigma_{t,n}, \mathcal{X}_{t,h , n}}+   \widetilde{\Vert} \mathcal{X}_{t ,h,n}\Vert_{s}\Big)\leqslant C_{0,T},
\end{equation}
where, $\beta_n(t)$ is given in accordance with \eqref{beta-def**} as 
\begin{equation}\label{beta_n}
\beta_n(t) \triangleq  \beta  + C \int_0^t W_n(\tau)d\tau 
\end{equation}
and 
\begin{equation}\label{Wn:def}
W_n(t)\triangleq\Big(\Vert  v_n(t)\Vert_{L(\Sigma_{t,n})}+\Vert \omega_n(t)\Vert_{L^a \cap L^\infty} \Big)\exp\bigg(\int_0^t\Vert v_n(\tau)\Vert_{LL}d\tau\bigg).
\end{equation}
Above, we use the notation 
\begin{equation*}
\Sigma_{t,n} = \Psi_n(t,\Sigma_0),
\end{equation*}
with
\begin{equation*}
\Psi_n(t,x)= x + \int_0^tv_n\big(\tau , \Psi_n(\tau,x)\big) d\tau,
\end{equation*}

and $\mathcal{X}_{t ,h,n}\triangleq (X_{t ,\lambda,h,n})_{\lambda\in \Lambda} $ which is given by
\begin{equation*}
X_{t,\lambda,h, n} = \partial_{X_{0,\lambda,h}} \Psi_n(t,\Psi_n^{-1}(t,\cdot)). 
\end{equation*} 
Using \eqref{vn-bounds} together with classical regularity of the flow associated to Yudovich solutions, we get uniformly with respect to $n\in \mathbb{N}^* $,
\begin{equation*}
\Psi_n \in L^\infty([0,T]; \hbox{Id}+C^{\exp(-C_{0,T})}),
\end{equation*}
while, \eqref{ixtd} implies  
\begin{equation} \label{ixtd-n}
I\big((\Sigma_{t,n})_{\delta_{t,n}^{-1}(h)},(\mathcal{X}_{t,n})_h\big)
 \geqslant   I\big((\Sigma_0)_{h},(\mathcal{X}_0)_{h}\big)h^{\int_0^t W_n(\tau)d\tau}, \quad \forall t\in [0,T],
\end{equation}
where,  we set 
\begin{equation}\label{delta-n}
 \delta_{t,n}^{-1} (h)\triangleq h^{\exp(-\int_0^t\Vert v_n(\tau)\Vert_{LL}d\tau)} .
\end{equation}  
Remark that, due to \eqref{ESS-T1}, \eqref{vn-bounds} and the uniform bounds on the initial data proved in \ding{203}, we have
$$\int_0^t W_n(\tau)d\tau \leqslant \Gamma(T) = m,\quad \forall t\in [0,T]$$
and 
$$ \delta_{t,n}^{-1} (h) \leqslant h^{\exp(-C_{0,T}) } \triangleq h^{\alpha_1}.$$
Hence, \eqref{Cs-n:reg}, \eqref{beta_n} and \eqref{ixtd-n} imply  
\begin{equation}\label{Cs-n:reg:2-----}
\sup_{ h \in   (0,e^{-1}]}h^{\beta+ Cm}\Big(\Vert \rho_n(t)\Vert^{s+1 }_{\Sigma_{t,n}, \mathcal{X} _{t,h, n}}+\Vert \omega_n(t)\Vert ^{s  }_{\Sigma_{t,n}, \mathcal{X}_{t,h , n}}+   \widetilde{\Vert} \mathcal{X}_{t ,h,n}\Vert_{s}\Big)\leqslant C_{0,T},
\end{equation}
\begin{equation} \label{ixtd-n:2-----}
h^{-(\gamma + Cm)}I\big((\Sigma_{t,n})_{h^{\alpha_1}},(\mathcal{X}_{t,n})_h\big)
 \geqslant  \inf_{h\in (0,e^{-1}]} h^{-\gamma}  I\big((\Sigma_0)_{h},(\mathcal{X}_0)_{h}\big)>0, \quad \forall t\in [0,T].
\end{equation}
In order to  proceed, we have to understand   the compactness of the sequences $( \Psi_n(t,\cdot),\mathcal{X}_{t,n})_{n\in \mathbb{N}^*} $
\ding{206}\textit{Stability estimates and convergence.} 
 A direct application of Corollary \ref{Coro:Cauchy seq.} shows that $(v_n ,\rho_n )_{n\in \mathbb{N}^*}$ is a Cauchy sequence in $L^\infty([0,T],L^2)$. More precisely, we have 
$$
\Vert v_{n}(t)-v_m(t)\Vert_{L^{2}} +\Vert \rho_{n}(t)-\rho_m(t)\Vert_{L^{2}} \leqslant {C}_{0,T}(\Vert v_{0,n}-v_{0,m}\Vert_{L^{2}} + \Vert \rho_{0,n}-\rho_{0,m}\Vert_{L^{2}})^{\exp(-{C}_{0,T})}, \quad \forall t\in [0,T],
$$
where, ${C} _{0,T}$ is a constant depending only on the size of the initial data and $T$ and it comes in particular from the a priori estimates of $(v_n,\rho_n)$. This is enough to conclude that $(v_n,\rho_n)$ is a Cauchy sequence in $L^\infty([0,T];L^2).$\\
 Mixing this strong convergence result with the uniform bounds obtained in \ding{205} allow to pass to the limit and construct a local solution $(\omega,\rho)$ to the  problem \eqref{eqn:omega}. 
  In addition, this solution  satisfies
  $$ (\omega,   \rho )\in L^\infty([0,T]; L^a \cap L^\infty)  \times  L^\infty([0,T];W^{1,a} \cap W^{1\infty}). $$
On the other hand,  since  $(v_n)_{n\in \mathbb{N}^*}$ is a Cauchy sequence in $L^\infty([0,T];L^2)$, together with the bounds of the $L^p$-norms of $\omega_n$, yields the strong convergence of $v_n$ towards $v$ in $L^\infty([0,T];L^\infty)$. This allows to get the following  strong  convergence of the flow sequence $(\Psi_n)$, for more details see for instance  \mbox{\cite[Lemma 5.5.2]{Chemin}}
\begin{equation*}
\lim_{n\rightarrow \infty}{\Psi_n} = \Psi \quad \text{in } \quad   L^\infty \Big([0,T];\hbox{Id}+ C^{\exp \left(-\varepsilon - \int_{0 }^T \| v(\tau,\cdot)\|_{LL}d\tau \right)} \Big).  
\end{equation*}
In fact, we can prove that the convergence holds in the following sense 
\begin{equation}\label{limit-flow}
\lim_{n\rightarrow \infty}{\Psi_n} = \Psi \quad \text{in } \quad  C \Big([0,T];\hbox{Id}+ C (\mathbb{R}^2)\Big).  
\end{equation}

 Consequently, we find that
$$
\lim_{n\rightarrow\infty}\sup_{t\in[0,T]\atop x\in\Sigma_0}|\Psi_n(t,x)-\Psi(t,x)|=0.
$$
Therefore, for any $h>0$, there exists $n_0\in \mathbb{N}^*$, such that for all $n\geqslant n_0$ and for all $t\in [0,T]$, 
\begin{equation*}
(\Sigma_t)_{2h}^c \subset  (\Sigma_{t,n})_{ h}^c .
\end{equation*}
Together with \eqref{vn-bounds}, we infer that for all $t\in [0,T]$ and $h\in (0,e^{-1}]$
\begin{equation*}
 \|\nabla v_n (t)\|_{L^\infty((\Sigma_t)_{2h}^c)}  \leqslant  \|\nabla v_n \|_{L^\infty((\Sigma_{t,n})_{ h}^c)}\leqslant (-\log h ) C_{0,T} . 
\end{equation*}
Hence, for any $x,y\in (\Sigma_t)_{3h}^c$, with $|x-y|\leqslant h$ we obtain that  the segment $ [x,y]\subset (\Sigma_t)_{2h}^c $. This implies by virtue of the mean value theorem   
\begin{equation*}
|v_n(x)-v_n(y)| \leqslant    (-\log h ) C_{0,T}  |x-y| .
\end{equation*}
Let $n$ go to infinity and $y$ to $x$, we end up with
\begin{equation*}
\frac{\|\nabla v\|_{L^\infty((\Sigma_t)_{3h}^c)}}{ -\log h} \leqslant C_{0,T}.
\end{equation*}
This leads to 
\begin{equation*}
\|v (t) \|_{L (\Sigma_t)  } {\leqslant} \,C_{0,T}.
\end{equation*}
On the other hand, it is easy to check that 
$$ 
\Vert v (\tau)\Vert_{LL} \leqslant \liminf_{n\rightarrow \infty} \Vert v_n(\tau)\Vert_{LL}
$$
and
$$ 
\| \omega(t)  \|_{L^a\cap L^\infty  }\leqslant \liminf_{n\rightarrow \infty} \| \omega_n(t) \|_{L^a\cap L^\infty}.
$$
The next step consists to prove that, for all $(\lambda,h)\in \Lambda \times (0,e^{-1}]$      
\begin{equation*}
\partial_{X_{0,\lambda,h}}\Psi\in L^\infty([0,T]; C^{s })
\end{equation*}
and to check   that for all $s'<s$, $h\in (0,e^{-1}] $ and uniformly with respect to $\lambda\in \Lambda   $, there holds
\begin{align}\label{Strong-Cv1}
\lim_{n\rightarrow \infty}  \partial_{X_{0,\lambda,h}}\Psi_n= \partial_{X_{0,\lambda,h}}\Psi \quad \text{in } \quad   L^\infty([0,T]; C^{s'}). 
\end{align}
{This can be done by using an interpolation inequality, the limit \eqref{limit-flow},  and the fact that the sequence $\left(\partial_{X_{0,\lambda,h}}\Psi_n\right)_{n\in \mathbb{N}^\star} $ is bounded in $L^\infty ([0,T]; C^s)$, for all $(\lambda, h)\in  \Lambda\times(0,e^{-1}]$. More precisely, one  gets
{\begin{align*}
\|\partial_{X_{0,\lambda,h}}\Psi_n- \partial_{X_{0,\lambda,h}}\Psi\|_{L^\infty_TC^{s^\prime}}&\lesssim\|\partial_{X_{0,\lambda,h}}\Psi_n- \partial_{X_{0,\lambda,h}}\Psi\|_{L^\infty_T{B_{\infty,\infty}^{-1}}}^{\delta} \|\partial_{X_{0,\lambda,h}}\Psi_n- \partial_{X_{0,\lambda,h}}\Psi\|_{L^\infty_TC^{s}} ^{1-\delta}\\
&\lesssim \big\|({X_{0,\lambda,h}},\textnormal{div}{X_{0,\lambda,h}}\big)\big\|^\delta_{L^\infty_TL^\infty}\|\Psi_n- \Psi\|_{L^\infty_T{L^\infty}}^{\delta} \sup_{n\in\N}\|\partial_{X_{0,\lambda,h}}\Psi_n\|_{L^\infty_TC^{s}} ^{1-\delta},
\end{align*}}
for some $\delta\in(0,1)$,  which is enough to obtain the desired result \ref{Strong-Cv1}.\\
Now,  Corollary \ref{cor:X-Psi:1} insures that the   vector fields
\begin{equation*}
X_{t,h,\lambda } =( \partial_{X_{0,h,\lambda}} \Psi) (t,\Psi (t,\cdot)) 
\end{equation*}
belong to $C^s$, for all $t\in [0,T]$. Furthermore, they are admissible in the sense of Definition \ref{def11}, for all $t\in [0,T]$,. This can be proved by virtue of the strong convergence of the flow \eqref{limit-flow} combined with  \eqref{ixtd-n:2-----} .
Finally,  straightforward arguments based on    \eqref{Cs-n:reg:2-----} together with the aforementioned  convergence   results for $\rho_n, \omega_n$ and $ X_{t,h,\lambda,n }$ yield  
\begin{equation*}
\sup_{ h \in   (0,e^{-1}]\atop t\in[0,T]}h^{\beta+ Cm}\Big(\Vert \rho(t)\Vert^{s+1 }_{\Sigma_{t}, \mathcal{X} _{t,h}}+\Vert \omega(t)\Vert ^{s  }_{\Sigma_{t}, \mathcal{X}_{t,h }}+   \widetilde{\Vert} \mathcal{X}_{t ,h}\Vert_{s}\Big)\leqslant C_{0,T}.
\end{equation*}
  This achieves the  estimates for the constructed solution.}\\

\ding{207}\textit{Uniqueness of the solutions.}
The proof   can be done along the  same lines as in \cite{Hassainia-Hmidi}. More precisely, it suffices to  use the stability estimates mentioned in \ding{206}.
\end{proof}
\section{Viscous case}\label{Viscous-case}
In this section, we shall examine  the partial viscous Boussinesq system written in terms of vorticity and density formulation
\begin{equation} \label{eqn:omega:2}
\left\{ \begin{array}{ll}
  \partial_t \omega_\kappa + v \cdot \nabla \omega_\kappa = \partial_1\rho_\kappa\quad (t,x)\in \mathbb{R}_+\times \mathbb{R}^2,&\vspace{2mm}\\ 
   \partial_t\rho_\kappa+v_\kappa\cdot\nabla \rho_\kappa-\kappa\Delta\rho_\kappa=0,&\vspace{2mm}\\
      v_\kappa= \nabla^\perp \Delta^{-1}\omega_\kappa,&\vspace{2mm}\\
  (\omega_\kappa,\rho_\kappa)|_{t=0}=(\omega_0,\rho_0).
  \end{array}\right.
\end{equation}
We aim to establish a local well-posedness result  for a large class of initial data allowing in particular  singular vortex patches as in the previous section. However, we want to get  estimates uniformly with respect to the vanishing diffusion parameter $\kappa\in(0,1)$, which will induce several new difficulties compared to the inviscid case.  
\subsection{General statement and consequence}
We intend to discuss a general  statement covering more cases than the soft  version detailed in  Theorem \ref{THEO:2:soft} and  restricted to the  singular vortex patches.   The main result of this section reads as follows.
 \begin{theorem}\label{th2}
 Let $(s,\varepsilon,p)\in(0,1) \times (0,\frac{s}{2}) \times (\frac{1}{\frac{s}{2}-\varepsilon}, \infty)$ and $(a,m)\in[2,\infty) \times (1,2).$ Consider  a compact negligible set of the plane  $\Sigma_0$  satisfying the integrability condition \eqref{int:cond0}.  Let $v_0$ be a divergence-free vector field with  vorticity  $\omega_ 0\in  L^{a}\cap L^\infty$   and   $\rho_0$ be a real-valued function in $W^{2,a}(\mathbb{R}^2)\cap W^{2,\infty}(\mathbb{R}^2)\cap C^{1+m}((\Sigma_0) _r) $, for some $r>0$ such that $\nabla \rho_0$ and ${\nabla ^2 \rho_0}$ are vanishing on $\Sigma_0.$  Consider  a   family of vector fields $\mathcal{X}_{0 }=(X_{0,\lambda,h })_{(\lambda,h)\in \Lambda\times (0,e^{-1}]}$ of class $B^s_{p,\infty}$ as well as their divergences and suppose that this family is  $\Sigma_0$-admissible of order $\Xi=(\alpha, \beta, \gamma)$ with
 \begin{equation*}
 \sup_{h\in (0,e^{-1}]} h^{\beta}\Big(\Vert\omega_0\Vert_{(\Sigma_0)_{h},(\mathcal{X}_0)}^{s,p}+ \Vert\rho_0\Vert_{(\Sigma_0)_{ h },(\mathcal{X}_0)}^{s+1,p} \Big)<\infty.
 \end{equation*} 
 Then, there exists $T>0$ independent of $\kappa\in(0,1)$ such that the viscous Boussinesq system \eqref{eqn:omega:2} has a unique solution
 $$
(\omega_\kappa,\rho_\kappa)\in L^\infty\big([0,T];L^a\cap L^\infty\big)\times L^\infty\big([0,T]; W^{1,a}\cap W^{1,\infty}\big) 
 $$  
 with uniform bounds on $\kappa\in(0,1)$ and 
satisfying in addition
$$
 \sup_{\underset{\kappa\in (0,1)}{ h\in(0,e^{-1}]}}\frac{\Vert \nabla v_\kappa(t)\Vert_{L^\infty((\Sigma_{t,\kappa})_h^c)}}{-\log h}\in L^\infty([0,T]).$$
Furthermore, 
there exists $C_{0,T}>0$, depending only on the initial data, $\Sigma_0$ and $T$, and there exists a triplet of  constants  $\Xi=(\alpha_1,\beta_1,\gamma_1)$, such that the following results  hold.
By denoting
$$\sigma_t = s \big( 1- \tfrac{\varepsilon t}{sT}\big) - \varepsilon,$$
we get for any $t\in[0,T]$
\begin{equation*} 
 \sup_{\underset{\kappa\in (0,1)}{ h\in(0,e^{-1}]}} h^{\beta_1}\Vert\omega_\kappa(t)\Vert_{(\Sigma_t)_{{h^{\alpha_1}}},(\mathcal{X}_{t,h,\kappa})}^{\sigma_t,p}+  \sup_{\underset{\kappa\in (0,1)}{ h\in(0,e^{-1}]}} h^{\beta_1}\Vert\rho_\kappa(t)\Vert_{(\Sigma_{t,\kappa})_{{h^{\alpha_1}}},(\mathcal{X}_{t,h,\kappa})}^{\sigma_{t,\kappa}+1,p}    \leqslant  C_{0,T}.
\end{equation*}  
Here, $\Sigma_{t,\kappa}$ is given by 
\begin{equation*}
  \Sigma_{t,\kappa} \triangleq \Psi_\kappa (t, \Sigma_0) 
\end{equation*}
 and the transported vector field  $\mathcal{X} _{t,\kappa}$ of $\mathcal{X} _{0 }$ by the flow $\Psi_{\kappa}(t)$,  defined by
$$X_{t,\lambda, h,\kappa}  \triangleq  \big(\partial_{X_{0,\lambda,h}}\Psi_\kappa\big)(t,\Psi^{-1}_\kappa (t,\cdot)), $$
is $\Sigma_t$--admissible of order $\Xi$ and belongs to $  B^{\sigma_t}_{p,\infty}$, for all $t\in [0,T], (\lambda,h) \in \Lambda\times (0,e^{-1}] $, uniformly with respect to $\kappa$. In addition, $X_{t,\lambda, h,\kappa} (\Psi _\kappa (t,\cdot))  $ belongs to  $  C^{\sigma_t-\frac{2}{p} }    $.
\end{theorem} 
\begin{remark}
 For the sake of simplicity, we shall drop the index $\kappa$ from the system \eqref{eqn:omega:2}.
\end{remark}
Before giving the proof of this general result, we shall sketch the proof of Theorem \ref{THEO:2:soft}.
\begin{proof}[Proof of Theorem $\ref{THEO:2:soft}$]  The proof in the viscous case differs slightly from the inviscid one, because we have to deal with Besov spaces. We shall follow the same lines as in  the inviscid case. First, we assume that the boundary of the initial patch is smooth except at some closed singular set, where the singularity is algebraic. More precisely, we assume that  
$$
\partial \Omega_0 = \big\{ x\in V:  f_0(x) =0 \big\},
$$
for a small neighborhood $V$ of $\partial \Omega_0$, and for a function $f_0 $ in $C^{1+s}$ satisfying, for some $\widetilde{\gamma}>0$
\begin{equation}\label{Non-deg11}
|\nabla f_0(x)|\geqslant  d(x,\Sigma_0)^{{\gamma} }, \quad \forall x\in V.
\end{equation}
 Without loss of generality, we may  assume that $f_0$ is compactly supported. Let $(\theta_h)_{h\in(0,e^{-1}]}$ be a family of functions of class $C^\infty_c$, with $\theta_h\equiv1$ in $(\Sigma_0)_h^{c}$ and supported in $(\Sigma_0)_{h^{\alpha }}^{c}$, for some $\alpha>1$, which satisfies for all positive real number $\tilde{s}$
$$
\|\theta_h\|_{C^{\tilde{s}}}\leqslant C_{\tilde{s}} h^{-\tilde{s}}.
$$ 
Let  $V_1$ be an open set containing $\partial \Omega_0$ and its closure set $\overline{V}_1$ is contained in $V$, that is $\overline{V}_1\subset V$.  Consider  a $C^\infty$ function $\widetilde{\theta}$ with $\widetilde{\theta}\equiv 1$ in $V_1$ and  $\supp\, \widetilde{\theta}\subset V$. Define  the family $\mathcal{X}_{0,h}\triangleq (X_{0,0,h},X_{0,1,h})$ defined through
$$
X_{0,0,h}=\nabla^\perp(\theta_h f_0),\quad X_{0, 1,h}=\theta_h(1-\widetilde{\theta})\begin{pmatrix}1\\
  0 \end{pmatrix}.
$$
From this construction together with the  assumption \eqref{Non-deg11} we obtain 
\begin{eqnarray}\label{NNNNNN1}
\inf_{h\in(0,e^{-1}]}h^{-\gamma}I\big((\Sigma_0)_h,\mathcal{X}_{0,h}\big)>0.
\end{eqnarray} 
Then, due to Lemma \ref{inverse-embedding:lemma} and the fact that $f_0$ is compactly supported, we deduce that for $h\in(0,e^{-1}]$ 
\begin{equation}\label{Besov:es:0}
\widetilde{ \|}X_{0,0,h}\|_{B^{s}_{p,\infty}}\leqslant C\widetilde{ \|}X_{0,0,h}\|_{C^s}, 
\end{equation}
where the constant $C$ depends only on the support of $f_0$. This asserts that $X_{0,0,h},\Div X_{0,0,h}\in B^{s}_{p,\infty}$. Unfortunately, the vector field $X_{0, 1,h}$ belongs to $C^s$ for all $s\in\RR$, but not to $B^{s}_{p,\infty}$. To remedy this defect, we shall make use of the partition of the unity given by Lemma \ref{lemma:partition of unit}. From the vector field $X_{0, 1,h}$, we construct a countable family $ \widetilde{\mathcal{X}} _{0,h }\triangleq(\widetilde{X} _{0,n,h})_{n\in \mathbb{N}^\star}$ to be defined as
\begin{equation*}
\widetilde{X}_{0,n,h}(x)=\psi_n(x)X_{0,1,h}(x),\quad \forall (n,x)\in \mathbb{N}^\star\times\RR^2,
\end{equation*}
where the functions $\psi_n$ are introduced  in Lemma \ref{lemma:partition of unit}. In particular, for each $n\in \mathbb{N}^\star$, $\psi_n$ is compactly supported in $\mathcal{O}_n$, with $|\mathcal{O}_n|\lesssim  1$. By the  product laws stated in Lemma \ref{para--products}, we have 
\begin{equation*}
\widetilde{\|} \widetilde{X} _{0,n,h}\|_{B^s_{p,\infty}} \lesssim \big(\|\psi_n \|_{B^s_{p,\infty}}+ \|\nabla \psi_n \|_{B^s_{p,\infty}} \big)\widetilde{\|}X _{0,1,h} \|_{C^s}, \quad \forall n\in \NN^\star.
\end{equation*}
On the other hand, by {Sobolev embeddings} and \eqref{psin-EST}, we may write
\begin{equation*}
\|\psi_n \|_{B^s_{p,\infty}}+ \|\nabla \psi_n \|_{B^s_{p,\infty}} \lesssim  \|  \psi_n \|_{L^p}  + \| \nabla  \psi_n \|_{L^p} + \| \nabla^2  \psi_n \|_{L^p} \lesssim 1, \quad \forall n\in \mathbb{N}^\star.
\end{equation*} 
 
Thus, we end up with 
\begin{equation}\label{Besov:es}
\widetilde{\|}  \widetilde{X} _{0,n,h}\|_{B^s_{p,\infty}} \lesssim  \widetilde{\|} X  _{0,1,h} \|_{C^s}  , \quad \forall n\in \NN^\star.
\end{equation}
Now, by denoting $\widetilde{\mathcal{X}}_{0,0,h}\triangleq X_{0,0,h} $, then, our next task is to show that the family $\widetilde{\mathcal{X}}_{0,h} = (\widetilde{X} _{0,n,h})_{n\in \mathbb{N}}$ is $\Sigma_0$- admissible of order $ \Xi $ in the sense of Definition \ref{def11}, for some $\Xi =(\alpha ,\beta ,\gamma )$. Indeed, an elementary computation yields  
\begin{eqnarray}\label{admissibility:000}
I((\Sigma_0)_h,\widetilde{\mathcal{X}}_{0,h })&=&\inf_{x\in (\Sigma_0)_h^c}  \sup_{n\in\NN }|\widetilde{ X}_{0,n,h}(x)|   \nonumber\\
&=& \inf_{x\in (\Sigma_0)_h^c} \max\big\{|X _{0,0,h}(x)| \;;\;\sup_{n\in\NN ^\star} \psi_n(x)| X _{0,1,h}(x)| \big\} .
\end{eqnarray}
Meanwhile, for all $x\in \mathbb{R}^2$, Lemma \ref{lemma:partition of unit} gives
\begin{eqnarray*} 
1=\sum_{n\in\NN^\star}\psi_n(x)&=&\sum_{\{n\in\NN^\star: x\in\mathscr{O}_n\}}\psi_n(x)\\
&\leqslant &N_d\sup_{n\in\NN^\star}\psi_n(x).
\end{eqnarray*} 
 Consequently, 
 $$\sup_{n\in\NN^\star}\psi_n(x)\ge\frac{1}{N_d},\quad \forall x\in\RR^2.$$
   Plugging this inequality in \eqref{admissibility:000} yields
\begin{eqnarray*}
I((\Sigma_0)_h,\widetilde{\mathcal{X}}_{0,h })\geqslant  \frac{1}{N_d}I\big((\Sigma_0)_h,\mathcal{X}_{0,h}\big).
\end{eqnarray*} 
Combining this estimate with \eqref{Non-deg11} allows to get 
\begin{equation*}
\inf_{h\in (0,e^{-1}]} h^{-\gamma} I((\Sigma_0)_h,\widetilde{\mathcal{X}}_{0,h })>0.
\end{equation*}
Finally, to achieve the proof of the admissibility of the family $\widetilde{\mathcal{X}}_{0,h}$, as stated in the Definition \ref{def11},  we need to show the existence of  a positive number $\beta>0$ such that
 \begin{equation}\label{claim:N_p}
 \sup_{h\in (0,e^{-1}]}h^{\beta }N_{s,p}((\Sigma_0)_h,\widetilde{\mathcal{X}}_{0,h}) <\infty.
 \end{equation}
Indeed, by the definition of $ N_{s,p}$ and \eqref{Besov:es}, we have 
\begin{align*}
  N_{s,p}((\Sigma_0)_h,\widetilde{\mathcal{X}}_{0,h}) = \sup_{n\in \mathbb{N}} \frac{ \widetilde{\|}  \widetilde{X} _{0,n,h}\|_{B^s_{p,\infty}} }{ I((\Sigma_0)_h,\widetilde{\mathcal{X}}_{0,h})} &\leqslant C\max_{j\in \{0,1\}} \frac{ \widetilde{\|}  X _{0,j,h}\|_{C^s} }{ I\big((\Sigma_0)_h,{\mathcal{X}}_{0,h}\big)}\\
  & \leqslant CN_{s }\big((\Sigma_0)_h,\mathcal{X}_{0,h}\big).
\end{align*}
Similar arguments as in   \cite{Chemin, Hassainia-Hmidi} show that   
\begin{align}\label{Haroune-mardi}
\sup_{h\in (0,e^{-1}]}h^{\beta }N_{s }\big((\Sigma_0)_h,\mathcal{X}_{0,h}\big)<\infty,
\end{align} 
for some index $\beta>0$,  leading to the  claim \eqref{claim:N_p}. All in all, we conclude that the family $\widetilde{\mathcal{X}}_{0,h}$ is $\Sigma_0$--admissible of   order $\Xi =(\alpha ,\beta ,\gamma )$. It remains to check that
\begin{equation*}
 \sup_{h\in (0,e^{-1}]} h^{\beta }\Big(\Vert\omega_0\Vert_{(\Sigma_0)_{h},(\widetilde{\mathcal{X}}_{0,h})}^{s,p}+ \Vert\rho_0\Vert_{(\Sigma_0)_{ h },(\widetilde{\mathcal{X}}_{0,h})}^{s+1,p} \Big)<\infty  .
 \end{equation*} 
 A direct computation shows that
$$\partial_{\widetilde{X}_{0,n,h}} \omega_0 = 0, \quad \forall n\in \NN,
 $$
which obviously gives in view of \eqref{Haroune-mardi}
\begin{equation*}
 \sup_{h\in (0,e^{-1}]} h^{\beta }\Vert\omega_0\Vert_{(\Sigma_0)_{h},(\widetilde{\mathcal{X}}_{0,h})}^{s,p}<\infty  .
 \end{equation*}
As to the estimate of the density, we apply  Lemma \ref{inverse-embedding:lemma} together with \eqref{Besov:es:0}, \eqref{Besov:es} and \eqref{Haroune-mardi} 
\begin{eqnarray*}
h^{\beta } \sup_{n\in\mathbb{N}}\frac{\|\partial_{\widetilde{X}_{0,n,h}} \rho_0  \|_{B^{s}_{p,\infty}}}{I((\Sigma_0)_h,\widetilde{\mathcal{X}}_{0,h})} \leqslant h^{ \beta }N_s((\Sigma_0)_h,\widetilde{\mathcal{X}}_{0,h}) \|\rho_0 \|_{C^{s + 1}}.
\end{eqnarray*}
This ends the proof of the desired estimate.
Therefore, the assumptions of Theorem \ref{th2} are completely  fulfilled and as a consequence we get  the local well-posedness  in Theorem \ref{THEO:2:soft}. Besides, Theorem \ref{th2} asserts that, for all $h\in (0,e^{-1}]$ and for all $t\in [0,T]$,
$$X_{t,0,h}( \Psi (t,\cdot))   \in  C^{\sigma_t-\frac{2}{p}}(\RR^2) \hookrightarrow C^{s-2\varepsilon-\frac2p}(\RR^2).  
$$ 
By taking $\varepsilon$ close to $0$ and $p$ large enough we obtain $X_{t,0,h}( \Psi (t,\cdot))   \in  C^{s^\prime}(\RR^2) $ with $s^\prime<s$ sufficiently close to $s$.
  Then the proof of the regularity persistence of the boundary $\Psi(t,\partial\Omega_0\backslash \Sigma_0)$  follows from the same arguments detailed  in  the inviscid case, see Section \ref{Inviscid-case}. This ends the proof of Theorem \ref{THEO:2:soft} provided that Theorem \ref{th2} is true.
\end{proof}
\subsection{A priori estimates} As explained in the introduction,  the viscous case is more delicate compared to the inviscid one and requires intricate analysis. The major drawback arises from the fact that the local information, such as the platitude assumption \eqref{hypothesis1}, will be diffused by the  dissipation mechanism. To elucidate the competition between the advection and the dissipation we need  suitable ansatz showing up different action regimes. 
\subsubsection{Reduction to transport-diffusion equation with a singular  potential }
 We shall examine weak $L^p$-estimates  for  $\nabla \rho$ and $\omega$ with uniform estimates on the conductivity when the initial vorticity is singular.  This part is  delicate and requires an appropriate ansatz that we should handle through suitable estimates on    transport-diffusion equation with logarithmic singular potential. The main  statement reads as follows.
\begin{proposition}\label{Proposition visous}
Let $\Sigma_0$ be a compact negligible set in $\mathbb{R}^2$ satisfying, for some $\sigma>0$ and for all small $\epsilon \in (0,1),$
 \begin{equation}\label{int:cond0}
\int_{ d(x,\Sigma_0)<\epsilon} \log^4\big(d(x,\Sigma_0)\big)  dx \lesssim \epsilon^{\sigma}.  
\end{equation}   
Let  $\omega_0\in L^2\cap L^\infty$,  $m\in(1, 2) $ and  $\rho_0$ be a scalar function satisfying  
$$
\nabla \rho_0, \nabla ^2 \rho_0 \in L^2 \cap L^\infty, 
$$
supplemented with the platitude conditions
\begin{equation}\label{hypothesis2}
 \sup_{x\notin \Sigma_0} (\varphi_0(x))^{-m} \left|\nabla \rho_0(x) \right|<\infty, \quad \quad \sup_{x\notin \Sigma_0} (\varphi_0(x))^{-m+1} \left|\nabla |\nabla\rho_0(x)|^2 \right|<\infty,
\end{equation}
where $\varphi_0$ is given by \eqref{tool:varphi0}. Consider a smooth maximal solution  $(\omega,\rho)$ to \eqref{eqn:omega:2} defined on $[0,T^\star)$ and introduce the function $t\in[0,T^\star)\mapsto \mathbb{V} (t)$  
\begin{equation}\label{V:def:0}
\mathbb{V} (t) \triangleq   \int_0^t \|v(\tau)\|_{LL\cap L(\Sigma_{\tau})}d\tau.
\end{equation} 
 There exists $C_0>0$ depending only on the initial data and $\Sigma_0$ and there exists $ C_{\sigma,m}>0$ depending only on $\sigma$ and $m$ such that the following holds.\\ If  ${\mathbb{V} (T)}\leqslant C_{\sigma,m}$, for some $T\in(0,T^\star)$, then for all $p\in [2,\infty]$, we have
 
\begin{equation}\label{es:G-W}
\forall \, t\in [0,T],\quad \|\nabla \rho(t) \|_{L^p}+\|\omega(t) \|_{L^p} \leqslant G(t),
\end{equation}  
with
\begin{equation}\label{G:def}
G(t) \triangleq	 C_0 (1+t)^3 \left(  1 +   \| A  \|_{L^6_t} \exp \left( C\| A  \|_{L^2_t}^2\right)\right) 
\end{equation}  
 and 
\begin{equation}\label{A:def:*}
A(t)\triangleq \|v(t)\|_{L(\Sigma_t)}e^{\int_0^t\|v(\tau)\|_{LL}d\tau}.
\end{equation}  
In addition, if $v_0 \triangleq \nabla ^\perp \Delta^{-1} \omega_0 \in L^\infty$ then 
\begin{equation}\label{Est-vLinfty}
\forall \, t\in [0,T],\quad\|v(t) \|_{L^\infty} \leqslant C G(t) e^{t G(t)}.
\end{equation}
Moreover,  we have the blowup criterion,
\begin{equation}\label{Visc-BLow}
T^\star<\infty\Longrightarrow \lim_{t\to T^\star}{\mathbb{V} (t)}> C_{\sigma,m}.
\end{equation}

\end{proposition} 

\begin{proof} Without loss of generality, we can assume that  $\rho_0$ is not a constant function. Otherwise the system \eqref{eqn:omega:2} reduces to Euler equations where the global regularity persistence is proved in \cite{Chemin}.   In the first step, we shall proceed as in Proposition \ref{Proposition inviscid}. Actually,  an elementary calculus allows us to check that 
\begin{align*}
 \big(\partial_t+ v\cdot\nabla-\kappa\Delta \big)|\nabla\rho|^2+2{\kappa}\sum_{i=1}^2|\nabla\partial_i\rho|^2&=-2\sum_{i=1}^2(\partial_iv)\cdot\nabla\rho\, \partial_i\rho\\
 &\leqslant 2|\nabla v||\nabla \rho|^2.
\end{align*}
Hence, in view of \eqref{Lsigma}, the function   $\bar{\varrho}\triangleq |\nabla\rho|^2$ satisfies the inequality 
\begin{align}\label{Eq-varrho}
 \big(\partial_t+ v\cdot\nabla-\kappa\Delta \big)\bar{\varrho}(t,x)\leqslant 2\|{v}(t)\|_{L(\Sigma_t)} \ln^+(d(x,\Sigma_t))\,\bar{\varrho}(t,x).
\end{align}
Now, we introduce the splitting 
$$
\bar{\varrho}(t,x)=\phi(t,x)\,\eta(t,x),
$$
with $\phi$ and $\eta$ are solution and sub-solution of  the following problems
\begin{equation} \label{Eq1}
\left\{ \begin{array}{ll}
   \big(\partial_t+ v\cdot\nabla-\kappa\Delta \big)\phi(t,x)= 2\|{v}(t)\|_{L(\Sigma_t)} \ln^+(d(x,\Sigma_t))\,\phi(t,x),&\vspace{2mm}\\ 
  \phi|_{t=0}=|\nabla\rho_0|^2 
  \end{array}\right.
\end{equation}
and
\begin{equation} \label{Eq2}
\left\{ \begin{array}{ll}
   \big(\partial_t+ \big(v-2\kappa\nabla \log(\phi)\big)\cdot\nabla-\kappa\Delta \big)\eta(t,x)\leqslant 0,&\vspace{2mm}\\ 
  \eta|_{t=0}=1.
  \end{array}\right.
\end{equation}
Note that by construction  $\bar{\varrho},\phi\geqslant 0$, implying in particular that $\eta\geqslant0,$  and the decomposition of $\bar{\varrho}$ is justified by the fact that 
\begin{equation}\label{claim000:phi}
\phi(t,\cdot)>0, \quad \forall t>0. 
\end{equation} 
The proof of \eqref{claim000:phi} can be checked using the fundamental solution for the convection-diffusion equation. In fact, according to \cite[Section 3]{Gallay000}, see, also \cite{Osada}, any smooth solution  $\phi$ to \eqref{Eq1} can be represented through 
\begin{equation}\label{IInt-rep1}
\phi(t,x) = \int_{\mathbb{R}^2}\Gamma_\kappa(t,x;0,y)  \phi_0(y) dy + \int_0^t \int_{\mathbb{R}^2}\Gamma _\kappa(t,x;s,y) f(s,y)\phi(s,y) dyds,
\end{equation}
where 
$$
f(t,x) = 2\|{v}(t)\|_{L(\Sigma_t)} \ln^+(d(x,\Sigma_t))\geqslant 0, \quad \phi_0 = |\nabla\rho_0|^2
$$
and $\Gamma_\kappa( t,x;t',x' )$ is the unique solution of
\begin{equation*}
 \left\{ \begin{array}{ll}
\left(\partial_t + v\cdot \nabla_x - \kappa \Delta_x \right) \Gamma_\kappa = 0, & (x,x')\in  \mathbb{R}^{4}, \quad 0\leqslant t'<t,\vspace{2mm}\\
\Gamma_\kappa (t',x;t',x') =\delta(x-x').
\end{array} \right.
\end{equation*} 
We also note that, for any $0\leqslant t' < t,$ the function $(x,x')\mapsto \Gamma_\kappa(t,x;t',x')$ is continuous and satisfies 
\begin{equation}\label{positivity:FS}
\Gamma_\kappa(t,x;t',x')>0.
\end{equation}  
We recall that both functions $\phi_0$ and $f\phi$ are positive. Hence, if $\phi(t_0,x_0)$ vanishes for some point  \mbox{$(t_0,x_0)\in (0,T] \times \mathbb{R}^2$}, then thanks to the integral representation and  \eqref{IInt-rep1}, this would imply that
$$\phi_0(y) = 0,\quad \forall y\in \mathbb{R}^2.$$
This contradicts the fact that the initial density is non-constant.  All in all, we deduce our claim \eqref{claim000:phi}. Our next claim is to show that 
\begin{equation}\label{claim000:eta}
\|\eta(t,\cdot) \|_{L^\infty} \leqslant 1 ,\quad  \forall t\in [0,T^\star).  
\end{equation}
To do so, we rewrite \eqref{Eq2} as
\begin{equation} \label{Eq2-F***}
\left\{ \begin{array}{ll}
   \big(\partial_t+  (v-   V_\kappa ) \cdot\nabla-\kappa\Delta \big)\eta  = -F  ,&\vspace{2mm}\\ 
  \eta|_{t=0}=1,
  \end{array}\right.
\end{equation}
where $F$ is a positive smooth function and $V_\kappa \triangleq 2 \kappa \nabla \log(\varphi)$. Let $\Phi$ be a smooth and  non--negative function satisfying the following
\begin{itemize}
\item { $\Phi$ is integrable, that is, $\displaystyle{\int_{\RR^2}\Phi(x)dx<\infty}.$} 
\item $ \Phi(x) \geqslant 1,$ for all $|x|\leqslant 1.$
\item There exists a constant $C>0$, such that 
\begin{equation}\label{grad-Phi}
\frac{\left| \nabla \Phi(x)\right|}{\Phi(x)}  \leqslant C, \quad \forall x\in \mathbb{R}^2.
\end{equation}
\end{itemize}
Define for $n\in \mathbb{N}^\star,$  
$$
\Phi_n(x) \triangleq \Phi\left(\frac{x}{n}\right).  
$$
Then,  multiplying \eqref{Eq2-F***} by $\eta^{p-1}\Phi_n$, for $p\geqslant 2$ and  integrating by parts over $\mathbb{R}^2$ we deduce in view of  the divergence-free condition of $v$
\begin{eqnarray}\label{eta-phi:n}
\frac{1}{p} \frac{d}{dt} \int_{\mathbb{R}^2}\eta ^p \Phi_n dx + (p-1)\kappa \int_{\mathbb{R}^2} |\nabla \eta|^2  \eta ^{p-2} \Phi_n dx &\leqslant  & \frac{1}{p} \int_{\mathbb{R}^2}| \Div V_k| \eta ^{p}    \Phi_n dx +\frac{1}{p} \int_{\mathbb{R}^2}| v - V_k| \eta ^{p} \left|\nabla \Phi_n \right|dx \nonumber\\ &&+\kappa\int_{\mathbb{R}^2} |\nabla \eta | \eta ^{p-1} \left|\nabla \Phi_n \right|dx.
\end{eqnarray}
Hence,  Cauchy-Schwarz inequality combined with \eqref{grad-Phi} and Young inequality provide
\begin{eqnarray*}
\int_{\mathbb{R}^2} |\nabla \eta | \eta ^{p-1} \left|\nabla \Phi_n \right|dx & \leqslant & \frac{C}{n}\left(\int_{\mathbb{R}^2}|\nabla \eta|^2  \eta ^{p-2} \Phi_n dx \right) ^\frac{1}{2} \left(\int_{\mathbb{R}^2} \eta ^p \Phi_n dx\right) ^\frac{1}{2}\\
& \leqslant & (p-1)\int_{\mathbb{R}^2}|\nabla \eta|^2  \eta ^{p-2} \Phi_n dx + \frac{C^2}{n^2(p-1)} \int_{\mathbb{R}^2} \eta ^p \Phi_n dx.
\end{eqnarray*}
Plugging this inequality into \eqref{eta-phi:n} and exploiting again \eqref{grad-Phi}, we obtain
\begin{equation*}
\frac{d}{dt} \int_{\mathbb{R}^2}  \eta ^p \Phi_n dx\leqslant \Xi_n(t) \int_{\mathbb{R}^2}\eta ^p \Phi_n dx,
\end{equation*}
with, $$\Xi_n(t) \triangleq \| \Div V_k(t) \|_{L^\infty(\mathbb{R}^2)} + \frac{C}{n}\| v(t)- V_k (t) \|_{L^\infty(\mathbb{R}^2)} + \frac{2\kappa C^2}{n^2}\cdot$$
Therefore, Gronwall lemma together with the estimate $\Phi_n(x)\geqslant 1$ for all $ |x|\leqslant n,$ imply
\begin{equation*}
\left(\int_{|x|\leqslant n}\eta^p(t,x) dx\right) ^\frac{1}{p} \leqslant \left(\int_{\mathbb{R}^2}  \eta^p(t,x) \Phi_n dx\right)^\frac{1}{p} \leqslant \left(\int_{\mathbb{R}^2}\Phi_n dx \right)^\frac{1}{p}\exp \left( \frac{1}{p}\int_0^t\Xi_n  (\tau) d\tau \right) .
\end{equation*}
Consequently, by letting  $p$ go to infinity allows to get for all $n\in \mathbb{N}^\star$
\begin{equation*}
0\leqslant \sup_{|x|\leqslant n} \eta (t,x) \leqslant 1.
\end{equation*}
Finally, taking  $n$ to $\infty$ yields \eqref{claim000:eta}. 
Let us now  estimate in  $L^p$ the solution  $\phi$ to  \eqref{Eq1}. For this aim, we introduce the decomposition
$$
\phi=\bar\varphi+\psi,
$$
with 
\begin{equation} \label{Eq3}
\left\{ \begin{array}{ll}
   \big(\partial_t+ v\cdot\nabla\big)\bar\varphi(t,x)= 2\|{v}(t)\|_{L(\Sigma_t)} \ln^+(d(x,\Sigma_t))\,\bar\varphi(t,x),&\vspace{2mm}\\ 
 \bar \varphi|_{t=0}=|\nabla\rho_0|^2
  \end{array}\right.
\end{equation}
and
\begin{equation} \label{Eq4}
\left\{ \begin{array}{ll}
   \big(\partial_t+ v\cdot\nabla-\kappa\Delta\big)\psi(t,x)= 2\|{v}(t)\|_{L(\Sigma_t)} \ln^+(d(x,\Sigma_t))\,\psi(t,x)+\kappa\Delta\bar\varphi,&\vspace{2mm}\\ 
  \psi|_{t=0}=0.
  \end{array}\right.
\end{equation}
Remark that, according to \eqref{Eq3}, the quantity $\bar\varphi $ satisfies in particular \eqref{equa-theta} treated in details in the inviscid case and the estimates  \eqref{P:es:INV:1} and \eqref{P:es:INV:2} can be applied.  In particular, if $t\in[0,T^\star)\mapsto m(t)$ is a function satisfying 
\begin{equation}\label{m-function1}
\forall t\in[0,T^\star),\quad m^\prime(t)+ \|{v}(t)\|_{L(\Sigma_t)} e^{\int_0^t\Vert v(\tau)\Vert_{LL}d\tau}\leqslant 0.
\end{equation}  
then we get for any $p\in [2,\infty]$, for any $t\in\mathtt{I},$
\begin{equation*}
\quad \|\bar\varphi(t)\|_{L^p}\leqslant C_p,
\end{equation*}
\begin{equation}\label{phi-estimate one}
\| \big(\varphi_t(\cdot)\big)^{-m(t)}\bar\varphi(t,\cdot) \|_{L^p(\mathbb{R}^2)} \leqslant C_p,
\end{equation}
where $\varphi_t$ is given by \eqref{phi:t:def} and 
$$
\mathtt{I}\triangleq \big\{t\in[0,T^\star), m(t)\geqslant0\big\},\quad C_p \triangleq \|  \varphi_0 ^{-m}(\cdot) \nabla\rho_0 (\cdot)  \|_{L^p} 
$$
 We point out that  $m$ is a continuous non-increasing function, which implies that the set $\mathtt{I}$ is an interval. 
On the other hand, the  inequality \eqref{m-function1} occurs  if we impose the strong condition
\begin{equation}\label{condition-one}
\forall \, t\in[0,T^\star), \quad m^\prime(t)+ \mathbb{V}^\prime(t)e^{\mathbb{V}(t)}\leqslant 0\quad\hbox{with}\quad  \mathbb{V}(t)= \int_0^t\|{v}(\tau)\|_{LL\cap L(\Sigma_\tau)}d\tau.
\end{equation}  
Let us now focus on the estimate of the $L^p$-norm of $\nabla \bar\varphi.$ Employing the following fact, $$|\nabla v(t,x)| \leqslant \|v(t) \| _{L(\Sigma_t)} \ln^+ (d(x,\Sigma_t)), $$
which is a consequence of \eqref{Lsigma} and differentiating in $x$
  the equation \eqref{Eq3} we infer after some  elementary computations 
\begin{align}\label{AAA0AAA}
 \big(\partial_t+ v\cdot\nabla  \big)|\nabla\bar\varphi|^2 &\leqslant 6\|{v}(t)\|_{L(\Sigma_t)} \ln^+(d(x,\Sigma_t))\,|\nabla \bar\varphi |^2 + 4M (t,x)   \left| \nabla  \bar\varphi\right|\frac{\bar\varphi}{d(x,\Sigma_t)},
\end{align}
where, we set
$$
M(t,x) = \left\{ \begin{array}{ll}
 |\nabla  d(x,\Sigma_t)| & \text{ if } d(x,\Sigma_t) < e^{-1},\vspace{2mm}\\
0 & \text{ if } d(x,\Sigma_t) \geqslant  e^{-1}.
\end{array}\right. 
$$
In order to estimate the term $\frac{1} {d(x,\Sigma_t)}, $ we apply Lemma \ref{s2lem1} with fixed  $h$ given by 
$$
h = \min \big\{d(\Psi^{-1}(t,x),\Sigma_0) , e^{-1}\big\} = \varphi_t(x)
$$
getting 
$$ 
d(\Psi^{-1}(t,x),\Sigma_0) \geqslant h \Longrightarrow d(x,\Sigma_t) \geqslant h ^{ e^{\int_0^t \|v(\tau)\|_{LL}d\tau} } = \left(\varphi_t(x)\right)^{  e^{\int_0^t \|v(\tau)\|_{LL}d\tau} }.
$$
Thus, we get
$$\frac{1}{d(x,\Sigma_t)}  \leqslant \left(\varphi_t(x)\right)^{- e^{\int_0^t \|v(\tau)\|_{LL}d\tau} } ,\quad \forall x\notin \Sigma_t $$
and \eqref{AAA0AAA} guides to
\begin{align}\label{BBB0}
 \nonumber\big(\partial_t+ v\cdot\nabla  \big)|\nabla\bar\varphi|^2 \leqslant& 6\|{v}(t)\|_{L(\Sigma_t)} \ln^+(d(x,\Sigma_t))\,|\nabla \bar\varphi |^2 \\
 &+ 4  M (t,x)  \left| \nabla  \bar\varphi\right| \bar\varphi  \left(\varphi_t(x)\right)^{- e^{\int_0^t \|v(\tau)\|_{LL}d\tau} }.
\end{align}
Let $t\in[0,T^\star)\mapsto \gamma(t)$ be a $C^1$ function, to be chosen later,  and introduce
\begin{equation}\label{VARRHO:DEF}
\widetilde{\varrho}(t,x)\triangleq \big(\varphi_t(x)\big)^{-\gamma(t)}|\nabla\bar\varphi(t,x)|^2. 
\end{equation}
One can scrutinize from \eqref{phi:t:def} that
$$
\big(\partial_t+ v\cdot\nabla \big)\big(\varphi_t(x)\big)^{-\gamma(t)}=-\gamma^\prime(t)\ln\big(\varphi_t(x)\big)\big(\varphi_t(x)\big)^{-\gamma(t)}.
$$
Therefore, in sight of \eqref{BBB0}, it holds 
\begin{eqnarray}\label{F-I1-2}
 \big(\partial_t+ {v}\cdot\nabla \big)\widetilde{\varrho} &\leqslant & \Big(\gamma^\prime(t)\ln^+\big(\varphi_t(x)\big)+6\|{v}(t)\|_{L(\Sigma_t)} \ln^+(d(x,\Sigma_t))\Big)\widetilde{\varrho} \\
\nonumber && +  4  M (t,x)  \sqrt{\widetilde{\varrho}}\; \bar\varphi \big(\varphi_t(x)\big)^{{-\frac{\gamma(t)  }{2} } - e^{\int_0^t \|v(\tau)\|_{LL}d\tau}} .
\end{eqnarray}
To handle the first part on the r.h.s. above we follow the same lines as in the inviscid case. Actually, by using the fact $\ln^+(\varphi_t(x)) = \ln^+\big(d\big(\Psi^{-1}(t,x),\Sigma_0\big)\big)$ combined with   Lemma \ref{elementary propertiy LN 2} we find 
\begin{align}\label{comparaison11}
\ln^+\big(d\big(x,\Sigma_t\big)\big)&\leqslant \ln^+\big(d\big(\Psi^{-1}(t,x),\Sigma_0\big)\big)\, e^{\int_0^t\Vert v(\tau)\Vert_{LL}d\tau}.
\end{align}
Inserting this estimate into \eqref{F-I1-2} gives
\begin{eqnarray*}
\big(\partial_t+ {v}\cdot\nabla \big)\widetilde{\varrho} &\leqslant&  \ln^+\big(\varphi_t(x)\big)\Big(\gamma^\prime(t)+6\|{v}(t)\|_{L(\Sigma_t)}e^{\int_0^t\Vert v(\tau)\Vert_{LL}d\tau} \Big)\widetilde{\varrho}\\
&&+  4  M (t,x)  \sqrt{\widetilde{\varrho}}\;\bar\varphi \big(\varphi_t(x)\big)^{{-\gamma(t) \over 2} - e^{\int_0^t \|v(\tau)\|_{LL}d\tau}}. \nonumber
\end{eqnarray*}
Hence we infer from the definition of $\mathbb{V}$ introduced  in \eqref{condition-one}
\begin{eqnarray}\label{equa000}
\big(\partial_t+ {v}\cdot\nabla \big)\widetilde{\varrho} &\leqslant&  \ln^+\big(\varphi_t(x)\big)\Big(\gamma^\prime(t)+6\mathbb{V}^\prime(t)e^{\mathbb{V}(t)} \Big)\widetilde{\varrho}\\
&&+  4  M (t,x)  \sqrt{\widetilde{\varrho}}\;\bar\varphi \big(\varphi_t(x)\big)^{{-\gamma(t) \over 2} - e^{\mathbb{V}(t)}}. \nonumber
\end{eqnarray}
Now, we should choose  $\gamma(t)$ and $m(t)$ such that 
 \begin{equation*}
\gamma^\prime(t)+ 6 \mathbb{V}^\prime(t)e^{\mathbb{V}(t)}  = 0, \quad \gamma(0)= \gamma
\end{equation*}  
and
\begin{equation*}
 m(t)=m+2\big(1-e^{\mathbb{V}(t)}\big) , \quad \frac{\gamma}{2} \triangleq m-1.
\end{equation*}
Therefore
\begin{align}\label{Jared-11}
\gamma(t)=\gamma+6\big(1-e^{\mathbb{V}(t)}\big)\quad\hbox{and}\quad {\frac{\gamma(t)}{2}+ e^{ \mathbb{V}(t)} { =}\,  m(t)} .
\end{align}
With this choice the  condition   \eqref{condition-one} follows easily. 

 Hence, with this choice,  we find from \eqref{equa000}
\begin{align}\label{diff-eq-LPL}
\big(\partial_t+ {v}\cdot\nabla \big)\sqrt{\widetilde{\varrho}}  \leqslant 2  M (t,x)  \;\bar\varphi \big(\varphi_t(x)\big)^{-m(t)}.
\end{align}
Remark that $x \mapsto d(x,\Sigma_t)$ is a 1-Lipschitz function. This gives, in particular
$$ 
\|M(t,\cdot) \|_{L^\infty }\leqslant 1.
$$
On the other hand, according to \eqref{phi-estimate one}, we know that
$$m(t) \geqslant 0 \Longrightarrow\|  (\varphi_t(\cdot ))^{-m(t)}\bar\varphi(t,\cdot) \|_{L^p(\mathbb{R}^2)} \leqslant C_p, \quad \forall p\in [2,\infty]. $$
Using \eqref{diff-eq-LPL} and the  definition of $\widetilde{\varrho}$ \eqref{VARRHO:DEF}  yields in view of the  above estimate 
\begin{equation}\label{xx}
|\nabla \bar\varphi(t,x)| \leqslant \big( \widetilde{C}_\infty  + t C_\infty \big) \big(\varphi_t(x)\big)^{ \gamma(t) \over 2}, \quad \forall x\notin \Sigma_t,
\end{equation}
where 
$$\widetilde{C}_\infty \triangleq \sup_{x\notin \Sigma_0}  \big(\varphi_0(x)\big)^{-m+1} |\nabla  \bar\varphi_0(x)| .$$
In fact, we can show that \eqref{xx} holds for all $x\in \mathbb{R}^2$. This can be done by   exploiting the  same arguments used in the inviscid case, see in particular \eqref{L:infty:1st:es}.
Hence, it follows that
$$\gamma(t)\geqslant 0 \Longrightarrow \|\nabla \bar\varphi(t,\cdot) \|_{L^\infty}\leqslant    \widetilde{C}_\infty  + t C_\infty .$$
Similarly, we obtain for $p\in [2,\infty)$ 
$$\| (\varphi_t(\cdot )\big)^{\frac{-\gamma(t)}{2}} \nabla \bar\varphi(t,\cdot)  \|_{L^p } \leqslant \widetilde{C}_{p} + t C_p, $$
where
$$\widetilde{C}_p \triangleq  \|  \varphi_0 (\cdot)^{-m+1}  \nabla \bar\varphi_0 (\cdot)  \|_{L^p}.$$
Let us point out that the assumption  \eqref{hypothesis2} together with the fact that $\Sigma_0$ is compact insure that $\widetilde{C} _p< \infty$. This can be proved along the same lines as in \eqref{ppp}.
On the other hand, we have 
$$ \gamma(t)\geqslant 0 \Longrightarrow |\nabla \bar\varphi(t,\cdot)|= (\varphi_t(\cdot )\big)^{\frac{-\gamma(t)}{2}} (\varphi_t(\cdot )\big)^{\frac{ \gamma(t)}{2}} |\nabla \bar\varphi(t,\cdot)| \leqslant  (\varphi_t(\cdot )\big)^{\frac{-\gamma(t)}{2}}  |\nabla \bar\varphi(t,\cdot)|  $$
Consequently, we obtain   for any $p\in[2,\infty]$
\begin{equation}\label{estimate-nabla phi}
\gamma(t)\geqslant 0 \Longrightarrow \|\nabla \bar\varphi(t,\cdot) \|_{L^p}\leqslant     \widetilde{C}_{p} + t C_p. 
\end{equation}  
Remark that  the choice \eqref{Jared-11}  provides the inequality 
 $$
 m(t)\geqslant \frac{\gamma(t)}{2} 
$$ 
which implies that the positivity of $\gamma$ will guarantee the positivity of $m$ in the same domain. More precisely, by the monotonicity and the continuity of $\gamma$   we have
\begin{equation}\label{T1}
[0,T_1)\triangleq \big\{t\in[0,T^\star): \; \gamma(t)\geqslant  0 \big\}\subset \mathtt{I}.
\end{equation} 
Consequently, we deduce from \eqref{phi-estimate one} and \eqref{estimate-nabla phi}  the existence of a  constant $C_0>0,$ depending only on the initial data, such that for all $p\in [2,\infty]$ and $t\in [0,T_1)$ 
\begin{equation}\label{result:1}
\| \bar{\varphi}(t,\cdot)\|_{W^{1,p}}\leqslant C_0(1+t).
\end{equation}

It remains to estimate $\psi$, the solution of \eqref{Eq4}. We shall restrict ourselves only for $p=2$ and $p= \infty$ and the other values will be done by an interpolation argument. \\

$\bullet$ {\it $L^2-$estimate}. A standard energy estimate for \eqref{Eq4} gives  
\begin{align}\label{L2:energy:ES}
 \frac12\frac{d}{dt}\|\psi(t)\|_{L^2}^2+\kappa\|\nabla \psi(t)\|_{L^2}^2\leqslant& 2\|v(t)\|_{L(\Sigma)}\int_{\mathbb{R}^2} \ln^+(d(x,\Sigma_t))\,\psi^2(t,x)dx\\
\nonumber&+\kappa\|\nabla\bar\varphi(t)\|_{L^2}\|\nabla\psi(t)\|_{L^2}.
\end{align}
Combining \eqref{comparaison11} with \eqref{result:1} and \eqref{L2:energy:ES}, it follows that  
\begin{equation}\label{L2-estimate}
\frac{d}{dt}\|\psi(t)\|_{L^2}^2+\kappa\|\nabla \psi(t)\|_{L^2}^2\leqslant 4A(t)\int_{\mathbb{R}^2} \ln^+(d\big(\Psi^{-1}(t,x),\Sigma_0\big))\,\psi^2(t,x)dx+\kappa C_0^2(1+t)^2,
\end{equation}
where $A(t)$ is given by \eqref{A:def:*}.  Let $\epsilon\in(0,e^{-1})$,  we introduce  
\begin{align}\label{ftttt}
f(t,x) \triangleq d\big(\Psi^{-1}(t,x),\Sigma_0\big)= f_0(\Psi^{-1}(t,x)), \quad f_0(x) \triangleq d(x,\Sigma_0)
\end{align}
 and we split 
\begin{align*}
\int_{\mathbb{R}^2} \ln^+\big(d\big(\Psi^{-1}(t,x),\Sigma_0\big)\big)\,\psi^2(t,x)&=\int_{f(t,x)>\epsilon} \ln^+(f(t,x))\,\psi^2(t,x)dx+\int_{f(t,x)\leqslant\epsilon}   \ln^+(f(t,x))\, \psi^2(t,x)dx \\
&\triangleq \textnormal{I}_1 + \textnormal{I}_2.
\end{align*}
To estimate $\textnormal{I}_1$ we employ the fact that $x\mapsto \ln^+(x)$ is a decreasing function in order to get
\begin{equation}\label{I1 estimate}
\textnormal{I}_1 \leqslant(\ln^+\epsilon) \|\psi(t)\|_{L^2}^2.
\end{equation}
To treat the term $\textnormal{I}_2$, H\"older and Gagliardo-Nirenberg inequalities yield
\begin{eqnarray*}
\textnormal{I}_2 &\leqslant &\left(\int_{f(t,x)\leqslant\epsilon}   \left(\ln^+(f(t,x))\right)^2  dx \right)^\frac{1}{2} \|\psi (t)\|_{L^4}^2\\
&\leqslant & \left(\int_{d(x,\Sigma_0)\leqslant\epsilon}  \left( \ln^+d(x,\Sigma_0) \right)^2  dx \right)^\frac{1}{2}\|\psi(t)\|_{L^2}\|\nabla \psi (t)\|_{L^2},
\end{eqnarray*}  
where  we have used that  ${\Psi}$ preserves the Lebesgue measure. On the other hand, we get from the inequality $\ln^+t\geqslant 1$,  the  assumption \eqref{int:cond0}  and some $\sigma>0$  
$$
\int_{d(x,\Sigma_0)\leqslant\epsilon}  \left( \ln^+d(x,\Sigma_0) \right)^2  dx   \leqslant \int_{d(x,\Sigma_0)\leqslant\epsilon}  \left( \ln^+d(x,\Sigma_0) \right)^4  dx     \lesssim \epsilon^{\sigma}. 
$$
Therefore, combining this estimate with Young inequality together with \eqref{I1 estimate} imply
$$
4A(t)\int_{\mathbb{R}^2} \ln^+\big(d(\Psi^{-1}(t,x),\Sigma_0)\big)\psi^2(t,x)dx \leqslant  \left(-4 A(t) \ln\epsilon + C  A^2(t) \epsilon^{\sigma} \kappa^{-1} \right)\|\psi (t)\|_{L^2}^2 + \frac{\kappa}{2} \|\nabla \psi (t)\|_{L^2}^2
$$
and then, \eqref{L2-estimate} yields   by applying Gronwall inequality 
$$
\|\psi (t)\|_{L^2}^2 \leqslant C_0^2(1+t)^3\kappa \epsilon ^{-4 \int_0^t A(\tau) d\tau} \exp \left(\epsilon^{\sigma} \kappa^{-1}C\int_0^t A^2(\tau) d\tau \right).  
$$
Hence, the choice $\epsilon = \kappa ^{\frac{1}{ \sigma}} $ enables to write 
$$
\|\psi (t)\|_{L^2}^2 \leqslant C_0^2(1+t)^3\kappa  ^{1-\frac{4}{\sigma} \int_0^t A(\tau) d\tau} \exp \left( C\int_0^t A^2(\tau)d\tau \right).   
$$
On the other hand, we infer from  \eqref{condition-one} and \eqref{A:def:*} 
\begin{equation}\label{A-V:es}
\int_0^t A(\tau) d\tau \leqslant  \int_0^t  \mathbb{V}^\prime(\tau)e^{\mathbb{V}(\tau)} d\tau= e^{\mathbb{V}(t)}-1
\end{equation}
which implies,  since   $\kappa\in(0,1)$, 
$$
\|\psi (t)\|_{L^2}^2 \leqslant C_0^2(1+t)^3\kappa  ^{\Theta_1(t) } \exp \left( C\int_0^t A(\tau)^2 d\tau \right),
$$
with
 $$\Theta_1(t) \triangleq 1+\frac{4}{\sigma}\Big( 1-e^{\mathbb{V}(t)} \Big).  $$
By setting 
\begin{equation}\label{T2}
 [0,T_2] \triangleq  \Big\{t\in [0,T_1]: \Theta_1(t) \geqslant  0\Big\}
\end{equation} 
we deduce that for any   $\kappa \in(0,1)$
\begin{equation}\label{L2-estimate-psi}
\forall\, t\in[0,T_2],\quad  \|\psi(t)\|_{L^2}^2   \leqslant  C_0^2(1+t)^3 \exp \left( C\int_0^t A^2(\tau) d\tau \right).
\end{equation} 

{\it $\bullet$ $L^\infty-$estimate}. Let $t\in [0,T_2]$, where $T_2$ is given by \eqref{T2} and decompose $\psi$, the solution of \eqref{Eq4},  as follows
\begin{align}\label{psi-decompos}
\psi=\psi_1+\psi_2+\psi_3,
\end{align}
with  \begin{equation} \label{Eq6}
\left\{ \begin{array}{ll}
   \big(\partial_t+ {v}\cdot\nabla-\kappa\Delta\big)\psi_1(t,x)= \|v(t)\|_{L(\Sigma_t)} \ln^+(d(x,\Sigma_t))\,{\bf{1}}_{f(t,x)>\epsilon}\psi(t,x),&\vspace{2mm}\\ 
  \psi|_{t=0}=0,
  \end{array}\right.
\end{equation}
\begin{equation} \label{Eq7}
\left\{ \begin{array}{ll}
   \big(\partial_t+ {v}\cdot\nabla-\kappa\Delta\big)\psi_2(t,x)= \|v(t)\|_{L(\Sigma_t)} \ln^+(d(x,\Sigma_t))\,{\bf{1}}_{f(t,x)\leqslant\epsilon}\psi(t,x)&\vspace{2mm}\\ 
  \psi|_{t=0}=0,
  \end{array}\right.
\end{equation}
and
\begin{equation} \label{Eq8}
\left\{ \begin{array}{ll}
   \big(\partial_t+ {v}\cdot\nabla-\kappa\Delta\big)\psi_3(t,x)= \kappa\Delta\bar\varphi,&\vspace{2mm}\\  
  \psi|_{t=0}=0,
  \end{array}\right.
\end{equation}
where the function $f$ is defined in \eqref{ftttt}. To handle \eqref{Eq6}, we combine the maximum principle with \eqref{comparaison11}  and \eqref{A:def:*} allowing to get for any $\epsilon\in(0,e^{-1})$
\begin{equation}\label{psi1 estimate}
\|\psi_1(t)\|_{L^\infty}\leqslant -\ln\epsilon\int_0^tA(\tau)\|\psi(\tau)\|_{L^\infty} d\tau.
\end{equation}
{To estimate $\psi_2$, it is enough to make use of  \eqref{comparaison11} and  \eqref{A:def:*}  and apply  Lemma \ref{Nash}, with $F=0$, $p_1 = 6$ and $q_1 =\frac{4}{3} $}
\begin{align*}
\|\psi_2(t)\|_{L^\infty}&\leqslant  C\big(1+(\kappa t)^{ {\frac{1}{12}} }\big)\kappa^{-\frac{5}{6}  }\left\| A \;\ln^+ f \;\mathbf{1}_{ \{f \leqslant \epsilon\}}  \psi  \right\|_{L^6_tL^\frac{4}{3}}\\
& \lesssim (1+t)^{{\frac{1}{12}}}\kappa^{-\frac56}\|A\|_{L_t^6}\sup_{\tau\in[0,t]}\left(\int_{f(\tau,x)\leqslant\epsilon}\left(\ln^+f(\tau,x))\right)^4dx\right)^{\frac14}\|\psi\|_{L^\infty_t L^2},
\end{align*}
Therefore we find by virtue of \eqref{ftttt} combined with  the fact that ${\Psi}$ preserves Lebesgue measure together with \eqref{int:cond0} and \eqref{L2-estimate-psi} 
\begin{equation}\label{psi2 estimate}
\|\psi_2(t)\|_{L^\infty} \leqslant  C_0(1+t)^{\frac{3}{2}+{\frac{1}{12}}} \kappa^{-\frac{5}{6}  }\epsilon^{\frac{\sigma}{4}} \| A  \|_{L^6_t} \exp \left( C\int_0^t A^2(\tau) d\tau \right).
\end{equation}
Let us now move to estimate $\psi_3$. Using once again Lemma \ref{Nash} with $G=0$, $q=\infty$ and $p= 3$ implies  
\begin{align*}
\|\psi_3(t)\|_{L^\infty}  \leqslant C\big(1+(\kappa t)^{\frac16}\big)\kappa^{\frac13}\|\nabla\bar\varphi\|_{L^3_t L^\infty}.
\end{align*}
Thanks  to \eqref{result:1}, it follows  that for $\kappa\in(0,1)$  
\begin{align}\label{psi3 estimate}
\|\psi_3(t)\|_{L^\infty}&\leqslant  C_0(1+t)^{\frac43+\frac16}\kappa^{\frac13}.
\end{align}
To achieve  the estimate of $\psi$ in $L^\infty-$norm, we collect \eqref{psi1 estimate}, \eqref{psi2 estimate} and \eqref{psi3 estimate} to obtain
$$ 
\|\psi(t)\|_{L^\infty}  \leqslant  C_0 (1+t)^2 \left( \kappa^{\frac{1}{3}} + \kappa^{-\frac{5}{6}  }\epsilon^{\frac{\sigma}{4}} \| A  \|_{L^6_t} \exp \left( C\int_0^t A^2(\tau) d\tau \right) \right) - \ln\epsilon\int_0^tA(\tau)\|\psi(\tau)\|_{L^\infty} d\tau.
$$
Thus,
Gronwall lemma yields
$$ 
\|\psi (t)\|_{L^\infty}  \leqslant  C_0 (1+t)^2 \left( \kappa^{\frac{1}{3}} + \kappa^{-\frac{5}{6}  }\epsilon^{\frac{\sigma}{4}} \| A  \|_{L^6_t} \exp \left( C\int_0^t A^2(\tau) d\tau \right) \right) \epsilon^{- \int_0^tA(\tau) d\tau}.
$$
By making the choice $\epsilon=\kappa^{\frac{14}{3\sigma}}$, we end up with 
\begin{equation*}
\|\psi (t)\|_{L^\infty}  \leqslant  C_0 (1+t)^2 \left(  1 +   \| A  \|_{L^6_t} \exp \left( C\int_0^t A^2(\tau) d\tau \right)\right)\kappa^{\frac13\big( 1-\frac{14}{\sigma}\int_0^tA(\tau)d\tau \big)}.
\end{equation*}
Thereafter, by using \eqref{A-V:es}, we find that
\begin{equation*}
\|\psi (t)\|_{L^\infty}  \leqslant  C_0 (1+t)^2 \left(  1 +   \| A  \|_{L^6_t} \exp \left( C\int_0^t A^2(\tau)d\tau \right)\right)\kappa^{\frac{\Theta_2(t)}{3}},
\end{equation*}
with
$$\Theta_2(t)\triangleq  1+\frac{14}{\sigma} \Big(1- e^{\mathbb{V}(t)} \Big).$$
By setting 
\begin{equation}\label{T1--3}
 [0,T_3] \triangleq  \Big\{ t\in [0,T_2]: \;\Theta_2(t) \geqslant 0 \Big\}. 
\end{equation}   
we  conclude that for $\kappa\in(0,1) $
\begin{equation*}
 \forall t\in [0,T_3],\qquad \|\psi (t)\|_{L^\infty}  \leqslant  C_0 (1+t)^2 \left(  1 +   \| A  \|_{L^6_t} \exp \left( C\int_0^t A(\tau)^2 d\tau \right)\right).
\end{equation*}
The $L^p$ estimates of $\psi$, for $p\in (2,\infty)$ can be easily obtained by interpolation. It is worthy to point out  from \eqref{T1}, \eqref{T2} and \eqref{T1--3} that,
$${\mathbb{V}(t)}\leqslant C_{\sigma,m} \implies  \min \Big\{\gamma(t), \Theta_1(t), \Theta_2(t) \Big\} \geqslant 0,$$
where, 
\begin{equation}\label{C-sigma-m}
C_{\sigma,m}\triangleq \log\Big( 1+ \min\Big\{  \frac{m-1}{3}, \frac{\sigma}{14}\Big\}\Big) >0 .
\end{equation}

Thereby, the final choice of $T$ is given by 
$$ [0,T] \triangleq \Big\{ t\in [0,T^\star):{\mathbb{V}(t)}\leqslant C_{\sigma,m} \;  \Big\}. $$
 This ends the proof of the $L^p$-estimates of $\nabla \rho$  which gives in turn the $L^p$-estimates of $\omega$. \\
 It remains to prove the estimate \eqref{Est-vLinfty} which will be done in the same spirit as  \cite{Hmidi-Keraani-2}.
 First,  Bernstein inequalities and Biot--Savart law allow to get for all $t\in [0,T]$ 
\begin{eqnarray}\label{v-es:0}
\|v(t) \|_{L^\infty} &\leqslant &\| \Delta_{-1} v(t)\|_{L^\infty} + \sum_{j\geqslant 0} \| \Delta_j v(t) \|_{L^\infty}\nonumber\\
& \lesssim & \| \Delta_{-1} v(t)\|_{L^\infty} + \sum_{j\geqslant 0} 2^{-j }\| \Delta_j \omega (t) \|_{L^\infty}\nonumber\\
& \lesssim & \| \Delta_{-1} v(t)\|_{L^\infty}  + \|   \omega (t) \|_{L^\infty} .
\end{eqnarray}
Now, in view of the estimate \eqref{es:G-W} the problem boils down to estimate the low frequencies of $ v(t) $.  Applying  $\Delta_{-1}$ to the velocity equation and taking the $L^\infty$ norm yields in view of Bernstein inequalities
 \begin{eqnarray}\label{v-es:1}
\|   \Delta_{-1} v(t)\|_{L^\infty} &\leqslant &\|  \Delta_{-1}  v_0\|_{L^\infty} + \int_0^t \Big(\| \Delta_{-1} (v\cdot \nabla v) (\tau) \|_{L^\infty}+\| \Delta_{-1} \nabla p (\tau) \|_{L^\infty} + \| \Delta_{-1}\rho(\tau) \|_{L^\infty} \Big)d\tau \nonumber\\
& \lesssim &\| v_0\|_{L^\infty} + \int_0^t \|  (v\cdot \nabla v ) (\tau) \|_{L^q}+\| \nabla p (\tau) \|_{L^q}d\tau + t\| \rho_0 \|_{L^q},
\end{eqnarray} 
where $ q\in [2,\infty)$ and we have used the fact that  
\begin{equation*}
\| \Delta_{-1}\rho(\tau) \|_{L^\infty} \leqslant \| \rho_0 \|_{L^\infty} , \quad \forall \tau \in [0,t].
\end{equation*} 
On the other hand, due to the divergence-free of $v$, we may write
$$\nabla p = -\nabla \Delta^{-1} \text{div} (v\cdot\nabla v) + \nabla \Delta^{-1} \partial_2 \rho. $$
Thus, we obtain, 
\begin{eqnarray}\label{p-es:0}
\|  (v\cdot \nabla v ) (\tau) \|_{L^q}+ \| \nabla p (\tau)\|_{L^q} &\lesssim &\|v(\tau) \|_{L^\infty}\|\omega(\tau) \|_{L^q} + \|\rho(\tau) \|_{L^q} \nonumber\\
& \lesssim &\|v(\tau) \|_{L^\infty}\|\omega(\tau) \|_{L^q} + \|\rho_0 \|_{L^q}.
\end{eqnarray} 
Gathering \eqref{v-es:0}, \eqref{v-es:1} and \eqref{p-es:0} yields
\begin{equation*}
\|v(t) \|_{L^\infty} \leqslant C\|v_0 \|_{L^\infty} + Ct \|\rho_0 \|_{L^q} + C\|   \omega (t) \|_{L^\infty} + C\int_0^t \|v(\tau) \|_{L^\infty}\|\omega(\tau) \|_{L^q} d\tau .
\end{equation*}
Gronwall lemma implies
\begin{equation}\label{gronwall}
\|v(t) \|_{L^\infty} \lesssim \Big( (1+ t )C_0  +  \|   \omega (t) \|_{L^\infty} \Big) e^{ C\int_0^t  \|\omega(\tau) \|_{L^q} d\tau }.
\end{equation}
On the other hand, from Proposition \ref{Proposition visous}, we have
$$  \|\omega(t) \|_{L^q \cap L^\infty} \leqslant G(t),$$
where, $G(t)$ is given by \eqref{G:def} and satisfies, up to a suitable modification in $C_0$,  $$G(t) \geqslant  C_0(1+t) .$$
Implementing these estimates in \eqref{gronwall} concludes the proof of \eqref{Est-vLinfty}.\\
 The proof of the last point of   Proposition \ref{Proposition visous} concerning the blowup criterion can be achieved following exactly the same argument as in the proof of Corollary \ref{cor-blowup}.
 
\end{proof}
\subsubsection{Co-normal regularity persistence} 
 The next purpose that we want to achieve is to track the  striated regularity of the vorticity in the the scale of  Besov spaces $B^{s}_{p,\infty}$, with finite $p$.  For this aim, the a priori estimates stated in Proposition \ref{Proposition visous} are the cornerstone in  dealing with the regularity persistence issue. The first result reads as follows. 
\begin{proposition}\label{prop2222}
Let $(s,\varepsilon,h,p)\in(0,1)\times(0,\frac{s}{2})   \times(0,e^{-1}]\times[ 1,\infty)$. Consider a compact negligible set $\Sigma_0$ of  the plane satisfying \eqref{int:cond0} and a vector field  $X_0$ of  class $B^{s}_{p,\infty}$ as well as its divergence and whose support  is embedded in $(\Sigma_0)_h^c$. Let $(\omega,\rho)$ be a smooth maximal solution of the system \eqref{eqn:omega:2} defined on   $[0,T^\star)$ with initial data $(\omega_0,\rho_0)$ satisfying the hypothesis of Proposition \ref{Proposition visous}. Let $X_t$   be the solution of the transport equation
\begin{equation*}
\left\{ \begin{array}{ll}
\big(\partial_t +v\cdot\nabla\big)X_t=\pxt v, &\vspace{2mm}\\
X_{| t=0}=X_{0} 
\end{array} \right.
\end{equation*} 
and introduce  
\begin{equation}\label{sigmat:def}
\sigma_t \triangleq s\Big(1-\frac{\varepsilon t}{sT}\Big)-\varepsilon,
\end{equation} 
with
\begin{equation}\label{T:PROP:ST}
T = \sup\Big\{t \leqslant  \min\{1,T^\star\} :\;\mathbb{V} (t) \leqslant C_{\sigma,m}\Big\},
\end{equation}
where $\mathbb{V}(t)$ and $C_{\sigma,m}$ are given by \eqref{V:def:0} and \eqref{C-sigma-m}, respectively.
Then, the following holds for \mbox{all $t\in [0,T]$}

 
\begin{equation}\label{ES:striated:reg}
 \widetilde{\|} X_t \|_{B^{\sigma_t}_{p,\infty}}+\| \partial_{X_t}\omega(t)\|_{B^{\sigma_t-1}_{p,\infty}} \lesssim        \widetilde{\zeta} _0 h^{- C t  \widetilde{W}^2(t) }
 \end{equation} 
 and
 \begin{equation}\label{ES:striated:reg:rho}
  \| \partial_{X_t}\rho(t)\|_{B^{\sigma_t}_{p,\infty}}\lesssim  \widetilde{\zeta} _0 h^{- C   \widetilde{W}^2(t) } ,
 \end{equation} 
with
\begin{equation*}
\widetilde{\zeta} _0 \triangleq \widetilde{\Vert} X_{0}\Vert_{  B^{s}_{p,\infty}} +\|\partial_{ X_{0}}\omega_0\|_{ B^{s-1}_{p,\infty}}+ \|\partial_{ X_{0}}\rho_0\|_{ B^{s}_{p,\infty}}  ,
\end{equation*}
\begin{equation}\label{widetilde:W:def}
\widetilde{W} (t)\triangleq\Big( \sup_{\tau\in[0,t]}\Vert  v(\tau)\Vert_{L(\Sigma_\tau)}+G(t) \Big)\exp\bigg(\int_0^t\Vert v(\tau)\Vert_{LL}d\tau\bigg)
\end{equation} 
and $G(t)$ is given by \eqref{G:def}.
\end{proposition}
 \begin{remark}\label{rmk:SR:reg}
The assumption $T\leqslant 1$ is not essential. It is merely required to avoid the time growth and get simpler  a priori estimates.  Remind that we are dealing with local well-posedness issue and the global well-posedness is still an  open problem. 
\end{remark}
\begin{proof}[Proof of Proposition \ref{prop2222}]
First, let us point out that defining   $T$ by \eqref{T:PROP:ST}  insures, in view of Proposition \ref{Proposition visous}
\begin{equation}\label{lp:es:****}
\|\nabla \rho(t) \|_{L^2\cap L^\infty} + \|\omega(t) \|_{L^2\cap L^\infty} \leqslant G(t),\quad \forall t\in [0,T],
\end{equation}
where $G(t)$ is given by \eqref{G:def}.
This implies that $W(t)$ which is given by \eqref{W:def} satisfies
$$ W(t)\leqslant \widetilde{W}(t), \quad \forall t\in [0,T].$$

 Consequently, Corollary \ref{coro-transport:es} yields for all $t\in [0,T]$ 
\begin{equation*}
\supp X_t\subset \big(\Sigma_t\big)_{\delta_t(h)}^c\quad \textnormal{with}\quad \delta_t(h)= h^{\exp\big(\int_0^t\Vert v(\tau)\Vert_{LL}d\tau\big)},
\end{equation*}
\begin{equation}\label{div:2}
\Vert \Div X_t\Vert_{B^{s}_{p,\infty}}\leqslant \Vert \Div X_{0}\Vert_{B^{s}_{p,\infty}}h^{-C\int_0^t \widetilde{W} (\tau)d\tau}
\end{equation}
and
\begin{eqnarray}\label{8-E:viscous}
\widetilde{\Vert} X_t\Vert_{B^{\sigma_{t}}_{p,\infty}}&\leqslant & h^{-C\int_{0}^{t}\widetilde{W}(\tau)d\tau}\Big(\widetilde{\Vert} X_0\Vert_{B^{s}_{p,\infty}} +C\int_0^t h^{C\int_{0}^{\tau}\widetilde{W}(\tau')d\tau'}\|\partial_{ X_{\tau}}\omega(\tau)\|_{B^{\sigma_{\tau}-1}_{p,\infty}}d\tau\Big).
\end{eqnarray} 
Now, to control the stratiad regularity of the vorticity, we apply the directional derivative $\partial_{X_t}$ to the equation of $\omega $  in \eqref{eqn:omega:2} to obtain 
\begin{equation}\label{X-t-omega}
\big(\partial_t+v\cdot\nabla\big)\partial_{X_t}\omega=\partial_{ X_t}\partial_1\rho.
\end{equation}  
It is easy to see that 
$$ \partial_{ X_t}\partial_1\rho = \partial_1\partial_{ X_t}\rho - \partial_{ \partial_1 X_t}\rho$$
and the estimate of the second term on the r.h.s follows from Corollary \ref{ppxr0} and the fact \mbox{that $\sigma_\tau\in (0,1)$ }
$$ \|\partial_{ \partial_1 X_\tau}\rho \|_{B^{\sigma_\tau-1}_{p,\infty}} \lesssim \|\nabla \rho\|_{L^\infty}\widetilde{\|} X_{\tau}\|_{B^{\sigma_\tau}_{p,\infty} }.$$
Hence, applying Proposition \ref{prop:transport:es} to \eqref{X-t-omega} gives
\begin{eqnarray*}
\|\partial_{ X_t}\omega\|_{B^{\sigma_t-1}_{p,\infty}}&\leqslant &  \|\partial_{ X_0}\omega_0\|_{B^{\sigma_t-1}_{p,\infty}}h^{-C\int_{0}^{t}\widetilde{W}(\tau)d\tau}+\int_{0}^{t}\|\partial_{ X_\tau}\rho(\tau)\|_{B^{\sigma_\tau}_{p,\infty}}h^{-C\int_{\tau}^{t}\widetilde{W}(\tau')d\tau'}d\tau\\\nonumber
&&+C\int_{0}^{t}\|\nabla\rho(\tau)\|_{L^\infty}\widetilde{\|}X_{\tau}\|_{B^{\sigma_\tau}_{p,\infty}}h^{-C\int_{\tau}^{t}\widetilde{W}(\tau')d\tau'}d\tau.
\end{eqnarray*}  
Next,  the elementary  inequalities below
$$\|\nabla \rho(\tau) \|_{L^\infty}\leqslant G(\tau)   \leqslant \widetilde{W}(\tau) ,\quad \forall \tau \in [0,T],$$
 combined with the fact that $\sigma_t \leqslant s$, allow to get 
\begin{eqnarray}\label{DXOMEGA::}
\|\partial_{ X_t}\omega\|_{B^{\sigma_t-1}_{p,\infty}}&\leqslant &  \|\partial_{ X_0}\omega_0\|_{B^{s-1}_{p,\infty}}h^{-C\int_{0}^{t}\widetilde{W}(\tau)d\tau}+\int_{0}^{t}\|\partial_{ X_\tau}\rho(\tau)\|_{B^{\sigma_\tau}_{p,\infty}}h^{-C\int_{\tau}^{t}\widetilde{W}(\tau')d\tau'}d\tau\\\nonumber
&&+ C\int_{0}^{t}\widetilde{W}(\tau)\widetilde{\|}X_{\tau}\|_{B^{\sigma_\tau}_{p,\infty}}h^{-C\int_{\tau}^{t}\widetilde{W}(\tau')d\tau'}d\tau.
\end{eqnarray}
Hence, gathering this estimate  with \eqref{8-E:viscous} yields 
\begin{eqnarray}\label{38-E-2}
\zeta(t) &\leqslant &     \zeta_0  + C  \int_0^t  \widetilde{W}(\tau)\zeta(\tau) d\tau  + \int_{0}^{t}\|\partial_{ X_\tau}\rho(\tau)\|_{B^{\sigma_\tau}_{p,\infty}}h^{ C\int_{0}^{\tau}\widetilde{W}(\tau')d\tau'}d\tau ,
\end{eqnarray}
where we use the fact $1 \leqslant \widetilde{W}(t)$ and the following notation 
\begin{equation}\label{Gamma:def**}
\zeta(t) \triangleq \left(\widetilde{\Vert} X_{\cdot}\Vert_{L^ \infty_t( B^{\sigma_{\cdot}}_{p,\infty})} +\|\partial_{ X_{\cdot}}\omega\|_{L^\infty_t( B^{\sigma_{\cdot}-1}_{p,\infty})} \right)h^{C\int_0^t \widetilde{W}(\tau)d\tau}, \quad \zeta_0 = \zeta(0)
\end{equation}  
and
\begin{equation}\label{space:00:def}
\widetilde{\Vert}X_\cdot \|_{L^\infty_t(B^{\sigma_\cdot}_{p,\infty})}  \triangleq  \sup_{\tau \in [0,t]} \widetilde{\Vert} X_\tau \|_{B^{\sigma_\tau}_{p,\infty}}, \quad \|\partial_{ X_{\cdot}}\omega \|_{L^\infty_t(B^{\sigma_\cdot}_{p,\infty})}  \triangleq  \sup_{\tau \in [0,t]} \Vert\partial_{ X_{\tau}}\omega  \|_{B^{\sigma_\tau}_{p,\infty}}. 
\end{equation}
 The problem boils down to estimate $\|\partial_{X_{t}}\rho(t)\|_{B^{\sigma_t}_{p,\infty}}$. To do so, we apply the directional derivative $\partial_{X_t}$ to $\rho-$equation to find that
\begin{equation}\label{TG}
\left\{\begin{array}{ll}
\big(\partial_t+v\cdot\nabla-\kappa \Delta\big)\partial_{X_t}\rho=-\kappa[\Delta,\partial_{X_t}]\rho &\vspace{2mm}\\
\partial_{X_t}\rho_{|_{t=0}}=\partial_{X_0}\rho_0.
\end{array}
\right.
\end{equation}
Above, $[\Delta,\partial_{X_t}]$ designates the commutator between $\Delta$ and $\partial_{X_t}$ which can be split in view of Bony's decomposition in the following way, see for instance   \cite{Danchin-1,Hmidi-0} 
$$  \kappa [\Delta,\partial_{X_t}] = {\bf F} + \kappa {\bf G}, $$
where 
\begin{eqnarray*}
{\bf F}&=& 2\kappa R(\nabla X^i_{t },\partial_i\nabla\rho)+\kappa R(\Delta X^i_{t},\partial_i\rho)
\end{eqnarray*}
and
\begin{eqnarray*}
{\bf G}&=& 2T_{\nabla X^{i}_{t }} \partial_i \nabla \rho +2T_{\partial_i\nabla\rho}\nabla X^{i}_{t }+T_{\Delta X^{i}_{t }}\partial_i\rho+T_{\partial_i\rho}\Delta X^{i}_{t }.
\end{eqnarray*}
Hence, applying the estimate \eqref{ES:before Gronwal} from Proposition \ref{p11} yields

\begin{eqnarray} \label{d_x(rho):1}
 \|\Delta_q  \partial_{X_t}\rho(t)\|_{L^p}&\leqslant & \|\Delta_q  \partial_{X_0}\rho_0\|_{L^p}+ \|{\Delta_q  {\bf F}}\|_{L^{1}_{t}L^p}+C(1+\kappa t)2^{-2q}\|{\Delta_q  {\bf G}}\|_{L^{\infty}_{t}L^p}\\ && -\log h \int_0^t \frac{2^{-q\sigma_\tau}}{1-\sigma_\tau^2} \| \partial_{X_\tau}\rho(\tau) \|_{B^{\sigma_\tau}_{p,\infty}} W (\tau) d\tau\nonumber. 
\end{eqnarray} 
  The estimates of the terms in ${\bf G}$ follow by applying Lemma \ref{para--products} and \eqref{lp:es:****}
\begin{eqnarray}\label{GG:estimate}
2^{(\sigma_t-2)q}\|{\Delta_q  {\bf G}}\|_{L^{\infty}_{t}L^p} \leqslant\|  {\bf G} \|_{\widetilde{L}^{\infty}_t B^{\sigma_t-2}_{p,\infty}} &=& \|  {\bf G} \|_{L^\infty _t B^{\sigma_t-2}_{p,\infty}} \nonumber\\
&\lesssim &  \| \nabla \rho\|_{L^\infty_t L^\infty} \|X_\cdot \|_{L^\infty_t B^{\sigma_t}_{p,\infty}} \\
&\lesssim & G(t)\|X_\cdot \|_{L^\infty_t B^{\sigma_\cdot}_{p,\infty}},\nonumber
\end{eqnarray}  
where we   used  in the last line   the non-increasing property of  $t \mapsto \sigma_t$  which allows to get 
$$
\|X_\cdot \|_{L^\infty_t(B^{\sigma_t}_{p,\infty})} \triangleq \sup_{\tau \in [0,t]}  \|  X_\tau \|_{B^{\sigma_t}_{p,\infty}} \leqslant    \|X_\cdot \|_{L^\infty_t(B^{\sigma_\cdot}_{p,\infty})}.
$$
Now, we estimate  ${\bf F}$ by exploring the smoothing-effect estimates   stated in \mbox{Corollary \ref{cor-Max-reg}}. First, we transform ${\bf F}$ in the following form by using Leibniz  rule
\begin{eqnarray*}
{\bf F}&=& 2\kappa \partial_i R(\nabla X^i_{t },\nabla\rho)+2\kappa\partial_i R(\Delta X^i_{t },\rho) - 2\kappa R(\nabla \text{div} X _{t },\nabla\rho)-2\kappa  R(\Delta \text{div}X _{t },\rho) .
\end{eqnarray*}
Then, Bernstein  Lemma implies 
\begin{eqnarray}\label{FF:ES}
\| \Delta_q{\bf F}\|_{L^1_tL^p}&\lesssim  &  2^q \kappa\sum_{j\geqslant q - 5} \|\widetilde{\Delta}_j X_{\cdot} \|_{L^\infty_tL^p} 2^{ j}\|\Delta_j \nabla \rho \|_{L^1_tL^\infty}\\
&+&\kappa \sum_{j\geqslant q - 5} \|\widetilde{\Delta}_j \Div X_{\cdot} \|_{L^\infty_tL^p} 2^{ j}\|\Delta_j \nabla \rho \|_{L^1_tL^\infty}.
\end{eqnarray}

On the other hand, Corollary \ref{cor-Max-reg} together with the definition of $T$ in \eqref{T:PROP:ST} insures that 
\begin{eqnarray*} 
 \kappa2^{2q} \Vert \Delta_q (\partial_i\rho) \Vert_{L_t^1 L^\infty}&\leqslant  & G(t)  \Big(1+\kappa t   + (q+2)\int_0^t\Vert v(\tau)\Vert_{LL}d\tau \Big) \nonumber\\
 & \leqslant &   G(t)   (q+2)   ,
\end{eqnarray*}
Thus,  plugging this  inequality in   \eqref{FF:ES} yields for all $t\in [0,T]$ 
\begin{eqnarray}\label{FF:estimate}
 2^{q \sigma_t }   \|{\Delta_q  {\bf F}}\|_{L^{1}_{t}L^p} &\leqslant &  2^{q(\sigma_t +1 )} G(t)  \sum_{j\geqslant q - 5} 2^{- j } (j+2)\Big( \| \Delta _j X_{\cdot} \|_{L^\infty_tL^p} + \| \Delta _j \text{div} X_{\cdot} \|_{L^\infty_tL^p} \Big).
\end{eqnarray}

Therefore,  by using \eqref{GG:estimate} and \eqref{FF:estimate} in \eqref{d_x(rho):1} together with the fact that $\sigma_t\leqslant s$ we obtain, for any $\kappa\in (0, 1)$ and $t \in [0,T]$
\begin{eqnarray} \label{D_X rho:ES}
2^{q \sigma_t } \|\Delta_q  \partial_{X_t}\rho(t)\|_{L^p}&\leqslant & \|  \partial_{X_0}\rho_0\|_{B^s_{p,\infty}}+C(-\log h) \int_0^t   \| \partial_{X_\tau}\rho(\tau) \|_{B^{\sigma_\tau}_{p,\infty}} \widetilde{W} (\tau) d\tau\nonumber \\ && +    G(t) \|  X_\cdot\|_{L^{\infty}_{t} B^{\sigma_\cdot}_{p,\infty}}\\
&& +G(t) 2^{q(\sigma_t +1 )} \sum_{j\geqslant q - 5} 2^{ -j } (j+2)\Big( \| \Delta _j X_{\cdot} \|_{L^\infty_tL^p} + \| \Delta _j \text{div} X_{\cdot} \|_{L^\infty_tL^p} \Big) \nonumber.
\end{eqnarray}
In order to estimate the term $\|  X_\cdot\|_{L^{\infty}_{t} B^{\sigma_\cdot}_{p,\infty}} $, we make use of \eqref{8-E:viscous}  to find
\begin{eqnarray} \label{X_t:ES XX}
\|  X_\cdot\|_{L^{\infty}_{t} B^{\sigma_\cdot}_{p,\infty}} &\leqslant & h^{-C\int_{0}^{t}\widetilde{W}(\tau)d\tau}\Big(\widetilde{\Vert} X_0\Vert_{B^{s}_{p,\infty}} +C\int_0^t h^{C\int_{0}^{\tau}\widetilde{W}(\tau')d\tau'}\|\partial_{ X_{\tau}}\omega(\tau)\|_{B^{\sigma_{\tau}-1}_{p,\infty}}d\tau\Big).
\end{eqnarray}
Now,  we deal with $\Div X_\cdot$ that appears in the last line of \eqref{D_X rho:ES} by exploiting \eqref{div:2} and the definition of $\sigma_t$ to achieve

\begin{equation*}
2^{q\sigma_\tau}\| \Delta _j \text{div} X_{\tau} \|_{ L^p}  \lesssim 2^{-j\varepsilon} \|  \text{div} X_{0} \|_{B^s_{p,\infty}}h^{-C\int_0^\tau \widetilde{W} (\tau') d\tau'}.
\end{equation*}
Then, the fact $h\in (0,e^{-1}]$ and the monotonicity of $\sigma_t$ imply  
\begin{equation}\label{div:2XX}
2^{q(\sigma_t +1  )}    \sum_{j\geqslant q - 5}   \| \Delta _j \Div X_{\cdot} \|_{L^\infty_tL^p}  2^{-j} (j+2)\lesssim \|  \text{div} X_{0} \|_{B^s_{p,\infty}}h^{-C\int_0^t \widetilde{W} (\tau ) d\tau }.
\end{equation}
We are left with the estimate of $\| \Delta _j X_{\cdot} \|_{L^\infty_tL^p}$ in the last line of  \eqref{D_X rho:ES}. First,  it follows from \eqref{X-B-sigma-0:XXX} and the fact   $W(t)\leqslant \widetilde{W}(t)$   that
\begin{eqnarray} \label{Last:es}
\nonumber\|\Delta_{j}X_t \|_{L^p}&\leqslant & \|\Delta_{j}X_0\|_{L^p}+C2^{-js}\int_0^t\widetilde{W} (\tau)  \|\text{div} X_\tau\|_{B^{s}_{p,\infty}}  d\tau  +  C \int_0^t 2^{-j\sigma_\tau}\| \partial_{X_{\tau}}\omega\|_{B^{\sigma_\tau-1}_{p,\infty}} d\tau \\&&+C(-\log h)\int_0^t 2^{-j\sigma_\tau}\|X_{\tau}\|_{B^{\sigma_{\tau}}_{p,\infty}}\widetilde{W}(\tau)d\tau.
\end{eqnarray} 
We will treat each term on the r.h.s of \eqref{Last:es} separately. To begin with, owing to the definition of $\sigma_t$ given by \eqref{sigmat:def}, we have
\begin{equation*}
 \|\Delta_{j}X_0\|_{L^p} \leqslant 2^{-j(\sigma_t+\varepsilon)}  \| X_0\|_{B^s_{p,\infty}}.
\end{equation*}
Hence, we get
\begin{equation}\label{+1}
2^{q(\sigma_t+1 )} \sum_{j\geqslant q-5}(j+2)2^{-j} \|\Delta_{j}X_0\|_{L^p} \leqslant C \| X_0\|_{B^s_{p,\infty}}.
\end{equation}
Next, to estimate the second term on the r.h.s of \eqref{Last:es}, we use again  \eqref{div:2} together with the fact that $1\leqslant -\log h $ to obtain
\begin{align}\label{A1}
C\int_0^t \widetilde{W}(\tau)\|\text{div} X_\tau\|_{ B^{s}_{p,\infty}}  d\tau &\leqslant -C\log h\int_0^t \widetilde{W}(\tau)    h^{-C\int_0^\tau \widetilde{W}(\tau')d\tau'} d\tau\|\text{div} X_0\|_{B^{s}_{p,\infty}}\nonumber\\ 
 & \leqslant     \|\text{div} X_0\|_{B^{s}_{p,\infty}}   h^{-C\int_0^t \widetilde{W}(\tau )d\tau }
\end{align}
 On the other hand, the definition of  $\sigma_t$ yields
\begin{equation}\label{A2}
2^{-js}(j+2) \leqslant 2^{-j\sigma_t} \frac{j+2}{2^{j\varepsilon}}\lesssim 2^{-j\sigma_t}.
\end{equation}   
 Hence, \eqref{A1} and \eqref{A2} imply that
\begin{align}\label{+2}
 C 2^{q(\sigma_t +1  )}    \sum_{j\geqslant q - 5} 2^{-js}(j+2)2^{ -j} \int_0^t \widetilde{W}(\tau) \|\text{div} X_\tau\|_{B^{s}_{p,\infty}}  d\tau       &\lesssim \|\text{div} X_0\|_{B^{s}_{p,\infty}} h^{-C\int_0^t \widetilde{W}(\tau )d\tau }.
\end{align}

Finally, the last two terms in \eqref{Last:es} can be treated simultaneously   by using the computation
\begin{equation*}
(j+2) \int_0^t 2^{-j\sigma_\tau} d\tau \lesssim  \; \frac{ j+2}{j} \;2^{-j \sigma_t} ,\quad \forall j\geqslant 1, 
\end{equation*}
to obtain, for $t \in [0,T]$
\begin{eqnarray}\label{+3}
2^{q(\sigma_t +1  )}    \sum_{j\geqslant q - 5}  2^{ -j} (j+2)\int_0^t2^{-j\sigma_\tau}\| \partial_{X_{\tau}}\omega\|_{B^{\sigma_\tau-1}_{p,\infty}} d\tau  & \lesssim  & \| \partial_{X_{\cdot}}\omega\|_{L^\infty_tB^{\sigma_\cdot-1}_{p,\infty}}  
\end{eqnarray} 
and
\begin{equation}\label{+4}
 2^{q(\sigma_t +1 )}    \sum_{j\geqslant q - 5}  2^{ -j} (j+2)\int_0^t2^{-j\sigma_\tau}\|  X_\tau\|_{B^{\sigma_\tau}_{p,\infty}} \widetilde{W}(\tau)d\tau  \lesssim \widetilde{W}(t) \| X_\cdot\|_{L^\infty_tB^{\sigma_\cdot}_{p,\infty}},
\end{equation}
where we have used the fact that $ t\mapsto \widetilde{W}(t) $ is increasing.
Now, gathering \eqref{+1}, \eqref{+2}, \eqref{+3} and \eqref{+4}  with \eqref{Last:es} lead to
\begin{eqnarray}\label{Delta_q X:ES}
2^{q(\sigma_t +1 )}    \sum_{j\geqslant q - 5} 2^{ -j} (j+2)  \|\Delta_{j}X_t \|_{L^p} &\lesssim  &  \widetilde{\|} X_0\|_{B^s_{p,\infty}} h^{-C\int_0^t \widetilde{W}(\tau )d\tau }\\
&&+  \| \partial_{X_{\cdot}}\omega\|_{L^\infty_tB^{\sigma_\cdot-1}_{p,\infty}}+ (-\log h)   \widetilde{W}(t) \| X_\cdot\|_{L^\infty_tB^{\sigma_\cdot}_{p,\infty}} .\nonumber 
\end{eqnarray}
Plugging \eqref{X_t:ES XX}, \eqref{div:2XX} and \eqref{Delta_q X:ES} in \eqref{D_X rho:ES}, together with the facts $ -\log h \geqslant 1$ and   $ \widetilde{W}(t)\geqslant 1$ yield 
\begin{eqnarray*}
  \| \partial_{X_t}\rho(t)\|_{B^{\sigma_t}_{p,\infty}}&\leqslant & \|  \partial_{X_0}\rho_0\|_{B^s_{p,\infty}}+C(-\log h) \int_0^t   \| \partial_{X_\tau}\rho(\tau) \|_{B^{\sigma_\tau}_{p,\infty}} \widetilde{W} (\tau) d\tau \nonumber\\ 
&& +  G(t) h^{-C\int_{0}^{t}\widetilde{W}(\tau)d\tau} \Bigg(\widetilde{\Vert} X_0\Vert_{B^{s}_{p,\infty}}   + \int_0^t h^{C\int_{0}^{\tau}\widetilde{W}(\tau')d\tau'}\|\partial_{ X_{\tau}}\omega(\tau)\|_{B^{\sigma_{\tau}-1}_{p,\infty}}d\tau \Bigg) \nonumber\\
&&  + G(t)    \| \partial_{X_{\cdot}}\omega\|_{L^\infty_tB^{\sigma_\cdot-1}_{p,\infty}}      + (-\log h) \widetilde{W}(t)    G(t)   \| X_\cdot\|_{L^\infty_tB^{\sigma_\cdot}_{p,\infty}}    .
\end{eqnarray*}
Thereafter, in view of the notation \eqref{Gamma:def**}, we use   \eqref{X_t:ES XX} together with \eqref{DXOMEGA::}, to end up with 
\begin{eqnarray*}  
  \| \partial_{X_t}\rho(t)\|_{B^{\sigma_t}_{p,\infty}}&\leqslant & \|  \partial_{X_0}\rho_0\|_{B^s_{p,\infty}} -\log h\; G(t)  \widetilde{W}(t) h^{-C\int_{0}^{t}\widetilde{W}(\tau)d\tau}   \zeta_0   \\
  && +(- \log h)\; G(t) \int_0^t  \widetilde{W}(\tau)  h^{-C\int_{\tau}^{t}\widetilde{W}(\tau')d\tau'} \| \partial_{X_\tau}\rho(\tau) \|_{B^{\sigma_\tau}_{p,\infty}}   d\tau \nonumber\\ 
&& +(- \log h)\;  G(t)    \int_0^t \widetilde{W}(\tau)  h^{-C\int_{\tau}^{t}\widetilde{W}(\tau')d\tau'}\Big( \|\partial_{ X_{\tau}}\omega(\tau)\|_{B^{\sigma_{\tau}-1}_{p,\infty}}  + \widetilde{\|}X_{\tau}\|_{B^{\sigma_\tau}_{p,\infty}} \Big) d\tau   .
\end{eqnarray*}
Consequently, we get
\begin{eqnarray*} 
   h^{C\int_{0}^{t}\widetilde{W}(\tau)d\tau} \| \partial_{X_t}\rho(t)\|_{B^{\sigma_t}_{p,\infty}}&\leqslant &  (- \log h)\; G(t)  \widetilde{W}(t)    \widetilde{\zeta}_0 \\
  && +(- \log h)\; G(t) \int_0^t  \widetilde{W}(\tau)  h^{C\int_{0}^{\tau}\widetilde{W}(\tau')d\tau'} \| \partial_{X_\tau}\rho(\tau) \|_{B^{\sigma_\tau}_{p,\infty}}   d\tau \nonumber\\ 
   && +(- \log h) \;G(t)  \int_0^t   \widetilde{W}(\tau)  \zeta(\tau)d\tau .
\end{eqnarray*}
Applying Gronwall lemma  leads then to 
\begin{eqnarray}\label{DX-rho:ponct} 
   h^{C\int_{0}^{t}\widetilde{W}(\tau)d\tau} \| \partial_{X_t}\rho(t)\|_{B^{\sigma_t}_{p,\infty}}&\leqslant &  (- \log h)\; G(t)  \widetilde{W}(t)    \widetilde{\zeta}_0 h^{- G(t)\int_{0}^{t}\widetilde{W}(\tau')d\tau'} \nonumber\\
   && +(- \log h) \;G(t)      \int_0^t   \widetilde{W}(\tau)  h^{-  G(t)\int_{\tau}^{t}\widetilde{W}(\tau')d\tau'}\zeta(\tau)d\tau .
\end{eqnarray}
An integration  with respect to $t$  together with the computation
\begin{eqnarray*}
(-\log h) \;\int_0^t     G(\tau)  \widetilde{W}(\tau)    h^{- G(\tau)\int_{0}^{\tau}\widetilde{W}(\tau')d\tau'} d\tau & \leqslant  & (-\log h)\; G(t)  \int_0^t     \widetilde{W}(\tau)      h^{- G(t)\int_{0}^{\tau}\widetilde{W}(\tau')d\tau'} d\tau  \\
&\leqslant &    h^{-  G(t)\int_{0}^{t}\widetilde{W}(\tau')d\tau'} 
\end{eqnarray*}
leads to
\begin{eqnarray*}  
&&\int_0^t   h^{C\int_{0}^{\tau}\widetilde{W}(\tau')d\tau'} \| \partial_{X_\tau}\rho(\tau)\|_{B^{\sigma_\tau}_{p,\infty}} d\tau \leqslant        \widetilde{\zeta}_0 h^{-  G(t)\int_{0}^{t}\widetilde{W}(\tau')d\tau'}\\
   &&\quad \qquad\qquad + (- \log h)\int_0^t  G(\tau) d\tau \Bigg(     \int_0^t   \widetilde{W}(\tau)  h^{- G(t)\int_{\tau}^{t}\widetilde{W}(\tau')d\tau'}\zeta(\tau)d\tau\Bigg).  
\end{eqnarray*}
Plugging this inequality in \eqref{38-E-2} and using the fact that $1\leqslant G(t) $ yield
\begin{eqnarray*}
\zeta(t) &\leqslant &      \widetilde{\zeta}_0 h^{- G(t)\int_{0}^{t}\widetilde{W}(\tau')d\tau'}\\
   && + (- \log h) \Big(1 + \int_0^t  G(\tau) d\tau\Big) \Bigg(     \int_0^t   \widetilde{W}(\tau)  h^{- G(t)\int_{\tau}^{t}\widetilde{W}(\tau')d\tau'}\zeta(\tau)d\tau\Bigg).
\end{eqnarray*}
Next, by setting 
$$ \Gamma(t) \triangleq h^{ G(t)\int_{0}^{t}\widetilde{W}(\tau')d\tau'} \zeta(t) $$
we deduce that  
\begin{eqnarray*}
\Gamma (t) &\leqslant &        \widetilde{\zeta} _0+ (  -\log h )      G(t)  \int_0^t   \widetilde{W}(\tau)  \Gamma (\tau)d\tau ,
\end{eqnarray*}
where we have used the the inequality 
$$ \int_0^t G(\tau) d\tau \leqslant G(t), \quad  \forall \; t\in [0,T] $$
which follows from the increasing property of $t \mapsto G(t)$ and the fact $T\leqslant 1.$
Gronwall lemma yields   
\begin{eqnarray}\label{GAMMA:TTT}
\Gamma (t) &\leqslant &         \widetilde{\zeta} _0 h^{- G(t)\int_{0}^{t}\widetilde{W}(\tau)d\tau} .
\end{eqnarray} 
On the other hand, by the  monotonicity of $t\mapsto \widetilde{W}(t)$ and the  definitions  \eqref{G:def} and \eqref{widetilde:W:def}, which insure that  $ G(t)\leqslant \widetilde{W}(t)$, we end up with      
\begin{eqnarray}\label{Last:GAMMA}
\widetilde{\|} X_t \|_{B^{\sigma_t}_{p,\infty}}+\| \partial_{X_t}\omega(t)\|_{B^{\sigma_t-1}_{p,\infty}} &\leqslant &    \widetilde{\zeta} _0 h^{-  Ct  \widetilde{W}^2(t) } .
\end{eqnarray} 
We are left now with the estimate of $\partial_{X_t}\rho$. First, by inserting \eqref{GAMMA:TTT} in  \eqref{DX-rho:ponct}, we deduce  that 
\begin{eqnarray*}  
   h^{C\int_{0}^{t}\widetilde{W}(\tau)d\tau} \| \partial_{X_t}\rho(t)\|_{B^{\sigma_t}_{p,\infty}}&\leqslant &  (- \log h)\; G(t)  \widetilde{W}(t)    \widetilde{\zeta}_0 h^{- G(t)\int_{0}^{t}\widetilde{W}(\tau')d\tau'}\\
   && + h^{- G(t)\int_{0}^{t}\widetilde{W}(\tau')d\tau'} (- \log h) \;G(t)      \int_0^t   \widetilde{W}(\tau)  \Gamma(\tau) d\tau \\
   &\leqslant & (- \log h)\; G(t)  \widetilde{W}(t)    \widetilde{\zeta}_0 h^{- G(t)\int_{0}^{t}\widetilde{W}(\tau')d\tau'} \\
   && +  (- \log h) \; \widetilde{\zeta}_0 G(t) \int_0^t  \widetilde{W}(\tau)     h^{- G(t)\int_{0}^{\tau}\widetilde{W}(\tau')d\tau'} d\tau.
\end{eqnarray*}
Then, straightforward computation yields
\begin{eqnarray*}
  \| \partial_{X_t}\rho(t)\|_{B^{\sigma_t}_{p,\infty}} & \leqslant & (- \log h)\;   \widetilde{\zeta}_0 \widetilde{W}^2(t)    h^{- G(t)\int_{0}^{t}\widetilde{W}(\tau')d\tau'} \\
 & \leqslant &  \widetilde{\zeta}_0     h^{-  C \widetilde{W}^2(t) } .
\end{eqnarray*}
  Proposition \ref{prop2222} is then proved.
\end{proof}
\subsubsection{Uniform time existence}
 Our purpose now is  to establish  suitable  a priori estimates on some interval $[0,T]$ where $T>0$ is related to the initial data covering singular vortex patches and it is  independent of $\kappa\in[0,1].$
Our result is stated as follows.
\begin{proposition} \label{Lip:Prop:visc}
Let $(s,\varepsilon,p)\in (0,1)\times (0,\frac{s}{2})\times (\frac{1}{\frac{s}{2}-\varepsilon}, \infty)$,  and  $m\in(1,2) $. Let $\Sigma_0$ be a negligible compact set of the plane satisfying \eqref{int:cond0} and consider $\mathcal{X}_0=\big(X_{0,\lambda,h}\big)_{\lambda\in\Lambda,h\in(0,e^{-1}]}$ to be a family   of $\Sigma_0 $-admissible  vector fields of order $\Xi_0=(\alpha,\beta,\gamma)$ and of  class $B^s_{p,\infty}$ as well as their divergence. Let  $(\omega,\rho)$ be a maximal smooth solution of the system \eqref{eqn:omega:2} defined  on a  time interval  $[0,T^\star)$ with the initial data $(\omega_0,\rho_0)$ satisfying   $$\omega_0 \in L^2\cap L^\infty , \quad  \rho_0 \in W^{2,2}\cap W^{2,\infty}, $$
\begin{equation*}
 \sup_{x\notin \Sigma_0}\big( \varphi_0(x)\big)^{-m}\left| \nabla \rho_0(x) \right|<\infty, \quad \quad \sup_{x\notin \Sigma_0}  \big(\varphi_0(x)\big)^{-m+1}\left| \nabla |\nabla\rho_0(x)|^2 \right|<\infty
\end{equation*}
and 
\begin{equation*}
 \sup_{h\in (0,e^{-1}]} h^{\beta  }\Big(\Vert\omega_0\Vert_{(\Sigma_0)_{h},(\mathcal{X}_0)}^{s,p}+ \Vert\rho_0\Vert_{(\Sigma_0)_{ h },(\mathcal{X}_0)}^{s+1,p} \Big)<\infty  ,
 \end{equation*} 
where $\varphi_0$ is given by \eqref{tool:varphi0}.
Then, there exists $T\in (0,T^\star)$ such that 
$$
t\in[0,T]\mapsto \Vert   v(t)\Vert_{L(\Sigma_t)\cap LL\cap L^\infty}\in  L^\infty\big([0,T]\big)
$$
and
 $$(\omega,\rho)\in L^\infty\big([0,T],L^2\cap L^\infty\big)\times L^\infty\big([0,T], W^{1,2}\cap W^{1,\infty}\big).
 $$  
In addition, by setting
$\sigma_t = s \Big( 1- \frac{\varepsilon t}{sT}\Big) - \varepsilon,$
  there exist  some continuous functions  $\Xi(t)=(\alpha(t),\beta(t),\gamma(t))$ and $C_{0,T}>0$, depending only on $T$ and the initial data such that  
 \begin{equation*} 
{\sup_{h\in (0,e^{-1}]} h^{\beta(t)}\Vert\omega(t)\Vert_{(\Sigma_t)_{\delta_t^{-1}(h)},(\mathcal{X}_{t,h})}^{\sigma_t}+\sup_{h\in (0,e^{-1}]} h^{\beta(t)}\Vert\rho(t)\Vert_{(\Sigma_t)_{\delta_t^{-1}(h)},(\mathcal{X}_{t,h})}^{\sigma_t+1}  \leqslant  C_{0,T},}
\end{equation*}  
where, $ \delta_{t}^{-1} (h) $ is defined in \eqref{delta:inverse} and $
 \Sigma_{t} \triangleq \Psi(t, \Sigma_0) .$
Moreover,   the transported vector-fields $\mathcal{X} _t$ of $\mathcal{X} _0$ by the {flow defined through  $\Psi$ }
$$t\mapsto  X_{t,\lambda, h}(\cdot) \triangleq  \partial_{X_{0,\lambda,h}}\Psi(t,\Psi^{-1}(t,\cdot) ), $$
is $\Sigma_t$-admissible of order $\Xi(t)$ and belongs to $L^\infty([0,T]; B^{\sigma_t}_{p,\infty}),$ for all $(\lambda,h) \in \Lambda \times (0,e^{-1}].$
   
\end{proposition}
\begin{proof} We  recall   that the results of Proposition \ref{Proposition visous} and Proposition \ref{prop2222} are satisfied  on the interval $[0,T]$ as long as 
$$\mathbb{V} (T)\leqslant C_{\sigma,m},$$
where $\mathbb{V}(T)$ and $C_{\sigma,m}$ are given by \eqref{V:def:0} and \eqref{C-sigma-m}, respectively. 
   In what follows, we consider a time  $T>0$ such that  
\begin{equation}\label{cond:choice-t}
\mathbb{V}(T) \leqslant C_{\sigma,m} .
\end{equation} 
which is true by a continuity argument, and we shall see at the end that we can get a lower bound for $T$  uniformly in  $\kappa\in[0,1].$  
Therefore, by \eqref{cond:choice-t} all the results of Proposition \ref{Proposition visous} and Proposition \ref{prop2222} hold on $[0,T]$. 
 Now, we intend to track the time evolution  of the stratified regularity with the related norms. We begin with the admissibility of the  family $\mathcal{X}_t $. To this end, we apply Proposition \ref{prop:admissible family} in order  to get
\begin{equation} \label{ixtd2}
I\big((\Sigma_t)_{\delta_t^{-1}(h)},(\mathcal{X}_t)_h\big)
 \geqslant  I\big((\Sigma_0)_{h},(\mathcal{X}_0)_{h}\big)h^{\int_0^t W(\tau)d\tau} \geqslant  I\big((\Sigma_0)_{h},(\mathcal{X}_0)_{h}\big)h^{ -  t \widetilde{W}^2(t)   } ,
\end{equation}
where we have used the monotonicity of $t\mapsto \widetilde{W}(t) $ together with the fact that $W(t) \leqslant\widetilde{W}^2(t) $.
 Next, we introduce the quantity $$
\Upsilon(t)\triangleq  \Vert\omega(t)\Vert_{L^\infty}\widetilde{\Vert} X_{t,\lambda,h}\Vert_{B_{p,\infty}^{\sigma_t}}+\Vert \partial_{X_{t,\lambda,h}} \omega(t)\Vert_{B_{p,\infty}^{\sigma_t-1}}.
$$ 
Owing to Proposition \ref{Proposition visous}, we have for all $t\in[0,T]$
\begin{equation}\label{OMEGA_LP:LLL}
\Vert\omega(t)\Vert_{L^2\cap  L^\infty} \leqslant G(t).
\end{equation}
Together with the  estimate \eqref{ES:striated:reg} and the fact that $1 \leqslant G(t)  $, we obtain 
\begin{equation*}
\Upsilon(t) \leqslant  \left(\widetilde{\Vert} X_{0}\Vert_{B^{ s}_{p,\infty}} +\|\partial_{ X_{0 }}\omega\|_{B^{s-1}_{p,\infty}}  + \| \partial_{X_0}\rho_0\|_{B^{s}_{p,\infty}}\right) G(t)     h^{-C  t \widetilde{W}^2(t) }.
\end{equation*}
Hence,   Definition \ref{def11} and \eqref{ixtd2} imply  
\begin{eqnarray*}
\Upsilon(t) &\leqslant & \Big(N_{s,p}\big((\Sigma_0)_{h},(\mathcal{X}_0)_{h}\big) + \Vert\omega_0\Vert^{s,p}_{(\Sigma_0)_{h},(\mathcal{X}_0)_{h}}+\Vert\rho_0\Vert^{s+1,p}_{(\Sigma_0)_{h},(\mathcal{X}_0)_{h}}  \Big)\notag   I\big((\Sigma_0)_{h},(\mathcal{X}_0)_{h}\big)  G(t)   h^{-C t \widetilde{W}^2(t)  }\\
& \leqslant &  
   C _0  I\big((\Sigma_t)_{\delta_t^{-1}(h)},(\mathcal{X}_t)_h\big) G(t)    h^{-\beta -C  t \widetilde{W}^2(t)},
\end{eqnarray*}
where $C_0$ depends only on the initial data and not on $h$. Consequently, in view of Definition \ref{Defintion-2.4}, we infer that 
\begin{equation}\label{SR:propagation}
 \Vert\omega(t)\Vert^{\sigma_t ,p}_{(\Sigma_t)_{\delta_t^{-1}(h)},(\mathcal{X}_t)_h} \leqslant  C _0     G(t)       h^{-\beta-C   t \widetilde{W}^2(t)}  .
\end{equation}
Similar arguments, using \eqref{ES:striated:reg:rho} instead of \eqref{ES:striated:reg}, lead to
\begin{equation}\label{SR:propagation(rho)}
 \Vert\rho(t)\Vert^{\sigma_t +1 ,p}_{(\Sigma_t)_{\delta_t^{-1}(h)},(\mathcal{X}_t)_h} \leqslant  C _0     G(t)      h^{-\beta-C   \widetilde{W}^2(t)}  .
\end{equation}
On the other hand, due to the fact that 
$\sigma_t\geqslant   s-2{\varepsilon}, \forall t\in[0,T], $
together with the assumption   $ p> \frac{1}{\frac{s}{2} - \varepsilon}$, we obtain in view of \eqref{SR:propagation}, $ \bar{s}\triangleq  s -2\varepsilon -\frac{2}{p} \in (0,1) $ and Sobolev embeddings 
\begin{equation}\label{embedding:0}
\Vert\omega(t)\Vert^{\bar{s}}_{(\Sigma_t)_{\delta_t^{-1}(h)},(\mathcal{X}_t)_h} \lesssim \Vert\omega(t)\Vert^{\sigma_t,p}_{(\Sigma_t)_{\delta_t^{-1}(h)},(\mathcal{X}_t)_h} \leqslant  C _0     G(t)    h^{-\beta-C   t \widetilde{W}^2(t)}  .  
\end{equation} 
Thus, Theorem \ref{propoo1} together with \eqref{embedding:0}   yield 
\begin{eqnarray*}
\Vert\nabla v(t)\Vert_{L^\infty((\Sigma_t)^{c}_{\delta_t^{-1}(h)})}&\lesssim &  \Vert\omega(t)\Vert_{ L^2} +    \Vert\omega(t)\Vert_{ L^\infty} \log\left(e+\tfrac{\Vert\omega\Vert^{\bar{s}}_{(\Sigma_t)_{\delta_t^{-1}(h)},(\mathcal{X}_t)_h}}{\Vert\omega(t)\Vert_{L^\infty}}\right)\nonumber \\
& \leqslant &   G(t) \log \Big(e + C _0         h^{-\beta-C   t \widetilde{W}^2(t)} \Big), 
\end{eqnarray*}
where  we have used the monotonicity of the functions $x \mapsto x \log (e + \frac{b}{x})$ and $x \mapsto  \log (e + \frac{x}{c})$ together with \eqref{OMEGA_LP:LLL}.
Therefore, we deduce for any $t\in[0,T]$ 
\begin{eqnarray*}
\Vert\nabla v(t)\Vert_{L^\infty((\Sigma_t)^{c}_{\delta_t^{-1}(h)})} &\leqslant &  G(t)  \bigg( C_0 +\beta+C   t \widetilde{W}^2(t)\bigg)(-\log h).
\end{eqnarray*}
Then, using the inequalities
\begin{equation}\label{a-z:inequa}
\|  v(t)\|_{LL}\lesssim \|\omega(t)\|_{L^2 \cap L^\infty} \lesssim G(t) \leqslant  \widetilde{W}(t),\quad \forall t \in [0,T].
\end{equation} 
together with the definition \eqref{delta:inverse} yield
\begin{equation*}
\frac{ \Vert\nabla v(t)\Vert_{L^\infty((\Sigma_t)^{c}_{\delta_t^{-1}(h)})}}{ -\log(\delta_t^{-1}(h))} \leqslant C _0  G(t)  \bigg(        1      +    t \widetilde{W}^2(t) \bigg) e^{\int_0^t\widetilde{W}(\tau) d\tau}.
\end{equation*}
Hence, we obtain 
\begin{equation}\label{I:estimate}
  \sup_{\tau \in [0,t]}\| v(\tau) \|_{L(\Sigma_\tau)} \leqslant   C _0  G(t)  \bigg(        1      +    t \widetilde{W}^2(t) \bigg) e^{\int_0^t\widetilde{W}(\tau) d\tau}.
\end{equation}
Consequently, the definition  \eqref{widetilde:W:def} together with  \eqref{a-z:inequa} imply 
\begin{equation}\label{SSSS}
\widetilde{W}(t) \leqslant C_0 G(t)   \bigg(        1      +    t \widetilde{W}^2(t) \bigg) e^{\int_0^t\widetilde{W}(\tau) d\tau}.
\end{equation}
On the other hand, from the definition of $G(t)$ and by virtue of \eqref{a-z:inequa}, we have that 
\begin{eqnarray*}
G(t) & \leqslant & C_0 + C_0 t^\frac{1}{6} \widetilde{W}(t) e^{ Ct\widetilde{W}^2(t)}, \quad \forall\, 0\leqslant t< \min(1,T^\star).
\end{eqnarray*}
Inserting this  inequality into \eqref{SSSS} and  using again the facts $ \widetilde{W}(t) \geqslant 1 $ and $t\leqslant 1$, we  obtain
\begin{align}\label{LAST-INEQU.W}
\nonumber\widetilde{W}(t) &\leqslant C_0    e^{ Ct\widetilde{W}^2(t)}\Big(1+ C_0 t^\frac{1}{6} \widetilde{W}(t) \Big)\\
&\leqslant  C_0    e^{ Ct^{\frac13}\widetilde{W}^2(t)}, \quad \forall\, 0\leqslant t< \min(1,T^\star).
\end{align}
Define
\begin{align}\label{T-defff}
T=\left(\tfrac{\ln 2}{ 4C C_0^2}\right)^3,
\end{align}
which is smaller than $1$ since the constants $C$ and $C_0$ can be taken large enough. Then we can check by the continuity of  $t\in[0,T^\star)\mapsto \widetilde{W}(t)$ and \eqref{LAST-INEQU.W} that
\begin{align}\label{induc-qq}
\forall t\in\big[0,\min(T, T^\star)\big), \quad \widetilde{W}(t)\leqslant 2C_0.
\end{align}
It remains to check that $T^\star>T$. Without loss of generality, we can assume that $T^\star<\infty$.  We shall argue by contradiction and assume that $T^\star\leqslant T.$ According to \eqref{Visc-BLow}, we should have
\begin{align}\label{assum-Neg}
\lim_{t\to T^\star}\mathbb{V}(t)>C_{\sigma,m},
\end{align}
where  $C_{\sigma,m}$ is defined in \eqref{C-sigma-m}.
Now, we write by definition
\begin{eqnarray*}
  \mathbb{V}(t)=\int_0^t  \|v(\tau ) \|_{L(\Sigma_\tau) \cap LL} d\tau 
 &\leqslant   & 
 t  \sup_{\tau\in [0,t]}\Big(  \|v(\tau ) \|_{L(\Sigma_\tau) } +  \|v(\tau ) \|_{LL} \Big)\\
 &\leqslant     & t  \widetilde{W}(t). 
\end{eqnarray*} 
Thus, using \eqref{induc-qq} and \eqref{T-defff} we find
\begin{align*}
  \forall t\in [0,T^\star),\quad \mathbb{V}(t)\leqslant&  2T C_0\\
\leqslant&  2\left(\tfrac{\ln 2}{ 4C }\right)^3 C_0^{-2}.
\end{align*}
Since $C$ and $C_0$ are taken large enough, then the latter inequality contradicts the assumption \eqref{assum-Neg}. Finally, we get that $T^\star>T$ and the time $T$ given by \eqref{T-defff} is only related to the size of the initial data and it is independent of $\kappa\in[0,1].$
 This ends the proof of Proposition \ref{Lip:Prop:visc}.
\end{proof}

 The next result  will be  useful in the study of the regularity of the transported initial boundary of the vortex patch stated in Theorem \ref{th2}.
\begin{coro}\label{cor:X-Psi:2}
Under the assumptions of Proposition \ref{Lip:Prop:visc}, we have
$$\partial_{X_{0,\lambda,h}} \Psi(t,\cdot) =  X_{t,\lambda,h}(\Psi(t,\cdot)) \in L^\infty_t\big([0,T]; C^{\sigma_t - \frac{2}{p}}  \big), \quad \forall (\lambda, h )\in \Lambda\times(0,e^{-1}] . $$
More precisely, there exists a universal constant $C_{0,h,T}>0$ depending only on the initial data, $h$ and $T$ such that, for all $t\in [0,T]$
\begin{equation}\label{X:es:viscous000}
\| X_{t,\lambda,h}(\Psi(t,\cdot))\|_{C^{\sigma_t - \frac{2}{p}}  }\leqslant C_{0,h,T}.
\end{equation}
If moreover, $ \supp X_{0,\lambda,h} \subset  K_{\lambda,h} \Subset \mathbb{R}^2,$ then one has 
\begin{equation}\label{X:es:viscous001}
\| X_{t,\lambda,h}(\Psi(t,\cdot))\|_{B^{\sigma_t - \frac{2}{p}}_{p,\infty}  }\leqslant C_{K_{\lambda,h}}C_{0,h,T}.
\end{equation}
\end{coro}
\begin{proof}
 The first part of the proof is similar to the one of Corollary \ref{cor:X-Psi:1}. Indeed, by assumptions of Proposition \ref{Lip:Prop:visc}, we have $\sigma_t- \frac{2}{p} \in (0,1). $ Thus, owing to the embedding 
 $$B^{\sigma_t}_{p,\infty} \hookrightarrow  C^{\sigma_t - \frac{2}{p}} $$ and repeating the same arguments in the proof of Corollary \ref{cor:X-Psi:1} leads to \eqref{X:es:viscous000}. The proof of \eqref{X:es:viscous001} follows by a direct application of Lemma \ref{inverse-embedding:lemma}.
\end{proof}
\subsection{Proof of Theorem \ref{th2}}\label{sec:xx} 
 The proof  of  the existence and the uniqueness for the system  \eqref{eqn:omega:2}is quite similar to that of the inviscid case developed in detail in Section \ref{Existence-uniq-inviscid}. Here we shall only sketch the main lines. First we construct approximate solutions by considering  the system 
\begin{equation}\label{B_n(viscous)}
\left\{ \begin{array}{lll}
\partial_{t}v_n+v_n\cdot\nabla v_n+\nabla p_n =\rho_n \begin{pmatrix}0\\
  1 \end{pmatrix}, &\\
\partial_{t}\rho_n+v_n\cdot\nabla\rho_n - \kappa \Delta \rho_n=0, &\vspace{2mm}\\
\textnormal{div}\, v_n=0,\\
v_{0,n} = S_nv_0,\quad \rho_{0,n} = \widetilde{S} _n\rho_0.
\end{array} \right. 
\end{equation}
where $S_n$ is the usual  cut-off in frequency defined in Section 2, and $\widetilde{S}_n$ is given by
$$\widetilde{S}_n \rho_0 \triangleq \chi_r \rho_0 + (1-\chi_r)S_n \rho_0 ,$$
with
$\chi_r$ is a smooth cut-off function with values in $[0,1]$ and satisfying 
$$\chi_r(x) = \left\{\begin{array}{ll}
1 & \text{ if } d(x,\Sigma_0)\leqslant\frac{r}{2},\vspace{2mm}\\
0 & \text{ if } d(x,\Sigma_0)\geqslant r.
\end{array} \right. $$
Now, by assumption, we have $ \chi_r \rho_0 \in C^{1+m} $, where $m\in (1,2).$ Also, it is clear that $v_{0,n}$ and $(1-\chi_r)S_n  \rho_0  $ are smooth and belong to $C^{1+m}$. Hence, the result \cite{Chae-Kim-Nam-2} applies and  implies the existence of a unique smooth  solution of   \eqref{B_n(viscous)}. The rest of the proof can be done in a similar way to the inviscid case, however, we need to check that  the platitude conditions  in Theorem \ref{th2}  imply the  assumptions \eqref{hypothesis2}  uniformly in $n$.   First, we observe that 
\begin{equation}\label{Sn0:C}
\nabla \widetilde{S}_n \rho_0(y)=   \nabla   \rho_0(y)  ,\quad \forall y\in (\Sigma_0)_\frac{r}{2}.
\end{equation}   
and consequently we get from straightforward arguments based on the platitude of $\rho_0$ and the law products
\begin{eqnarray*}
\sup_{\underset{x\neq y , \; |x-y|< {r}/{2}}{y\in \Sigma_0 }} \frac{|\nabla| \nabla   \widetilde{S}_n\rho_0(x)|^2|}{|x-y|^{\alpha-1}}&=&\sup_{\underset{x\neq y , \; |x-y|<{r}/{2}}{y\in \Sigma_0 }} \frac{  \big| \nabla| \nabla  \rho_0(x)|^2 -   \nabla| \nabla  \rho_0(y)|^2\big|}{|x-y|^{\alpha-1}} \\ &\lesssim &\| \nabla | \nabla  \rho_0|^2 \|_{C^{\alpha-1}((\Sigma_0)_r)}\\
&\lesssim & \|  \rho_0  \|_{C^{\alpha+1}((\Sigma_0)_r)}^2. 
\end{eqnarray*} 
This proves  the second condition in \eqref{hypothesis2} uniformly in $n$. As to the first condition, we write
\begin{eqnarray*}
\sup_{\underset{x\neq y , \; |x-y|<{r}/{2}}{y\in \Sigma_0 }} \frac{|\nabla  \widetilde{S}_n  \rho(x)|}{|x-y|^\alpha}&=&\sup_{\underset{x\neq y , \; |x-y|<{r}/{2}}{y\in \Sigma_0 }} \frac{|\nabla \rho_0(x) -\nabla \rho_0(y)|}{|x-y|^\alpha}\\
& = & \int_0^1 \frac{\big|\nabla ^2 \rho_0(x + t(y-x))-\nabla ^2 \rho_0(y)\big|}{ |x-y|^{\alpha -1}}dt\\
& \lesssim & \| \nabla ^2 \rho_0 \|_{C^{\alpha -1}((\Sigma_0)_r)}
 \lesssim  \|   \rho_0 \|_{C^{\alpha +1}((\Sigma_0)_r)}.
\end{eqnarray*}
This ensures the hypothesis \eqref{hypothesis2} and the a priori estimates of  Proposition \ref{Lip:Prop:visc} can be used.
}
\section{Inviscid limit}\label{Inviscid-limit} 
This section is devoted to the inviscid limit problem for the system  \eqref{B-0-kappa}    when the diffusivity parameter goes to zero. This will be implemented in a classical way using energy method combined  with Yudovich arguments. Refined convergence results specific to singular vortex patches will be also addressed.  The main result can be stated as follows.
\begin{theorem}\label{Thm:limit}
Let $\kappa\in [0,1]$ and $(v_\kappa,\rho_\kappa)$ $($resp. $(v,\rho))$ be a local solution of \eqref{B-0-kappa}  $($resp.  of \eqref{B0}$)$ with initial data $(v_{0,\kappa},\rho_{0,\kappa})$ $($resp. $(v_0,\rho_0))$.  Let  $\Psi_\kappa$ $($resp. $\Psi)$ denote the flow associated with  $v_\kappa$ $($resp. $v)$ . Assume that all the solutions are defined on   some common interval of time  \mbox{$[0,T]$} and we suppose that  $ v_{0,\kappa}-v_{0}\in L^2$ and for all $\forall t\in [0,T]$
\begin{equation}\label{CC-assum}
\sup_{\kappa\in[0,1]}\left(\big\|\big(v_{\kappa} (t),v(t)\big)\big\|_{L^\infty}+\big\|\big(\omega_{\kappa} (t),\omega (t)\big)\big\|_{L^2\cap L^\infty}  + \big\|\big(\rho_{\kappa} (t), \rho(t)\big)\big\|_{W^{1,2}\cap W^{1,\infty}}\right) \leqslant C_{0,T},
\end{equation}
for some constant $C_{0,T}$, with   $\omega_\kappa=\text{rot } v_\kappa$ and $\omega=\text{rot } v$. Then, the following results  occur  for \mbox{any $t\in [0,T]$}
\begin{enumerate}
\item[{\bf(1)}] $\Vert v_{\kappa}(t)-v(t)\Vert_{L^{2}} +\Vert \rho_{\kappa}(t)-\rho(t)\Vert_{L^{2}} \leqslant C_{0,T}\Big(\Vert v_{0,\kappa}-v_0\Vert_{L^{2}} +\Vert \rho_{0,\kappa}-\rho_0\Vert_{L^{2}} + \kappa t\Big) ^{\frac{1}{2} \exp( -TC_{0,T} )}$,
~~\\
\item[{\bf(2)}] $\| \Psi_\kappa(t,\cdot)- \Psi(t,\cdot)\|_{L^\infty} \leqslant  C_{0,T} \Big(\Vert v_{0,\kappa}-v_0\Vert_{L^{2}} +\Vert \rho_{0,\kappa}-\rho_0\Vert_{L^{2}} + \kappa t\Big) ^{{\frac{1}{4} \exp( -TC_{0,T} ) }}  ,
$
\end{enumerate}
as long as 
\begin{equation*} 
C_{0,T}\Big(\Vert v_{0,\kappa}-v_0\Vert_{L^{2}} +\Vert \rho_{0,\kappa}-\rho_0\Vert_{L^{2}} + \kappa t\Big) ^{{\frac{1}{2} \exp( -TC_{0,T} )}} \lesssim 1.
\end{equation*}

 Moreover, if $\omega_{0,\kappa}=\omega_{0}= \mathrm{1}_{\Omega}$ and $\rho_{0,\kappa}=\rho_0$ satisfy the assumptions in Theorem  \ref{THEO:2:soft}, then the following holds. For any connected component of $\partial \Omega \backslash \Sigma_0$, there exists $\gamma^0$ a $ C^{1+s}(\mathbb{R})$-parametrization of this connected component such that  
\begin{enumerate}
\item[{\bf(3)}]  For all $t\in [0,T]$, $\gamma_{t,\kappa}(\cdot)=\Psi_{\kappa}(t,\gamma^0(\cdot))$ converges to $\gamma_{t}(\cdot)=\Psi(t,\gamma^0(\cdot))$, as $\kappa$ goes to zero, in $C^{1+\bar{s}}(\mathbb{R})$, for all $\bar{s}<s$.
\end{enumerate}
\end{theorem}
Few remarks are in order.
\begin{remark}
We point out the fact that $\hbox{curl} v_0\in L^1\cap L^\infty$ does not imply in general that $v_0\in L^2$ unless the circulation is zero, that is, $
 \displaystyle{\int \hbox{curl}}\, v_0 dx=0.$ Hence  in \mbox{Theorem \ref{Thm:limit}}  we have to assume initially that $v_{0,\kappa}-v_{0}\in L^2$.  Such condition is automatically satisfied if for instance $\hbox{curl}\, v_{0,\kappa}=1_{\Omega_{0,\kappa}}$ and  $\hbox{curl}\, v_{0}=1_{\Omega_{0}}$ with  $|\Omega_{0,\kappa}|=|\Omega _0|$.
\end{remark}

%

\begin{proof}{\bf(1)} The idea of the proof of the first part is based on the Yudovich arguments \cite{Yudovich} and well-developed in \cite[Theorem 7.37]{Bahouri-Chemin-Danchin}.
 Setting $U_{\kappa}=v_{\kappa}-v,\; \vartheta_{\kappa}=\rho_{\kappa}-\rho$ and $P_{\kappa}=\pi_{\kappa}-\pi$ then we can show from the equations that  $(U_{\kappa},\vartheta_{\kappa},P_{\kappa})$ satisfies 
\begin{equation}\label{princ-syst1}
\left\{ 
\begin{array}{ll}
\partial _{t}U_{\kappa}+v_{\kappa}\cdot\nabla U_{\kappa}+\nabla P_{\kappa}=\vartheta_{\kappa} e_{2}-U_{\kappa}\cdot\nabla v  ,\vspace{2mm}\\ 
\partial _{t}\vartheta_{\kappa}+v_{\kappa}\cdot \nabla\vartheta _{\kappa}-\kappa \Delta \vartheta_{\kappa}=-U_{\kappa}\cdot\nabla \rho +\kappa\Delta\rho  ,\vspace{2mm}\\ 
\text{div}\;U_{\kappa}=0, &  \vspace{2mm}\\ 
U_{|t=0}=U_{0,\kappa},\quad \vartheta_{|t=0}=\vartheta_{0,\kappa}. 
\end{array}
\right.   
\end{equation}
Dotting $U_{\kappa}$-equation (resp. $\vartheta_{\kappa}$-equation) by $U_{\kappa}$ (resp. $\vartheta_{\kappa}$) respectively, after some integration by parts, interpolation and the Biot--Savart law we obtain, for any $q\in [2,\infty)$
\begin{eqnarray} \label{U-estimate}
\frac{1}{2}\frac{d}{dt}\Vert U_{\kappa}(t)\Vert_{L^{2}}^{2}   &\leqslant & \Vert \vartheta_{\kappa}(t)\Vert_{L^{2}}\Vert U_{\kappa}(t)\Vert_{L^{2}}+\Vert \nabla v(t)\Vert_{L^{q}}\Vert U_{\kappa}(t)\Vert_{L^{2q'}}^2  \\
\nonumber &\leqslant &  \frac{1}{2} \big( \Vert \vartheta_{\kappa}(t)\Vert_{L^{2}}^2 + \Vert U_{\kappa}(t)\Vert_{L^{2}}^{2} \big)+C q\Vert   \omega(t)\Vert_{L^2} \cap L^{\infty}\Vert U_{\kappa}(t)\Vert_{L^{2}}^\frac{2}{q'} \Vert U_{\kappa}(t)\Vert_{L^{\infty}}^\frac{2}{q } 
\end{eqnarray}
and
\begin{align*}
\frac{1}{2}\frac{d}{dt}\Vert \vartheta_{\kappa}(t)\Vert_{L^{2}}^{2}+&\kappa\Vert \nabla \vartheta_{\kappa}(t)\Vert_{L^{2}}^{2}  
\leqslant  \kappa\Vert \nabla \rho(t)\Vert_{L^{2}}\Vert \nabla \vartheta_{\kappa}(t)\Vert_{L^{2}}+\Vert U_{\kappa}(t)\Vert_{L^{2}} \Vert \nabla \rho(t)\Vert_{L^{\infty}}\Vert \vartheta_{\kappa}(t) \Vert_{L^{2}} \\
\nonumber &\qquad\leqslant   \frac{\kappa}{2}  \Vert \nabla \rho(t) \Vert _{L^{2}}^2  +\frac{\kappa}{2}  \Vert \nabla \vartheta_{\kappa}(t)\Vert_{L^{2}}^2 + \frac{1}{2}\Vert \nabla \rho(t)\Vert_{L^{\infty}} \big(\Vert U_{\kappa}(t)\Vert_{L^{2}}^2 + \Vert \vartheta_{\kappa}(t)\Vert_{L^{2}}^2 \big),
\end{align*}
where $q^\prime$ is the conjugate exponent of $q$.
Thus, we deduce that the quantity
$$\zeta_{\kappa}(t) \triangleq  \Vert U_{\kappa}(t)\Vert_{L^{2}}^2+\Vert \vartheta_{\kappa}(t)\Vert_{L^{2}}^2$$
obeys the differential inequality 
\begin{equation}\label{zeta-ES}
\frac{1}{2}\frac{d}{dt} \zeta_{\kappa}(t) \leqslant  \kappa  C_{0,T}^2  +\Big(  C_{0,T}   + 1 \Big)\zeta_{\kappa}(t) + C_{0,T}^{1+\frac{2}{q}}\,q(\zeta_\kappa(t))^{1-\frac{1}{q}},
\end{equation} 
where we have used the control assumption \eqref{CC-assum}.   In the sequel, $C_{0,T}$ may vary  from line to another. 
Hence, \eqref{zeta-ES} implies 
  \begin{equation}\label{zeta-ES:2}
 \frac{d}{dt} \zeta_{\kappa}(t) \leqslant  \kappa C_{0,T} ^2  + C_0\zeta_{\kappa}(t) + qC_{0,T} ^{1+\frac{2}{q}} (\zeta_{\kappa}(t))^{1-\frac{1}{q}}  .
\end{equation}
 Assuming $\zeta_{\kappa}(t)\leqslant 1$ on $[0,T]$ and $C_{0,T} \geqslant 1$, then \eqref{zeta-ES:2} yields  
  \begin{equation*} 
 \frac{d}{dt} \zeta_{\kappa} (t) \leqslant  \kappa C_{0,T} ^2 +  qC_{0,T} ^{1+\frac{2}{q}} (\zeta_{\kappa}(t))^{1-\frac{1}{q}}  .
\end{equation*}
Therefore, the quantity $$\widetilde{\zeta}_{\kappa}(t) \triangleq \frac{\zeta_{\kappa} (t)}{C_{0,T} ^2}$$
satisfies 
  \begin{equation*} 
 \frac{d}{dt} \widetilde{\zeta}_{\kappa}(t) \leqslant  \kappa + qC_{0,T}  (\widetilde{\zeta}_{\kappa}(t))^{1-\frac{1}{q}}  .
\end{equation*}
Thus, with the choice $q = 2 - 2 \log \widetilde{\zeta}_{\kappa}(t)$, and up to increase if necessary  $C_{0,T}$, we find 
  \begin{equation*} 
   \widetilde{\zeta}_{\kappa}(t) \leqslant   \widetilde{\zeta}_{\kappa}(0)+ \kappa t +     C_{0,T} \int_0^t   \widetilde{\zeta}_{\kappa}(\tau) \log \Bigg( \frac{e^2}{\widetilde{\zeta}_{\kappa}(\tau) }  \Bigg) d\tau.
\end{equation*}
Applying Osgood Lemma \ref{Osgood-Lemma} yields
$$\Vert v_{\kappa}(t)-v(t)\Vert_{L^{2}} +\Vert \rho_{\kappa}(t)-\rho(t)\Vert_{L^{2}} \lesssim C_{0,T}\Big(\Vert v_{0,\kappa}-v_0\Vert_{L^{2}} +\Vert \rho_{0,\kappa}-\rho_0\Vert_{L^{2}} + \kappa t\Big) ^{\frac{1}{2} \exp( -TC_{0,T} )} $$
and  the first inequality in Theorem \ref{Thm:limit} follows.
\\

{\bf(2)}Let us now prove the convergence of the flow. We write first, for all $t\in [0,T]$
\begin{equation*}
\Psi(t,x)- \Psi_\kappa(t,x) = \int_0^t v(\tau,\Psi(\tau,x))  - v_\kappa(\tau,\Psi_\kappa(\tau,x)) d\tau.
\end{equation*}
This yields 
\begin{eqnarray*}
\| \Psi(t,\cdot)- \Psi_\kappa(t,\cdot)\|_{L^\infty} &\leqslant &\int_0^t \|v(\tau,\Psi(\tau,\cdot))  - v(\tau,\Psi_\kappa(\tau,\cdot))\|_{L^\infty} d\tau  \\
&&+ \int_0^t \|v(\tau,\Psi_\kappa(\tau,\cdot))  - v_\kappa(\tau,\Psi_\kappa(\tau,\cdot))\|_{L^\infty} d\tau.
\end{eqnarray*}
For the first term on the r.h.s, we use the fact that $v$ is log-Lipschitz, which follows from the assumption $\omega\in L^2\cap L^\infty$, together with the increasing property  of the function $x\mapsto x\log (\frac{e}{x})$, $x\in (0,1)$. Then  interpolation argument leads to
\begin{eqnarray*}
\| \Psi(t,\cdot)- \Psi_\kappa(t,\cdot)\|_{L^\infty} &\leqslant &\int_0^t \|v(\tau)\|_{LL}\| \Psi(\tau,\cdot)   -  \Psi_\kappa(\tau,\cdot) \|_{L^\infty} \log\Bigg(\frac{e}{\| \Psi(\tau,\cdot)   -  \Psi_\kappa(\tau,\cdot)\|_{L^\infty}}\Bigg) d\tau  \\
&&+ \int_0^t \|v(\tau, \cdot )  - v_\kappa(\tau, \cdot )\|_{L^2}^\frac{1}{2}\|v(\tau, \cdot )  - v_\kappa(\tau, \cdot )\|_{\dot{B}^1_{\infty,\infty}}^\frac{1}{2} d\tau .
\end{eqnarray*}
Hence, the first point {\bf(1)}, together with the fact that
$$\|v(\tau, \cdot )  - v_\kappa(\tau, \cdot )\|_{\dot{B}^1_{\infty,\infty}}^\frac{1}{2} \leqslant \|\omega(\tau, \cdot )  - \omega_\kappa(\tau, \cdot )\|_{L^\infty}^\frac{1}{2} \leqslant  C_{0,T}^\frac{1}{2} .$$
  imply that
\begin{eqnarray*}
\| \Psi(t,\cdot)- \Psi_\kappa(t,\cdot)\|_{L^\infty} &\leqslant & C_{0,T}\int_0^t \| \Psi(\tau,\cdot)   -  \Psi_\kappa(\tau,\cdot) \|_{L^\infty} \log\Bigg(\frac{e}{\| \Psi(\tau,\cdot)   -  \Psi_\kappa(\tau,\cdot)\|_{L^\infty}}\Bigg) d\tau  \\
&&+ t  C_{0,T}\Big(\Vert v_{0,\kappa}-v_0\Vert_{L^{2}} +\Vert \rho_{0,\kappa}-\rho_0\Vert_{L^{2}} + \kappa T\Big) ^{\frac{1}{4} \exp( -TC_{0,T} )} .
\end{eqnarray*}
Consequently,   by using  Osgood's Lemma \ref{Osgood-Lemma}, we find that  
$$\| \Psi_\kappa(t,\cdot)- \Psi(t,\cdot)\|_{L^\infty} \leqslant C_{0,T} \Big(\Vert v_{0,\kappa}-v_0\Vert_{L^{2}} +\Vert \rho_{0,\kappa}-\rho_0\Vert_{L^{2}} + \kappa t\Big) ^{  \frac{1}{4} \exp( -TC_{0,T} ) e^{-t}} .
$$ 
Thus the second result follows by changing the value of $ \widetilde{C}_{0,T}.  $
\\

{\bf(3)} We turn now to prove the last convergence result. We write first
\begin{eqnarray*}
\|\gamma_{t,\kappa}(\cdot)-\gamma_t(\cdot)\|_{ L^\infty}& = & \|\Psi_{\kappa}(t,\gamma^0(\cdot))-\Psi(t,\gamma^0(\cdot))\|_{L^\infty} \nonumber\\
&\leqslant & \|\Psi_{\kappa}(t, \cdot)-\Psi(t, \cdot )\|_{L^\infty}  
\end{eqnarray*}
Thereafter,
in view of the estimate $\textbf{(2)}$
 we infer that, for all $t\in [0,T]$
\begin{equation}\label{gamma-L-infty}
\|\gamma_{t,\kappa}(\cdot)-\gamma_t(\cdot)\|_{ L^\infty} \lesssim C_{0,T} \Big(\Vert v_{0,\kappa}-v_0\Vert_{L^{2}} +\Vert \rho_{0,\kappa}-\rho_0\Vert_{L^{2}} + \kappa t\Big) ^{\frac{1}{4} \exp( -TC_{0,T} )} 
\end{equation} 
which yields the convergence in $L^\infty.$ On the other hand, owing to Theorem \ref{THEO:1:soft} and Theorem \ref{THEO:2:soft}, we get
$$\|\gamma_{t,\kappa}(\cdot) - \gamma_t(\cdot)  \|_{C^{s'}}\leqslant C_{0,T}, \quad \forall s'<s.$$
 Consequently,  by interpolation we find that  for all $\bar{s}\in (0,s')$
\begin{eqnarray*}
\|\gamma_{t,\kappa}(\cdot)-\gamma_t(\cdot)\|_{C^{\bar{s}}} &\lesssim & \|\gamma_{t,\kappa}(\cdot)-\gamma_t(\cdot)\|_{L^\infty}^{\frac{s'-\bar{s}}{s'}}\|\gamma_{t,\kappa}(\cdot)-\gamma_t(\cdot)\|_{ C^{s'}}^{\frac{\bar{s}}{s'} }\\
& \leqslant &  \|\gamma_{t,\kappa}(\cdot)-\gamma_t(\cdot)\|_{L^\infty}^{\frac{s'-\bar{s}}{s'}} C_{0,T}.
\end{eqnarray*}
Therefore,  \eqref{gamma-L-infty}  yields the convergence in $ C^{\bar{s}}(\mathbb{R})$ for all $\bar{s}< s$.  Theorem \ref{Thm:limit} is then proved. 
\end{proof}
We emphasize that the following variation of Theorem \ref{Thm:limit} holds. The proof is straightforward and uses the same arguments of the proof of Theorem \ref{Thm:limit}.
 \begin{coro}\label{Coro:Cauchy seq.}
 Let $(n,\kappa)\in   \mathbb{N} \times [0,1]$ and $(v_{n,\kappa},\rho_{n,\kappa})$   be a local solution of   \eqref{B-0-kappa} with initial data $(v_{n}^0,\rho_{n}^0) \in (L^2(\mathbb{R}^2))^2$. Assume that the sequence of solutions $(v_{n,\kappa},\rho_{n,\kappa})_{(n,\kappa)\in \mathbb{N}\times [0,1]} $ is defined on   some common interval of time  \mbox{$[0,T]$} and we suppose that    $\omega_{n,\kappa}=\text{rot } v_{n,\kappa}$ satisfies, for all \mbox{$(n,\kappa)\in \mathbb{N}\times [0,1]$}
\begin{equation*}
\sup_{n\in\NN, \kappa\in[0,1]}\left(\big\|\big(v_{\kappa} (t),v(t)\big)\big\|_{L^\infty}+\| \omega_{n,\kappa} (t) \|_{L^2\cap L^\infty}  + \| \rho_{n,\kappa} (t) \|_{W^{1,2}\cap W^{1,\infty}}\right) \leqslant C_{0,T}, \quad \forall t\in [0,T],
\end{equation*}
for some constant $C_{0,T}$.  Then,   the following holds   for any $t\in [0,T]$ and for any $(n,m,\kappa)\in \mathbb{N}^2 \times [0,1]$
  $$\Vert v_{n,\kappa}(t)-v_{m,\kappa}(t)\Vert_{L^{2}} +\Vert \rho_{n,\kappa}(t)-\rho_{m,\kappa}(t)\Vert_{L^{2}} \leqslant  C_{0,T}\left(\Vert v_{n}^0-v^0_m\Vert_{L^{2}} +\Vert \rho_{n}^0-\rho^0_{m}\Vert_{L^{2}}  \right) ^{\frac{1}{2} \exp( -TC_{0,T} )}$$
  as long as 
  \begin{equation*} 
C_{0,T}\Big(\Vert v_{0,\kappa}-v_0\Vert_{L^{2}} +\Vert \rho_{0,\kappa}-\rho_0\Vert_{L^{2}}  \Big) ^{\frac{1}{2} \exp( -TC_{0,T} )} \lesssim 1.
\end{equation*}
 \end{coro}
\begin{remark}
Corollary \ref{Coro:Cauchy seq.} states that the solutions to both systems \eqref{B0} and \eqref{B-0-kappa} are $L^2$-stable provided that they are regular enough and satisfy good estimates. This was used before in  step 3 in the proof of Theorem \ref{Th1:general:version} concerning the $L^2$-stability of scheme \eqref{schema-0}.
\end{remark} 
\appendix
\section{Useful lemmas}\label{Appendix}
We shall collect some classical  results used throughout this paper. We start with  recalling the following version of Osgood lemma which is a particular case of \cite[Lemma 3.4]{Bahouri-Chemin-Danchin}.
\begin{lemma}\label{Osgood-Lemma}  Let $f$ be a measurable function from $[t_0,T]$ to $[0,a]$, $\gamma$ a locally integrable function from $[t_0,T]$ to $\mathbb{R}^+$. Assume that, for some non-negative real number $c$, the function $f$ satisfies 
$$ f(t) \leqslant c + \int_{t_0}^t \gamma(\tau) f(\tau) \log\Bigg(\frac{a}{f(\tau)} \Bigg) d\tau, \quad for \text{ a.e. } t\in [t_0,T].$$
Then, we have, for a.e. $t\in [t_0,T]$
\begin{equation*}
 f(t) \leqslant c^{ \exp \left(- \displaystyle\int_{t_0}^t  \gamma(\tau) d\tau \right)}a^{ 1- \exp \left(- \displaystyle\int_{t_0}^t  \gamma(\tau) d\tau \right) } .
\end{equation*}  
\end{lemma}
We recall now a technical estimate due to M. Vishik in \cite{Vishik}, which is very important in proof of the smoothing effect in Proposition \ref{Max-reg}.  
\begin{lemma}\label{Tec-lem-V}
Let $d\geqslant 2$. There exists a positive constant depends only of the dimension $d$ be such that, for every function $f$ in Shwartz's class and for every diffeomorphism $\Psi$ of $\RR^d$ preserving Lebesgue's measure, we have for all $(p,l)\in[1,\infty]\times\NN$ and for all $j,q\geqslant -1$
$$
\|\Delta_j(\Delta_q f\circ\Psi(\cdot))\|_{L^p}\leqslant C^l2^{-l|j-q|}\|\nabla^{l}\Psi^{\varepsilon(j,q)}\|_{L^\infty}\|\Delta_q f\|_{L^p},
$$
with 
$$
\varepsilon(j,q)=\left\{\begin{array}{ll}
 \Sg(j-q) & \text{if } j\neq q ,\vspace{2mm}\\
0 & \text{if }   j=q .
\end{array}
\right.
$$
\end{lemma}

Next, we recall the following commutator estimate from \cite{Hmidi-0}. 
\begin{lemma}\label{Tec-lem-1}
Let $v$ be a vector field of $\RR^d$  in free-divergence and belongs to $LL$. Let a given $f\in L^p$, with $p\in[1,\infty]$, and we define the commutator $ {\bf R}_q$
$${\bf R}_q \triangleq S_{q-1}v\cdot\nabla f_q-\Delta_q(v\cdot\nabla f), \quad \forall q\geqslant -1.$$
 Then, there exists an absolute constant $C>0$ be such that, for every $q\geqslant -1$
 and $s\in(-1,1)$, we have
$$
\|{\bf R}_q\|_{L^p}\leqslant C\bigg(\frac{1}{1-s}+\frac{1}{(1+s)^2}\bigg)\|v\|_{LL}(q+2)2^{-qs}\|f\|_{B^{s}_{p,\infty}}.
$$

\end{lemma}
Now, we provide a self-contained proof of the De Giorgi-Nash estimates in any dimension $d\geqslant 2$. The proof below is based on suitable modification of the proof from \cite{HR2} in dimension three.
\begin{lemma}
\label{Nashlem}
 Let $d\geqslant  2$,    $(p,\,  q, \, p_{1}, \, q_{1} ) \in [1, \infty]^4$ and $ r\in[2,\infty]$ with 
   $${2 \over p}+ {d \over q}<1, \quad{ 2 \over p_{1} }+ {d \over q _{1}} <2.
  $$
  Consider the equation 
 \begin{equation}\label{td}
 \partial_{t}f + u  \cdot \nabla  f - \kappa \Delta f = \nabla \cdot F + G , \quad  t>0, \quad x \in \mathbb{R}^d, \quad f(0, x)= f_{0}(x).
 \end{equation}
   
   There exists $C>0$ such that  for  every smooth divergence-free vector field $u$,
    for every $F\in L^p_{T} L^q$ and for every $f_{0} \in L^r$,  the solution of \eqref{td} satisfies
     the following estimate, for every $t>0$
   \begin{eqnarray}
   \label{Nash}
  \nonumber  \|f(t)  \|_{L^\infty} &\leqslant & C \Big( 1 + {(\kappa t)^{ -d  \over  2r } }\Big)  { \|f_{0} \|_{L^r}  } +C \Big( 1 +  (\kappa t)^{ \frac{1}{2} - \big( {1 \over p } + {d \over {2q}}\big ) } \Big)  { \kappa^{\frac1p-1}\|F \|_{L^p_{t} L^q}  } \\&&+ C \Big( 1 +  (\kappa t)^{ 1- \big( {1 \over p_{1} } + {d \over 2q_{1}} \big) } \Big) 
      { \kappa^{\frac{1}{p_1}-1} \|G \|_{L^{p_{1}}_{t} L^{q_{1} } }  } .
   \end{eqnarray}
   \end{lemma}
\begin{proof} 
Since the equation is linear, the problem can equivalently be reduced to the study of the following three systems
\begin{equation}\label{three-problems}
\left\{ \begin{array}{l}
\mathcal{P}_\kappa f = \nabla \cdot F,\vspace{2mm}\\
f(0,x)=0.
\end{array} \right. \quad
\left\{ \begin{array}{l}
\mathcal{P}_\kappa f = G,\vspace{2mm}\\
f(0,x)=0.
\end{array} \right. \quad
\left\{ \begin{array}{l}
\mathcal{P}_\kappa f = 0,\vspace{2mm}\\
f(0,x)=f_0(x),
\end{array} \right.  
\end{equation}
where, $\mathcal{P}_\kappa f \triangleq \left(\partial_t + u \cdot \nabla   - \kappa \Delta \right) f.$
In fact, we treat the case $\kappa =1$, then by a rescaling argument we obtain the desired estimate for the general case $\kappa >0$. Let us denote $\mathcal{P} \triangleq \mathcal{P}_1$ and we begin with the first problem in \eqref{three-problems}. 

$\bullet$ \textit{\textbf{Step 1}: $T=1$ and $\|F \|_{L^p_1L^q} \leqslant 1$.} Thanks to the divergence-free condition of $u$, a standard $L^q$-energy estimate gives 
\begin{equation*}
\frac{d}{dt} {1\over q} \| f(t) \|_{L^q}^q + \left(q-1 \right) \int_{\mathbb{R}^d} \left|\nabla f \right|^2 \left| f\right|^{q-2} dx \leqslant (q-1) \int_{\mathbb{R}^d}   \left|F \right| \left|\nabla f \right|\left| f\right|^{q-2}dx.
\end{equation*}
On the other hand, we write 
\begin{equation*}
\left|F \right| \left|\nabla f \right|\left| f\right|^{q-2} = \left|F \right| \left|\nabla f \right|\left| f\right|^{\frac{q-2}{2}} \left| f\right|^{\frac{q-2}{2}}.
\end{equation*}
By assumption $q>d \geqslant 2 $, Young and H\"older inequalities, we infer that
\begin{equation*}
\frac{d}{dt} {1\over q} \| f(t) \|_{L^q}^q \leqslant (q-1) \|F \|_{L^q}^2\|f \|_{L^q}^{q-2}.
\end{equation*}
By using the fact that $p>2$ and $\|F \|_{L^p_1L^q} \leqslant 1$, then by an integration in time, we get
\begin{equation}\label{Lq-energy}
\| f\|_{L^\infty_1L^q} \leqslant C_q \left( \int_0^1 \| F(t)\|_{L^q}^2dt \right)^\frac{1}{2} \leqslant C_q.
\end{equation}
The improvement of the above estimate to reach the $L^\infty$ norm in space requires exploring the iteration argument following De Giorgi-Nash. For this purpose, let us introduce the parameter $M>0$ to be chosen later, and let $(M_k)_{k\geqslant 0}$ be a positive increasing sequence converging to $M$
$$
M_k \triangleq M\left(1-\frac{1}{k+1} \right). 
$$
An elementary computation gives then
$$\frac{d}{dt} \left( \frac{1}{2}  \| (f-M_k)_+(t) \|_{L^2}^2\right) + \| \nabla(f-M_k)_+\|_{L^2}^2 \leqslant \left(  \int_{f\geqslant M_k} |F|^2 dx\right)^\frac{1}{2}  \| \nabla(f-M_k)_+\|_{L^2}, 
$$
with $x_+ \triangleq \max (x,0)$. By denoting 
$$
U_k \triangleq      \| (f-M_k)_+(t) \|_{L^\infty_1L^2}^2  + \| \nabla(f-M_k)_+\|_{L^2_1L^2}^2,
$$
we obtain
\begin{equation}\label{U_k estimate 1}
U_k \leqslant \int_0^1 \int_{f\geqslant M_k} |F|^2  dxdt.
\end{equation}
We claim an estimate of the form 
\begin{equation}\label{claim}
U_k \leqslant C(M,k) U_k^\gamma,
\end{equation}  
for some constant $C(M,k)$ depending on $M$ and $K$, and some $\gamma>1.$ To do so, we introduce the positive quantity 
\begin{equation*}
 m_k(t) \triangleq  \left| \left\{ x, f(t) \geqslant M_k \right\} \right|
\end{equation*}
and we notice that if $f(t,x) \geqslant  M_k$ then 
\begin{equation}\label{f-m_k}
f(t,x) - M_{k-1} \geqslant M_k-M_{k-1} \geqslant 0
\end{equation}
and 
$$\mathbf{1}_{f(t,x) - M_{k }} \leqslant \frac{(k+1)^2}{M} \left( f(t,x) - M_{k-1}\right)_+. $$
This yields 
\begin{equation}\label{m_k(t)}
m_k(t) \leqslant \frac{(k+1)^{2m}}{M^m} \| \left( f(t,x) - M_{k-1}\right)_+ \|_{L^m}^m,
\end{equation}
for some $m\geqslant  1$ to be chosen later. Now, we resume from \eqref{U_k estimate 1} and we use H\"older inequality to obtain
\begin{equation}\label{U_k-estimate2}
 U_k \leqslant \int_0^1 \|F(t) \|_{L^q}^2 m_k(t)^{1-\frac{2}{q}}dt \leqslant \|F \|_{L^p_1L^q}^2  \left(\int_0^1 m_k(t) ^{(1-\frac{2}{q})(\frac{p}{p-2}) }dt \right) ^{1-\frac{2}{p}}.
\end{equation}
Implementing \eqref{m_k(t)} yields to
\begin{equation*}
U_k \leqslant\|F \|_{L^p_1L^q}^2 \left(\frac{(k+1)^{2 }}{M }\right)^{m (1-\frac{2}{p})} \left(\int_0^1 \| \left( f(t,x) - M_{k-1}\right)_+ \|_{L^m}^{m(1-\frac{2}{q})(\frac{p}{p-2}) }dt \right) ^{1-\frac{2}{p}}.
\end{equation*}
Let $d^\star$ be given by
$$d^\star \triangleq \left\{  \begin{array}{ll}
\infty & \text{if }d =2,\vspace{2mm}\\
\frac{2d}{d-2}, & \text{ if } d\geqslant 3.
\end{array} \right.$$
By interpolation, one gets
\begin{equation}\label{interpolation inequa}
\|\left(f(t)-M_{k-1} \right)_+ \|_{L^\alpha_1 L^\beta}^2 \leqslant U_{k-1},
\end{equation}
for all\footnote{Notice that when $d\geqslant 3$, $ \beta$ is allowed to take the value $d^\star$ in \eqref{interpolation inequa}.}  $\alpha\geqslant 2$, $\beta \in [2, d^\star)$ and $\frac{2}{\alpha} + \frac{d}{\beta} \geqslant \frac{d}{2}.$ 
 In order to achieve our claim \eqref{claim}, we need to choose $m$ in $[2,d^\star)$ such that 
 \begin{equation}\label{choice of m}
 m \left(1- \frac{2}{q} \right) >2 ,\quad \quad 2 \frac{1-\frac{2}{p}}{m \left( 1- \frac{2}{q}\right)} + \frac{d}{m} \geqslant \frac{d}{2}.
 \end{equation}
The first condition is satisfied whenever $$\frac{2}{1-\frac{2}{q}} < d^\star.$$
This is equivalent to say that $q > \frac{2d^\star}{d^\star-2} = d$, which is automatically satisfied due to the hypothesis $ \frac{2}{p} + \frac{d}{q} <1.$
For the second condition in \eqref{choice of m}, it is satisfied if 
$$2 \left( 1- \frac{2}{p} \right) + d \left( 1- \frac{2}{q} \right) \geqslant m \left( 1- \frac{2}{q} \right) \frac{d}{2} \geqslant 2 \cdot \frac{d}{2} = d, $$
which is equivalent to the assumption $ \frac{2}{p} + \frac{d}{q} <1.$ Hence, as a consequence we obtain \eqref{claim}. More precisely, we find 
\begin{equation*}
U_k \leqslant \left(\frac{(k+1)^2}{M} \right)^{m(1-\frac{2}{p})}U_{k-1}^\gamma, \quad \forall k\geqslant 1.
\end{equation*}
In particular, if $U_0$ is sufficiently small, we obtain $\displaystyle\lim_{k \rightarrow +\infty} U_k = 0$. Thus, Fatou Lemma yields for all $t\in [0,1]$
$$\int_{\mathbb{R}^3} \left(f(t,x) - M \right)_+dx \leqslant 0,$$
and therefore, for almost $(t,x) \in [0,1]\times \mathbb{R}^2$
$$ f(t,x) \leqslant M.$$
By changing $f$ into $-f$, similar arguments lead to 
$$ f(t,x) \geqslant -M,$$
and finally we obtain 
$$\| f\|_{L^\infty_1L^\infty} \leqslant M.$$
Let us recall that the above scenario is true provided that $U_0$ is sufficiently small, and this can happen by choosing $M$ large enough in the following inequality which comes from \eqref{U_k-estimate2}, the Tchebychev inequality and the energy inequality \eqref{Lq-energy}
$$U_0 \leqslant \left( \int_0^1 m_0 (t) ^{(1- \frac{2}{q})(\frac{p}{p-2})} \right) ^{1- \frac{2}{p}} \leqslant \left(\frac{\|f \|_{L^\infty_1L^q}}{M-1} \right)^{q-2} \leqslant \left(\frac{C_q}{M-1} \right) ^{q-2 }.$$
This ends the estimates for the first problem in \eqref{three-problems} in the case $T\leqslant 1$ and $\|F \|_{L^p_1L^q}\leqslant 1.$\\

$\bullet$ \textit{\textbf{Step 2}: $T\geqslant1$ and $\|F \|_{L^p_1L^q} \geqslant 1$.} In this case, we just need to follow closely \cite{HR2} and make use of the following rescaling argument 
$$\widetilde{f}(\tau,X) \triangleq \frac{1}{K} f (T \tau , {T}^{\frac{1}{2}} X), $$
$$\widetilde{F}(\tau,X) \triangleq \frac{{T}^{\frac{1}{2}}}{K} F(T \tau , {T}^{\frac{1}{2}} X), $$
$$K\triangleq {T}^{\frac{1}{2} -(\frac{1}{p} + \frac{d}{2q})} .$$
Under the previous rescaling, $\widetilde{f}$ satisfies the equation 
$$\partial_\tau \widetilde{f}+ \widetilde{u} \cdot \nabla_X \widetilde{f} - \Delta_X = \nabla _X \cdot \widetilde{F}, $$
and 
$$ \|\widetilde{F} \|_{L^p_1L^q}=1. $$
Thanks to the result of step 1, we derive
$$\|f \|_{L^\infty_TL^\infty}= K \| \widetilde{f}\|_{L^\infty_1L^\infty} \leqslant MK = M {T}^{\frac{1}{2} -\big(\frac{1}{p} + \frac{d}{2q}\big)} \|F \|_{L^p_TL^q}. $$
Let us now outline the proof of  the estimate for the second problem in \eqref{three-problems}. We only give some details of the proof in the case $q_1<\infty$, while $q_1=\infty$ is a direct consequence of the Maximum principle. 
For $G$ in $L^{p_1}_TL^{q_1}$, we can write $G= \nabla \cdot F$, for some $F\in L^{p_1}_TW^{1,q_1} \hookrightarrow L^{p_1}L^{q_1^{\star}}$, where $q^\star = \frac{dq_1}{d-q_1}$ and 
$$ \|F \|_{L^{p_1}_TL^{q^*}} \lesssim \|G \|_{L^{p_1}_TL^{q_1}}. $$
Note that $\frac{2}{p_1} + \frac{d}{q_1^{\star}} <1$ is equivalent to $\frac{2}{p_1} + \frac{d}{q_1} <2 $. Hence, the estimate already proved for the first system in \eqref{three-problems} achieves the proof of the desired estimate for the second system.\\

Finally, the estimates of the last system in \eqref{three-problems} can be done along the same lines as in \cite{HR2}. Indeed, the proof presented in \cite{HR2} is done in two steps: the first step consists in proving 
$$\sup_{t\geqslant 1}\|f(t) \|_{L^\infty} \leqslant M.$$
This step does not exploit any particular properties related to the spacial variables, hence, it can be easily extended to any dimension $d\geqslant 2$. Meanwhile, the second step consists in treating the case $t\leqslant 1$. We emphasize that the idea uses the same arguments explored in our treatment of the first system in \eqref{three-problems}. This justifies the apparition of the term $t^{-\frac{d}{2r}}$. The details of this are left to the reader. Lemma \eqref{Nashlem} is then proved.
\end{proof}

\end{document}